\documentclass[10pt,reqno,twoside]{amsart}
\usepackage[a4paper,margin=2.5cm]{geometry}

\usepackage{amsmath,amssymb,amsthm,mathtools,mathrsfs}
\usepackage{enumitem}
\usepackage{microtype}
\usepackage[hidelinks]{hyperref}
\usepackage{fancyhdr}
\usepackage{setspace}

\makeatletter

\newcommand{\Rmnum}[1]{\expandafter\@slowromancap\romannumeral #1@}
\makeatother

\newcommand{\Id}{\operatorname{Id}}
\newcommand{\End}{\operatorname{End}}
\newcommand{\Herm}{\operatorname{Herm}}
\newcommand{\tr}{\operatorname{tr}}
\newcommand{\rk}{\operatorname{rk}}
\newcommand{\osc}{\operatorname{osc}}
\newcommand{\Vol}{\operatorname{Vol}}

\newtheorem{theorem}{Theorem}[section]
\newtheorem{proposition}[theorem]{Proposition}
\newtheorem{lemma}[theorem]{Lemma}
\newtheorem{corollary}[theorem]{Corollary}
\theoremstyle{remark}
\newtheorem{remark}[theorem]{Remark}

\numberwithin{equation}{section}
\begin{document}
%\setstretch{1.035}
\pagestyle{fancy}
\fancyhf{}
\fancyhead[LE,RO]{\thepage}
\renewcommand{\headrulewidth}{0pt}
\setlength{\headheight}{14pt}

\title[Higgs--Demailly system]{Higgs--Demailly System and Positivity of Higgs Bundles}
\author{Liangdi Zhang}
\address{Mathematical Science Research Center, Chongqing University of Technology}
\address{Chongqing 400054, China}
\email{ldzhang91@163.com}
\date{}

\begin{abstract}
We prove that the Higgs--Demailly system on a compact Riemann surface
admits, for suitable fixed parameters, a smooth admissible solution at its
terminal parameter if and only if the Higgs bundle is H-ample.  Here H-ampleness is understood in the
modified sense of Biswas--Misra--Ray.  In particular, the argument gives an
analytic proof of the equivalence between this H-ampleness and
Griffiths-positive Hitchin--Simpson curvature.  The proof extends the Demailly--Pingali--Murakami approach
using a priori estimates and Leray--Schauder degree theory.  The missing
scalar lower bound follows from a Higgs-compatible quotient construction:
a blow-up sequence produces a nonzero Higgs quotient of nonpositive
degree, contradicting H-ampleness.
\vspace*{5pt}

\noindent{\it Keywords}: Higgs bundles, Griffiths positivity, H-ampleness,
Hitchin--Simpson curvature, Demailly system.

\end{abstract}

\maketitle

%\tableofcontents

%%%%%%%%%%%%%%%%%%%%%%%%%%%%%%%%%%%%%%%%%%%%%%%%%%%%%%%%%%%%%%%%%%%%%%%%%%%%%%
%%%%%%%%%%%%%%%%%%%%%%%%%%%%%%%%%%%%%%%%%%%%%%%%%%%%%%%%%%%%%%%%%%%%%%%%%%%%%%
\section{Introduction}
\label{sec1}

A central theme in complex differential geometry is the relation between
algebro-geometric stability or positivity and canonical Hermitian metrics.
For holomorphic vector bundles, the Narasimhan--Seshadri theorem
\cite{NarasimhanSeshadri1965}, Donaldson's gauge-theoretic work
\cite{Donaldson1983,Donaldson1985}, and the work of Uhlenbeck and Yau
\cite{UhlenbeckYau1986,UhlenbeckYau1989} led to the
Donaldson--Uhlenbeck--Yau correspondence: a holomorphic vector bundle over a
compact K\"ahler manifold is polystable if and only if it admits a
Hermitian--Einstein metric.  Standard differential-geometric accounts
of Hermitian vector bundles and the Kobayashi--Hitchin correspondence include
\cite{Kobayashi1987,LubkeTeleman1995}.
Hitchin \cite{Hitchin1987} established the corresponding picture for Higgs
bundles on compact Riemann surfaces, and Simpson
\cite{Simpson1988,Simpson1990,Simpson1992} developed the higher-dimensional
K\"ahler and nonabelian Hodge theory.
Together with Corlette's theorem \cite{Corlette1988} on harmonic metrics for
reductive flat bundles, these results form the analytic foundation of
nonabelian Hodge theory.  The algebraic and moduli-theoretic development of
Higgs bundles and representations was carried further by Nitsure
\cite{Nitsure1991} and Simpson \cite{Simpson1994I,Simpson1994II}.

A recent development is the prescribed Hermitian--Yang--Mills problem.
Here it is important to distinguish endomorphism-valued curvature from its
covariant Hermitian tensor.  Write $R_h=\nabla_h^2$ for the Chern curvature,
and $R_{D_h}=D_h^2$ for the curvature of the Hitchin--Simpson connection
defined in \eqref{eq:Higgs-connection-full}.  We denote the associated
covariant Hermitian--Yang--Mills tensor by
\[
 \mathscr T_{h,\theta}(v,\bar w)
 :=h\bigl((\sqrt{-1}\Lambda_\omega R_{D_h})v,w\bigr),
 \qquad \mathscr T_h:=\mathscr T_{h,0}.
\]
Wang, Yang, and Yau \cite{WangYangYau2026} proved that, if a holomorphic
vector bundle over a compact K\"ahler manifold admits a metric with
$\mathscr T_{h_0}>0$, then for every positive-definite Hermitian tensor
$P\in\Gamma(M,E^*\otimes\overline E^*)$ there is a unique Hermitian metric
$h$ satisfying $\mathscr T_h=P$.
Fan, Wang, Yang, and Yau \cite{FanWangYangYau2026} proved the Higgs analogue
on compact Hermitian manifolds, with $\mathscr T_h$ replaced by
$\mathscr T_{h,\theta}$.  Cao, Sun, and Zhang
\cite{CaoSunZhang2026} developed a Hermitian--Yang--Mills iteration for stable
bundles, together with a Higgs extension, whereas Xiong, Yang, and Yau
\cite{XiongYangYau2026} obtained an algebro-geometric solvability criterion
for the prescribed Hermitian--Yang--Mills equation via a parabolic flow on
holomorphic vector bundles.  These results show that
curvature prescription problems for ordinary bundles can have nontrivial Higgs
counterparts.

The positivity problem considered here is the Higgs counterpart of Griffiths'
conjecture.  A holomorphic vector bundle $E$ over a projective manifold is
ample when the tautological line bundle
$\mathcal O_{\mathbb P(E^*)}(1)$ is ample \cite{Hartshorne1966}.  Griffiths \cite{Griffiths1969} conjectured that $E$ is ample if and only if it
admits a smooth Hermitian metric with Griffiths-positive Chern curvature.  The
implication from curvature positivity to ampleness is classical, while the
converse remains open in higher dimension.  Classical work surrounding this
problem relates curvature positivity, Nakano positivity, Chern-class
positivity, and algebraic ampleness
\cite{BlochGieseker1971,DemaillyPeternellSchneider1994,DemaillySkoda1980,
FultonLazarsfeld1983,Gieseker1971,Hartshorne1971,Nakano1955}.
Useful general references are
\cite{Demailly2012,GriffithsHarris1978,Kobayashi1987,Lazarsfeld2004}.
A complementary analytic literature develops positivity of direct-image
bundles and singular Hermitian metrics
\cite{Berndtsson2009,MourouganeTakayama2007,MourouganeTakayama2008,
Raufi2015}.
On smooth projective curves, Griffiths' conjecture follows from the work of
Umemura \cite{Umemura1973} and Campana--Flenner \cite{CampanaFlenner1990}.
More recently, Xiong, Yang, and Yau
\cite[Corollary~1.6 and Remark~1.8]{XiongYangYau2026} strengthened the
ordinary-bundle curve case to
a prescribed-curvature statement: for an ample vector bundle over a compact
Riemann surface and any K\"ahler metric $\omega$, every positive-definite
Hermitian tensor $P$ can be realized by a Hermitian metric $h$ satisfying
$\mathscr T_h=P$.  In complex dimension one, this prescribes the positive
covariant curvature form as
$\sqrt{-1}R_h^{\mathrm{cov}}=\omega\otimes P$, where
$R_h^{\mathrm{cov}}(v,\bar w)=h(R_hv,w)$.

Numerical positivity for Higgs bundles has been studied through
semistability, Higgs-nefness, numerical flatness, and related positivity
notions in
\cite{BruzzoGrana2007Metrics,BruzzoGrana2007NF,BruzzoGrana2011,
BruzzoHernandez2006}.  Bruzzo, Capasso, and Gra\~na Otero
\cite{BruzzoCapassoOtero2025} studied a recursive notion of H-ampleness
using universal Higgs quotient bundles.  Biswas, Misra, and Ray
\cite{BiswasMisraRay2026} subsequently observed that this recursive notion
does not in general reduce to the usual ampleness of the underlying vector
bundle when the Higgs field is zero, and introduced a modified notion that
does.  Throughout this paper, H-ampleness means the modified notion of
\cite[Definition~2.5]{BiswasMisraRay2026}.  On a smooth projective curve it
is characterized by positivity of the degree of the bundle and of every
nonzero Higgs quotient bundle
\cite[Theorem~3.8]{BiswasMisraRay2026}; this is the criterion used below.

A closely related development is the theory of twisted
Hermitian--Einstein metrics.  Wang, Yang, and Yau
\cite[Theorem~1.1]{WangYangYauTwisted2026} proved a solvability theorem
governed by the minimal quotient slope.  In
\cite[Remark~1.8]{WangYangYauTwisted2026} they state that their main results
also hold for Higgs bundles, indicating an adaptation of the frames in
\cite{FanWangYangYau2026}, without separately developing the Higgs proof.
Combined with the numerical criterion on curves, this stated extension
implies the existence of a Griffiths-positive Hitchin--Simpson metric on
an H-ample Higgs bundle: the minimal Higgs quotient slope is positive,
and one takes the twisting parameter to be zero.  Our result concerns
the solvability of the particular coupled Higgs--Demailly system, not
only the existence of a positive metric.  The proof below does not use
the Higgs extension stated in that remark; instead it treats this coupled
system directly.

Our analytic approach originates in the nonlinear Hermitian--Yang--Mills
systems introduced by Demailly \cite{Demailly2021}.  These systems are designed
so that a solution at the terminal parameter yields a Griffiths-positive, and
in a suitable version dual-Nakano-positive, metric.  Pingali
\cite{Pingali2021} showed that the cushioned Hermitian--Einstein equation in
Demailly's original formulation is essentially unique on stable bundles,
indicating the need for a modified system.  Pingali \cite{Pingali2023} later
solved the modified system for direct sums of ample line bundles over a compact
Riemann surface and reduced the general case to an a priori estimate by
Leray--Schauder degree theory.  Mandal \cite{Mandal2023} solved related systems
for vortex bundles.  Murakami \cite{Murakami2026} obtained the missing
estimate by combining Pingali's reduction with a quotient form of the
Uhlenbeck--Yau argument, thereby giving an analytic proof of Griffiths'
conjecture on compact Riemann surfaces.

Complex dimension one is essential.  On a Riemann surface the
Hitchin--Simpson curvature is an $\operatorname{End}(E)$-valued $(1,1)$-form
and is therefore determined by its contraction with $\Lambda_\omega$.
Consequently,
\[
  \sqrt{-1}\Lambda_\omega R_{D_h}>0
\]
is equivalent to Griffiths positivity of $R_{D_h}$ with the convention
in \eqref{eq:Griffiths-positive-full}.  Moreover, every
torsion-free coherent sheaf on a smooth curve is locally free, and the
codimension-two singular set in the Uhlenbeck--Yau construction is empty
\cite{UhlenbeckYau1986,UhlenbeckYau1989}.  Thus a weak Higgs projection gives
an actual Higgs subbundle and hence a Higgs quotient bundle, to which the
numerical criterion for H-ampleness applies.  These features do not persist in
the same form in higher dimension.

We extend the Demailly--Pingali--Murakami scheme to prove terminal
solvability of the Higgs--Demailly system under H-ampleness.  A prescribed
covariant-tensor theorem does not directly solve its coupled determinant
and trace-free equations.  In adapting the system argument, the Higgs
commutator changes the metric variation, the maximum-eigenvalue estimate,
and the matrix-power inequalities.  The limiting quotient must be invariant
under the Higgs field, not merely holomorphic.  We obtain this compatibility
using the spectral-cutoff method of Simpson and Jacob
\cite{Simpson1988,Jacob2015}, applied on the quotient side of the blow-up
argument.  These are the points at which the Higgs case differs from the
ordinary-bundle argument.

Let $(M,\omega)$ be a connected compact Riemann surface and let $(E,\theta)$ be
a Higgs bundle of rank $r>0$.  For a Hermitian metric $h$ on $E$, the
Hitchin--Simpson connection is
\begin{equation}\label{eq:Higgs-connection-full}
 D_h=\nabla_h+\theta+\theta_h^{*},
\end{equation}
where $\nabla_h$ is the Chern connection and $\theta_h^{*}$ is the
$h$-adjoint of the Higgs field.  Its curvature is
\[
 R_{D_h}=D_h^2\in\Gamma\bigl(M,\Lambda^2T^*M\otimes\mathrm{End} E\bigr).
\]
We say that $R_{D_h}$ is \emph{Griffiths positive} if
\begin{equation}\label{eq:Griffiths-positive-full}
\bigl\langle R_{D_h}(\xi,\bar\xi)v,v\bigr\rangle_h>0
\end{equation}
for every $x\in M$, every nonzero $\xi\in T_x^{1,0}M$, and every nonzero
$v\in E_x$.  Thus, if $\omega=\sqrt{-1}g_{z\bar z}\,dz\wedge d\bar z$
and $R_{D_h}=A\,dz\wedge d\bar z$, the Hermitian endomorphism used to
test positivity is $A$, while
$\sqrt{-1}\Lambda_\omega R_{D_h}=g_{z\bar z}^{-1}A$.
We denote the trace-free part by
\begin{equation}\label{eq:tracefree-full}
 R^0_{D_h}=R_{D_h}-\frac1r\mathrm{tr}_E(R_{D_h})\otimes\Id_E.
\end{equation}
Similarly, $R_h^0$ denotes the trace-free part of the Chern curvature $R_h$.

Fix a reference metric $h_{\rm ref}$ on $E$ and write
\begin{equation}\label{eq:metric-decomp-intro-full}
 h_t=e^{-f_t}g_th_{\rm ref},
 \qquad \det g_t=1,
\end{equation}
where $f_t\in C^\infty(M,\mathbb R)$ and $g_t$ is a positive
$h_{\rm ref}$-Hermitian endomorphism.  For constants $\alpha,\lambda>0$
and a fixed smooth positive function $a_0$, determined by the initial
solution, we consider the following Higgs extension of Demailly's system:
\begin{equation}\label{eq:HD-intro-full}
 \left\{
 \begin{aligned}
 \det\bigl(\sqrt{-1}\Lambda_\omega R_{D_{h_t}}+(1-t)\alpha\Id_E\bigr)
 &=e^{\lambda f_t}a_0,\\
 \sqrt{-1}\Lambda_\omega R^0_{D_{h_t}}&=-e^{f_t}\log g_t.
 \end{aligned}\right.
\end{equation}
A solution of \eqref{eq:HD-intro-full} is called \emph{admissible} if
\[
 \sqrt{-1}\Lambda_\omega R_{D_{h_t}}
 +(1-t)\alpha\Id_E>0.
\]
At $t=1$, admissibility is exactly the positivity of
$\sqrt{-1}\Lambda_\omega R_{D_{h_t}}$, while the second equation controls the
trace-free metric variable.

The term $[\theta,\theta_h^*]$ is zeroth order, but it cannot be discarded as
a harmless perturbation.  On a curve, the modified H-ampleness criterion is
formulated in terms of degrees of Higgs quotient bundles
\cite[Theorem~3.8]{BiswasMisraRay2026}.  Hence the limiting quotient in the
blow-up argument must be Higgs invariant, not merely holomorphic.  This is
precisely the form of the algebraic input needed below.

The main result is the following.

\begin{theorem}\label{thm:main-intro-full}
Let $(M,\omega)$ be a connected compact Riemann surface and let $(E,\theta)$ be
a Higgs bundle of positive rank.  Fix
$h_{\rm ref}$ and construct $g_0$ as in
Proposition~\ref{prop:g0-detailed}.  Choose
$\alpha\geq\alpha_0$ and $\lambda\geq\lambda_0$ so that
Lemmas~\ref{lem:t0-global-uniqueness} and~\ref{lem:t0-nondegenerate} apply,
and define $a_0$ by \eqref{eq:initial-A-detailed}.  Then the following
statements are equivalent:
\begin{enumerate}[label=(\arabic*)]
 \item the Higgs--Demailly system \eqref{eq:HD-intro-full} has a smooth
 admissible solution at $t=1$;
 \item $(E,\theta)$ admits a Hermitian metric with Griffiths-positive
 Hitchin--Simpson curvature;
 \item $(E,\theta)$ is H-ample in the modified sense of
 \cite[Definition~2.5]{BiswasMisraRay2026}.
\end{enumerate}
\end{theorem}

\begin{remark}\label{rem:theta-zero-intro}
When $\theta=0$, the modified H-ampleness used here coincides with the
usual ampleness of the underlying holomorphic vector bundle
\cite[Remark~3]{BiswasMisraRay2026}, and the Hitchin--Simpson curvature
reduces to the Chern curvature.  In this case the theorem recovers the
ordinary vector-bundle problem: Pingali \cite{Pingali2023} reduced the
Demailly-system approach to a uniform lower bound, and Murakami
\cite{Murakami2026} established that bound by a quotient form of the
Uhlenbeck--Yau construction, obtaining an analytic proof of Griffiths'
conjecture on compact Riemann surfaces.  The curve case was already known by
the earlier work of Umemura \cite{Umemura1973} and Campana--Flenner
\cite{CampanaFlenner1990}.  Thus the new analytic issues in the
present paper occur when the Higgs field is nontrivial.
\end{remark}

For a nonzero Higgs field, the weak limiting projection must satisfy both
\[
 (\Id_E-\pi)\bar\partial_E\pi=0
 \qquad\text{and}\qquad
 (\Id_E-\pi)\theta\pi=0.
\]
The second identity is indispensable, since otherwise the limiting quotient
need not be a Higgs quotient and H-ampleness gives no contradiction.  There
are two further Higgs-specific points.  The variation of the metric adjoint
$\theta_h^*$ produces a double commutator whose logarithmic pairing gives the
nonnegative energy term $|[\theta_h^*,K]|^2$; this is used in the uniqueness
and nondegeneracy arguments at $t=0$.  In the blow-up analysis, the matrix
power inequality acquires an additional Higgs term, and the quotient degree
formula contains the nonnegative off-diagonal Higgs energy.  These sign
properties are what allow the ordinary Uhlenbeck--Yau mechanism to survive in
the Higgs setting.

Our proof separates the scalar and trace-free metric variables.  The initial
metric is obtained from a strictly perturbed Higgs Hermitian--Einstein
equation, and sufficiently large $\alpha$ and $\lambda$ give uniqueness and
nondegeneracy at $t=0$.  Along the continuity path we establish
\[
 e^f|\log g|+|\Delta_\omega^{\mathbb C} f|+\operatorname{osc}_M f\le C.
\]
A uniform lower bound for $f$ then gives two-sided metric bounds and higher
regularity.  If this lower bound fails, a maximum-eigenvalue normalization of
$g$, followed by a quotient-side Simpson--Uhlenbeck--Yau construction,
produces a nonzero Higgs quotient of nonpositive degree, contradicting
H-ampleness.  The limiting subsheaf may be zero, in which case the quotient is
$E$ itself; no additional determinant normalization of the blow-up sequence is needed.

The paper is organized as follows.  Section~\ref{sec2} records the curvature
identities.  Section~\ref{sec3} constructs and analyzes the initial solution.
Section~\ref{sec4} proves the a priori estimates.  Section~\ref{sec5} carries
out the Leray--Schauder reduction.  Section~\ref{sec6} establishes the missing
lower bound by the Higgs quotient construction, and Section~\ref{sec7}
completes the proof of Theorem~\ref{thm:main-intro-full}.  The appendices
contain the local matrix calculations and compactness arguments used in the
proof.

\section{Preliminaries}
\label{sec2}

Let $(E,h,\theta)$ be a Hermitian Higgs bundle over a complex manifold. Thus
$E$ is a holomorphic vector bundle, $h$ is a Hermitian metric, and
\[
 \theta\in\Omega^{1,0}(M,\mathrm{End} E)
\]
satisfies
\begin{equation}\label{eq:Higgs-condition-full}
 \bar\partial_E\theta=0,
 \qquad \theta\wedge\theta=0.
\end{equation}
The Chern connection is
\(
 \nabla_h=\partial^h+\bar\partial_E
\), and the Hitchin--Simpson connection is
\begin{equation}\label{eq:HS-connection-full}
 D_h=\partial^h+\bar\partial_E+\theta+\theta_h^{*}.
\end{equation}
In local holomorphic coordinates $z$ and a local holomorphic frame
$\{e_\alpha\}$, write
\begin{equation}\label{eq:theta-local-detailed}
 \theta=\Theta\,dz,
 \qquad
 \Theta=(\theta^\beta_{\ \alpha}),
 \qquad
 \theta_h^*=h^{-1}\Theta^{\dagger}h\,d\bar z.
\end{equation}
Here $\Theta^{\dagger}$ denotes the conjugate transpose in the chosen frame.
The $(1,1)$-part of the Higgs commutator is
\begin{equation}\label{eq:higgs-comm-local-detailed}
 [\theta,\theta_h^*]
 =\bigl(\Theta h^{-1}\Theta^{\dagger}h
        -h^{-1}\Theta^{\dagger}h\Theta\bigr)
   dz\wedge d\bar z,
\end{equation}
and therefore
\begin{equation}\label{eq:trace-comm-zero-detailed}
 \mathrm{tr}_E[\theta,\theta_h^*]=0.
\end{equation}

Since every Hermitian metric on a compact
Riemann surface is K\"ahler, no torsion term occurs in the integrations by
parts below.  We use the convention
\[
 \Delta^\mathbb{C}_\omega f=\sqrt{-1}\Lambda_\omega\partial\bar\partial f,
\]
so that $\Delta^\mathbb{C}_\omega f\leq0$ at a maximum point of $f$.

These formulas will be used repeatedly when the endomorphism relating two
metrics is diagonalized at a point.

For later use, we recall the sets defined in \cite{FanWangYangYau2026,WangYangYau2026}. Let
\[
 \mathrm{Herm}(E,h)=\{S\in\Gamma(M,\mathrm{End} E):S^{*_{h}}=S\}
\]
denote the space of $h$-Hermitian tensors on $E$, let
\[
 \mathrm{Herm}_0(E,h)=\{S\in\mathrm{Herm}(E,h):\mathrm{tr}_E S=0\}
\]
be its trace-free part, and let
\[
 \mathrm{Herm}^+(E,h)=\{S\in\mathrm{Herm}(E,h):S>0\}
\]
be the space of smooth positive $h$-Hermitian endomorphisms of $E$.  Each
$S\in\mathrm{Herm}^+(E,h)$ determines the Hermitian metric $Sh$ by
$(Sh)(v,w)=h(Sv,w)$.

\begin{lemma}\label{lem:change-metric-full}
Let $h_1$ and $h_2$ be Hermitian metrics on $E$.  Suppose that
$h_2=Kh_1$ for some $K\in\mathrm{Herm}^+(E,h_1)$.  Then
\begin{equation}\label{eq:Chern-change-full}
 \partial^{h_2}-\partial^{h_1}=K^{-1}\partial^{h_1}K
\end{equation}
and
\begin{equation}\label{eq:Higgs-adjoint-change-full}
 \theta_{h_2}^{*}=K^{-1}\theta_{h_1}^{*}K.
\end{equation}
\end{lemma}

\begin{proof}
For smooth sections $w,v$ of $E$, metric compatibility gives
\[
 \partial h_2(w,v)=h_2(\partial^{h_2}w,v)+h_2(w,\bar\partial_E v).
\]
On the other hand,
\begin{align*}
 \partial h_2(w,v)
 &=\partial h_1(Kw,v)\\
 &=h_1((\partial^{h_1}K)w,v)+h_1(K\partial^{h_1}w,v)
   +h_1(Kw,\bar\partial_E v)\\
 &=h_2(K^{-1}(\partial^{h_1}K)w,v)+h_2(\partial^{h_1}w,v)
   +h_2(w,\bar\partial_E v).
\end{align*}
Comparison proves \eqref{eq:Chern-change-full}.  Moreover,
\begin{align*}
 h_2(\theta w,v)
 &=h_1(K\theta w,v)=h_1(w,\theta_{h_1}^{*}Kv)\\
 &=h_1(Kw,K^{-1}\theta_{h_1}^{*}Kv)
 =h_2(w,K^{-1}\theta_{h_1}^{*}Kv),
\end{align*}
which proves \eqref{eq:Higgs-adjoint-change-full}.
\end{proof}

\begin{lemma}\label{lem:HS-curvature-full}
On a Riemann surface the Hitchin--Simpson curvature is of type $(1,1)$ and is
\begin{equation}\label{eq:HS-curvature-full}
 R_{D_h}
 = R_h
  +[\theta,\theta_h^{*}].
\end{equation}
In particular,
\(
 \mathrm{tr}_E[\theta,\theta_h^{*}]=0
\), so
\(
 \mathrm{tr}_E R_{D_h}=\mathrm{tr}_E R_h
\).
\end{lemma}

\begin{proof}
Expanding \eqref{eq:HS-connection-full} gives
\[
 \begin{aligned}
 R_{D_h}={}&R_h+[\theta,\theta_h^{*}]
 +\partial^h\theta+\bar\partial_E\theta
 +\partial^h\theta_h^{*}+\bar\partial_E\theta_h^{*}\\
 &+\theta\wedge\theta+\theta_h^{*}\wedge\theta_h^{*}.
 \end{aligned}
\]
The Higgs condition gives $\bar\partial_E\theta=0$ and
$\theta\wedge\theta=0$, and metric compatibility gives
$\partial^h\theta_h^{*}=(\bar\partial_E\theta)^{*}_{h}=0$.
The remaining terms other than $R_h+[\theta,\theta_h^{*}]$ have type
$(2,0)$ or $(0,2)$.  They vanish on a complex curve (and, in any dimension,
do not contribute after contraction with $\sqrt{-1}\Lambda_\omega$),
which proves \eqref{eq:HS-curvature-full}.  The trace assertion follows from the trace of a commutator being zero.
\end{proof}

On the curve $M$ we use the degree convention
\begin{equation}\label{eq:degree-convention}
\deg E:=\frac1{2\pi}\int_M
\operatorname{tr}\bigl(
\sqrt{-1}\Lambda_\omega R_h
\bigr)\,\omega.
\end{equation}
It is independent of the Hermitian metric $h$.  Every compact Riemann
surface is projective.  By GAGA, the holomorphic vector bundles and
coherent sheaves used below may be viewed algebraically, so the numerical
criterion for H-ampleness on a smooth projective curve applies to $M$.

\subsection{Higgs quotients and H-ampleness on a curve}

We briefly recall the algebro-geometric positivity criterion used in the
blow-up argument. A coherent subsheaf $\mathcal F\subset E$ is called a
Higgs subsheaf if
\[
  \theta(\mathcal F)\subset \mathcal F\otimes K_M
\]
as coherent subsheaves of $E\otimes K_M$. Equivalently, the quotient
sheaf $\mathcal Q=E/\mathcal F$ carries the induced Higgs sheaf
structure. The subsheaf $\mathcal F$ is saturated if and only if
$\mathcal Q$ is torsion-free. Since $M$ is a nonsingular complex curve,
every torsion-free coherent sheaf on $M$ is locally free. Consequently,
the quotient by a saturated coherent Higgs subsheaf is a Higgs quotient
bundle.

The earlier recursive notion of H-ampleness in
\cite{BruzzoCapassoOtero2025} was shown by Biswas, Misra, and Ray
\cite{BiswasMisraRay2026} not to reduce, in general, to ordinary ampleness
when the Higgs field is zero.  We therefore use their modified definition
\cite[Definition~2.5]{BiswasMisraRay2026}: in addition to H-nefness, one
requires the determinant of $E$ and the tautological determinant line bundles
on all Higgs Grassmann schemes to be ample.  This modified notion agrees with
ordinary ampleness when $\theta=0$ \cite[Remark~3]{BiswasMisraRay2026}.
For curves, the only form of the definition needed in the proof is the
following numerical criterion.

\begin{theorem}[{\cite[Theorem~3.8]{BiswasMisraRay2026}}]
\label{thm:Hample-criterion-detailed}
Let $(E,\theta)$ be a Higgs bundle over a compact Riemann surface $M$.
Then $(E,\theta)$ is H-ample if and only if
\[
\deg E>0
\]
and
\[
\deg \mathcal Q>0
\]
for every nonzero Higgs quotient bundle $(\mathcal Q,\theta_{\mathcal Q})$ of $(E,\theta)$.
\end{theorem}

\begin{remark}
\label{rem:coherent-quotients-detailed}
Theorem~\ref{thm:Hample-criterion-detailed} is stated for Higgs
quotient bundles, whereas the Uhlenbeck--Yau regularity theorem is
formulated in terms of coherent subsheaves. This causes no difficulty
on a Riemann surface. Indeed, a weak Higgs projection determines a
saturated coherent Higgs subsheaf $\mathcal F\subset E$ which is a holomorphic
subbundle away from an analytic subset of complex codimension at least
two. Since $\dim_{\mathbb C}M=1$, this analytic subset is empty.
Consequently, $\mathcal F$ is a holomorphic Higgs subbundle on all of
$M$, and $\mathcal Q=E/\mathcal F$ is an actual Higgs quotient bundle.
Thus Theorem~\ref{thm:Hample-criterion-detailed} applies directly to
the quotient produced by the analytic construction.
\end{remark}

\subsection{The logarithmic pairing identity}

We record a pointwise identity that will be used both in the construction
of the initial metric and in the uniform estimates.  Let \(h_{\rm ref}\)
be a fixed Hermitian metric on \(E\), let
\[
 g=e^u\in\Herm^+(E,h_{\rm ref}),
 \qquad
 u=u^{*_{h_{\rm ref}}},
\]
and set
\[
 B:=g^{-1}\partial^{h_{\rm ref}}g.
\]
All pairings below are the real pointwise pairings induced by
\(h_{\rm ref}\) and \(\omega\).

\begin{lemma}[Logarithmic pairing identity]
\label{lem:logarithmic-pairing}
Define
\begin{align}
 \mathscr Q_\nabla(u)
 &:=
 \operatorname{Re}
 \left\langle
 B,\partial^{h_{\rm ref}}u
 \right\rangle_{h_{\rm ref},\omega},
 \label{eq:Qnabla-definition}\\
 \mathscr Q_\theta(u)
 &:=
 \left\langle
 \sqrt{-1}\Lambda_\omega
 [\theta,g^{-1}\theta_{h_{\rm ref}}^*g
          -\theta_{h_{\rm ref}}^*],
 u
 \right\rangle_{h_{\rm ref}}.
 \label{eq:Qtheta-definition}
\end{align}
Then
\begin{align}
&\left\langle
 \sqrt{-1}\Lambda_\omega
 \left(
   \bar\partial B
   +[\theta,g^{-1}\theta_{h_{\rm ref}}^*g
                 -\theta_{h_{\rm ref}}^*]
 \right),u
\right\rangle_{h_{\rm ref}}
\notag\\
&\qquad
=
-\frac12\Delta_\omega^{\mathbb C}|u|_{h_{\rm ref}}^2
+\mathscr Q_\nabla(u)+\mathscr Q_\theta(u).
\label{eq:logarithmic-pairing}
\end{align}
Moreover,
\[
 \mathscr Q_\nabla(u)\geq0,
 \qquad
 \mathscr Q_\theta(u)\geq0.
\]
If \(\|u\|_{C^0}\leq L\), then there exists a constant \(c_L>0\),
depending only on \(L\), such that
\begin{equation}\label{eq:Qnabla-coercive}
 \mathscr Q_\nabla(u)
 \geq
 c_L|\partial^{h_{\rm ref}}u|_{h_{\rm ref},\omega}^2.
\end{equation}
\end{lemma}

\begin{proof}
The differential formula for the exponential gives
\begin{equation}\label{eq:B-integral-exponential}
\begin{aligned}
 B
 &=
 e^{-u}\partial^{h_{\rm ref}}e^u\\
 &=
 \int_0^1
 e^{-su}(\partial^{h_{\rm ref}}u)e^{su}\,ds.
\end{aligned}
\end{equation}
Since \(u\) commutes with \(e^{su}\), cyclicity of the trace yields
\begin{align}
 \operatorname{tr}_E(Bu)
 &=
 \int_0^1
 \operatorname{tr}_E
 \left(
 e^{-su}(\partial^{h_{\rm ref}}u)e^{su}u
 \right)\,ds
 \notag\\
 &=
 \operatorname{tr}_E
 \left(
 (\partial^{h_{\rm ref}}u)u
 \right)
 =
 \frac12\partial\operatorname{tr}_E(u^2).
 \label{eq:Bu-trace}
\end{align}
Since \(u\) is \(h_{\rm ref}\)-Hermitian,
\[
 |u|_{h_{\rm ref}}^2=\operatorname{tr}_E(u^2).
\]
Using \eqref{eq:Bu-trace}, the Leibniz rule, and the convention for
\(\Delta_\omega^{\mathbb C}\), we obtain
\begin{align}
 \left\langle
 \sqrt{-1}\Lambda_\omega\bar\partial B,u
 \right\rangle_{h_{\rm ref}}
 &=
 -\frac12\Delta_\omega^{\mathbb C}|u|_{h_{\rm ref}}^2
 +
 \operatorname{Re}
 \left\langle
 B,\partial^{h_{\rm ref}}u
 \right\rangle_{h_{\rm ref},\omega}
 \notag\\
 &=
 -\frac12\Delta_\omega^{\mathbb C}|u|_{h_{\rm ref}}^2
 +\mathscr Q_\nabla(u).
 \label{eq:differential-log-pairing}
\end{align}
Adding the Higgs commutator pairing
\eqref{eq:Qtheta-definition} proves
\eqref{eq:logarithmic-pairing}.

We next verify the positivity of the two quadratic terms.  Fix a point
\(x\in M\).  Choose a local holomorphic coordinate which is unitary for
\(\omega\) at \(x\), and an \(h_{\rm ref}\)-unitary frame which is normal
at \(x\) and diagonalizes \(u(x)\).  Write
\[
 u(x)=\operatorname{diag}(\ell_1,\ldots,\ell_r),
 \qquad
 \theta(x)=\Theta\,dz.
\]
Define
\[
 \chi(s):=
 \begin{cases}
 \dfrac{e^s-1}{s},&s\neq0,\\[5pt]
 1,&s=0.
 \end{cases}
\]
The divided-difference formula for the exponential gives
\begin{equation}\label{eq2.17}
 B_{ab}
 =
 \chi(\ell_b-\ell_a)
 (\partial^{h_{\rm ref}}u)_{ab}
\end{equation}
at \(x\).  Consequently,
\begin{equation}\label{eq:Qnabla-exact}
 \mathscr Q_\nabla(u)
 =
 \sum_{a,b}
 \chi(\ell_b-\ell_a)
 \left|(\partial^{h_{\rm ref}}u)_{ab}\right|^2
 \geq0.
\end{equation}
If \(\|u\|_{C^0}\leq L\), then
\[
 |\ell_b-\ell_a|\leq2L.
\]
Since \(\chi\) is continuous and strictly positive on \(\mathbb R\),
\[
 c_L:=
 \min_{|s|\leq2L}\chi(s)>0,
\]
and \eqref{eq:Qnabla-coercive} follows.

For the Higgs term, one has
\[
 \left(
 g^{-1}\theta_{h_{\rm ref}}^*g
 -\theta_{h_{\rm ref}}^*
 \right)_{ba}
 =
 \bigl(e^{\ell_a-\ell_b}-1\bigr)
 \overline{\Theta_{ab}}\,d\bar z.
\]
A direct computation therefore gives
\begin{equation}\label{eq:Qtheta-exact}
 \mathscr Q_\theta(u)
 =
 \sum_{a,b}
 (\ell_a-\ell_b)
 \bigl(e^{\ell_a-\ell_b}-1\bigr)
 |\Theta_{ab}|^2.
\end{equation}
Since
\[
 s(e^s-1)\geq0
 \qquad
 \text{for every }s\in\mathbb R,
\]
we conclude that
\[
 \mathscr Q_\theta(u)\geq0.
\]
\end{proof}

\begin{corollary}[Weighted logarithmic estimate]
\label{cor:weighted-logarithmic-estimate}
Let \(f\in C^\infty(M,\mathbb R)\), let
\(g=e^u\in\Herm^+(E,h_{\rm ref})\), and suppose that
\begin{equation}\label{eq:abstract-tracefree-equation}
 \sqrt{-1}\Lambda_\omega
 \left(
  \bar\partial(g^{-1}\partial^{h_{\rm ref}}g)
  +[\theta,g^{-1}\theta_{h_{\rm ref}}^*g
             -\theta_{h_{\rm ref}}^*]
 \right)
 +e^f u+\Psi=0.
\end{equation}
Then
\begin{equation}\label{eq:logarithmic-identity-general}
 e^f|u|_{h_{\rm ref}}^2
 -\frac12\Delta_\omega^{\mathbb C}|u|_{h_{\rm ref}}^2
 +\mathscr Q_\nabla(u)+\mathscr Q_\theta(u)
 =
 -\langle\Psi,u\rangle_{h_{\rm ref}}.
\end{equation}
In particular,
\begin{equation}\label{eq:logarithmic-inequality-general}
 e^f|u|_{h_{\rm ref}}^2
 -\frac12\Delta_\omega^{\mathbb C}|u|_{h_{\rm ref}}^2
 \leq
 |\Psi|_{h_{\rm ref}}|u|_{h_{\rm ref}},
\end{equation}
and
\begin{equation}\label{eq:weighted-log-bound-general}
 e^{\min_M f}\|u\|_{C^0}
 \leq
 \|\Psi\|_{C^0}.
\end{equation}
\end{corollary}

\begin{proof}
Taking the real trace pairing of
\eqref{eq:abstract-tracefree-equation} with \(u\) and applying
Lemma~\ref{lem:logarithmic-pairing} gives
\eqref{eq:logarithmic-identity-general}.  Since
\(\mathscr Q_\nabla(u)\) and \(\mathscr Q_\theta(u)\) are nonnegative,
\eqref{eq:logarithmic-inequality-general} follows.

Let \(x_0\) be a maximum point of
\(|u|_{h_{\rm ref}}^2\).  Then
\[
 \Delta_\omega^{\mathbb C}|u|_{h_{\rm ref}}^2(x_0)\leq0,
\]
and hence
\[
 e^{f(x_0)}\|u\|_{C^0}^2
 \leq
 \|\Psi\|_{C^0}\|u\|_{C^0}.
\]
Since \(f(x_0)\geq\min_M f\), this proves
\eqref{eq:weighted-log-bound-general}.
\end{proof}

\section{The Higgs--Demailly system near the initial parameter}
\label{sec3}
Put
\begin{equation}\label{eq:beta-detailed}
 \beta:=\sqrt{-1}\Lambda_\omega\tr_E R_{D_{h_{\rm ref}}}
       =\sqrt{-1}\Lambda_\omega\tr_E R_{h_{\rm ref}}.
\end{equation}
The following formula is the basic algebraic reduction of the system.

\begin{lemma}\label{lem:system-rewrite-detailed}
Let $h_t=e^{-f_t}g_t h_{\rm ref}$, where $g_t>0$ is
$h_{\rm ref}$-Hermitian and $\det g_t=1$.  Then
\begin{equation}\label{eq:HS-mean-rewrite}
 \sqrt{-1}\Lambda_\omega R_{D_{h_t}}
 =\left(\frac{\beta}{r}+\Delta_\omega^{\mathbb C} f_t\right)\Id_E
  +\sqrt{-1}\Lambda_\omega R^0_{D_{g_t h_{\rm ref}}}.
\end{equation}
Consequently, the Higgs--Demailly system is equivalent to
\begin{equation}\label{eq:HD-detailed}
 \left\{
 \begin{aligned}
  \det A_t&=e^{\lambda f_t}a_0,\\
  \sqrt{-1}\Lambda_\omega\left(
    R^0_{h_{\rm ref}}+\bar\partial(g_t^{-1}\partial^{h_{\rm ref}}g_t)
    +[\theta,g_t^{-1}\theta_{h_{\rm ref}}^*g_t]
  \right)&=-e^{f_t}\log g_t,
 \end{aligned}
 \right.
\end{equation}
where
\begin{equation}\label{eq:A-detailed}
 A_t:=\left(\frac{\beta}{r}+\Delta_\omega^{\mathbb C} f_t+(1-t)\alpha\right)\Id_E
      -e^{f_t}\log g_t.
\end{equation}
Under this reformulation, admissibility is equivalent to $A_t>0$.
\end{lemma}

\begin{proof}
By Lemma~\ref{lem:change-metric-full},
\[
 \partial^{h_t}=\partial^{h_{\rm ref}}+g_t^{-1}\partial^{h_{\rm ref}}g_t
                 -\partial f_t\,\Id_E,
 \qquad
 \theta_{h_t}^*=g_t^{-1}\theta_{h_{\rm ref}}^*g_t.
\]
Therefore
\begin{eqnarray}\label{eq3.5}
 R_{h_t}&=&R_{h_{\rm ref}}+\bar\partial_{\End E}(g_t^{-1}\partial^{h_{\rm ref}}g_t
                 -\partial f_t\,\Id_E)\notag\\
 &=&R_{h_{\rm ref}}+\bar\partial(g_t^{-1}\partial^{h_{\rm ref}}g_t)
          +\partial\bar\partial f_t\otimes\Id_E.
\end{eqnarray}
Since $\det g_t=1$,
\[
 \tr_E(g_t^{-1}\partial^{h_{\rm ref}}g_t)
 =\partial\log\det g_t=0.
\]
The trace of a commutator vanishes, hence
\begin{equation}\label{eq3.6}
 \sqrt{-1}\Lambda_\omega\tr_E R_{D_{h_t}}
 =\beta+r\Delta_\omega^{\mathbb C} f_t.
\end{equation}
\eqref{eq:HS-mean-rewrite} follows.

 Using \eqref{eq:HS-mean-rewrite} and the second equation in \eqref{eq:HD-intro-full}, we obtain the first equation in \eqref{eq:HD-detailed}.

By \eqref{eq:trace-comm-zero-detailed}, \eqref{eq:HS-curvature-full} and \eqref{eq3.5}, the trace-free part is
\begin{eqnarray}\label{eq3.7}
 \sqrt{-1}\Lambda_\omega R^0_{D_{h_t}}&=&\sqrt{-1}\Lambda_\omega(R_{h_t}
  +[\theta,\theta_{h_t}^{*}])-\frac1r\sqrt{-1}\Lambda_\omega\mathrm{tr}_E(R_{h_t}
  +[\theta,\theta_{h_t}^{*}])\otimes\Id_E\notag\\
 &=&\sqrt{-1}\Lambda_\omega\left(R^0_{h_{\rm ref}}+\bar\partial(g_t^{-1}\partial^{h_{\rm ref}}g_t)+[\theta,g_t^{-1}\theta_{h_{\rm ref}}^*g_t]\right).
\end{eqnarray}
Substitution of \eqref{eq3.7} in the second equation of \eqref{eq:HD-intro-full} yields
the second equation in \eqref{eq:HD-detailed}.
\end{proof}

\subsection{The strictly perturbed trace-free equation}
\label{subsec:strictly-perturbed}

Set
\begin{equation}\label{eq:Psi-detailed}
 \Psi:=\sqrt{-1}\Lambda_\omega
 \bigl(R^0_{h_{\rm ref}}+[\theta,\theta_{h_{\rm ref}}^*]\bigr)
 \in\Gamma(M,\Herm_0(E,h_{\rm ref})).
\end{equation}
We construct the initial metric by a two-stage covariant continuity
argument.  The first stage removes the background forcing while a large
zeroth-order coefficient supplies uniform coercivity.  The second stage
reduces that coefficient to one.

For a positive $h_{\rm ref}$-Hermitian endomorphism $g$, put
$h=gh_{\rm ref}$ and define the covariant transport of $\Psi$ by
\begin{equation}\label{eq:covariant-Psi-detailed}
 \mathscr S_g\Psi:=g^{-1/2}\Psi g^{1/2}.
\end{equation}
The endomorphism $\mathscr S_g\Psi$ is trace-free and
$h$-Hermitian. Indeed, by cyclicity of the trace,
\[
 \tr_E(\mathscr S_g\Psi)
 =\tr_E(g^{-1/2}\Psi g^{1/2})
 =\tr_E\Psi=0.
\]
Since
$h=gh_{\rm ref}$, and $g$, $\Psi\in\Herm(E,h_{\rm ref})$, we have
\[
 (\mathscr S_g\Psi)^{*}_{h}
 =g^{-1}
   (\mathscr S_g\Psi)^{*}_{h_{\rm ref}}g
 =g^{-1}(g^{1/2}\Psi g^{-1/2})g
 =\mathscr S_g\Psi.
\]
Thus $\mathscr S_g\Psi\in\Herm_0(E,h)$.

Moreover, since $u=\log g$ is a function of $g$, it commutes with
$g^{1/2}$ and $g^{-1/2}$. Hence, by cyclicity of the trace,
\begin{equation}\label{eq:transport-pairing-detailed}
\tr_E\bigl((\mathscr S_g\Psi)u\bigr)=\tr_E\bigl(g^{-1/2}\Psi g^{1/2}u\bigr)=\tr_E\bigl(\Psi g^{1/2}u g^{-1/2}\bigr)=\tr_E(\Psi u).
\end{equation}
Thus the transport $\mathscr S_g$ places the fixed endomorphism $\Psi$
in the moving space $\Herm_0(E,h)$ without changing its trace pairing
with $u=\log g$. Consequently, the usual logarithmic $C^0$ estimate is
preserved.

Fix \(\tau_0\in[0,1]\) and let \(g\) be a
solution of
\begin{equation}\label{eq3.10}
\sqrt{-1}\Lambda_\omega R^0_{D_{g h_{\rm ref}}}
+c_*\log g
=(1-\tau_0)\mathscr S_g(\Psi),
\end{equation}
where $c_*$ is a fixed prescribed positive constant.

Put
\(h=gh_{\rm ref}\).  Every determinant-one Hermitian metric sufficiently
close to \(h\) can be written uniquely as
\[
h_K=e^Kh,
\qquad
K\in C^{k+2,\gamma}(\Herm_0(E,h)).
\]
Set \(g_K=ge^K\) so that \(h_K=g_Kh_{\rm ref}\).  For \(\tau\) near \(\tau_0\), define the full equation operator by
\begin{equation}\label{eq3.11}
\mathcal F(\tau,K)=\sqrt{-1}\Lambda_\omega R^0_{D_{h_K}}
+c_*\log g_K-(1-\tau)\mathscr S_{g_K}\Psi.
\end{equation}
Each term is trace-free and \(h_K\)-Hermitian, and hence
\[
\mathcal F(\tau,K)
\in
C^{k,\gamma}\bigl(\Herm_0(E,h_K)\bigr).
\]
Moreover, \eqref{eq3.10} gives
\begin{equation}\label{eq3.12}
\mathcal F(\tau_0,0)=0.
\end{equation}

To regard the equation as a map between fixed Banach spaces, set
\[
\widetilde{\mathcal F}(\tau,K)
:=
e^{K/2}\mathcal F(\tau,K)e^{-K/2}.
\]
Since conjugation by \(e^{K/2}\) identifies \(h_K\)-Hermitian
endomorphisms with \(h\)-Hermitian endomorphisms, we obtain a smooth map
\[
\widetilde{\mathcal F}:
[0,1]\times
C^{k+2,\gamma}\bigl(\Herm_0(E,h)\bigr)
\longrightarrow
C^{k,\gamma}\bigl(\Herm_0(E,h)\bigr).
\]
Its zeros are precisely the nearby solutions of the original equation.

Let
\(
L\in C^{k+2,\gamma}\bigl(\Herm_0(E,h)\bigr).
\)
Differentiating with respect to \(K\) at \((\tau_0,0)\) and using \eqref{eq3.12}, we obtain
\begin{eqnarray*}
D_K\widetilde{\mathcal F}(\tau_0,0)[L]&=&\frac12 L\,\mathcal F(\tau_0,0)
+D_K\mathcal F(\tau_0,0)[L]-\frac12\mathcal F(\tau_0,0)L\notag\\
&=&D_K\mathcal F(\tau_0,0)[L].
\end{eqnarray*}
Thus no additional term is introduced into the linearized equation by
the above identification of the moving Hermitian spaces.

We shall also use the Fr\'echet derivative of the matrix logarithm.  If
$G\in\Herm^+(E,h_{\rm ref})$ and $A\in\End(E)$, then
\begin{equation}\label{eq:Dlog-detailed}
 \mathcal D(\log)_G(A)
 =
 \int_0^\infty
 (G+\rho\Id_E)^{-1}
 A
 (G+\rho\Id_E)^{-1}\,d\rho .
\end{equation}
In particular, for a variation $G_s$ with $G_0=G$ and
$\dot G_0=B$, one has
\[
 \left.\frac{d}{ds}\right|_{s=0}\log G_s
 =
 \mathcal D(\log)_G(B).
\]

We now carry out the first stage of the construction, in which the
background forcing term is removed while a sufficiently large zeroth-order
coefficient provides uniform coercivity.

\begin{lemma}[Covariant forcing removal]
\label{lem:g0-forcing-removal}
There exists a constant $c_*>1$, depending only on the fixed background
data, such that, for every $\tau\in[0,1]$, the equation
\begin{equation}\label{eq:g0-path-detailed}
 \sqrt{-1}\Lambda_\omega R^0_{D_{g_\tau h_{\rm ref}}}
 +c_*\log g_\tau
 =(1-\tau)\mathscr S_{g_\tau}\Psi,
 \qquad \det g_\tau=1,
\end{equation}
admits a unique smooth positive $h_{\rm ref}$-Hermitian solution $g_\tau$.
 In particular, at $\tau=1$ there  exists a unique smooth positive
determinant-one endomorphism $g_*$ satisfying
\begin{equation}\label{eq:gstar-detailed}
 \sqrt{-1}\Lambda_\omega R^0_{D_{g_*h_{\rm ref}}}
 +c_*\log g_*=0.
\end{equation}
\end{lemma}

\begin{proof}
At $\tau=0$, $g=\Id_E$ is a solution because
$\sqrt{-1}\Lambda_\omega R^0_{D_{h_{\rm ref}}}=\Psi=\mathscr S_{\Id_E}\Psi$.
We prove openness and closedness.

We first establish a uniform estimate and choose \(c_*\). Let $g$ be a solution of \eqref{eq:g0-path-detailed}, and put $h=gh_{\rm ref}$ and $u=\log g$, and set
$B=g^{-1}\partial^{h_{\rm ref}}g$.  By \eqref{eq3.7},
\eqref{eq:g0-path-detailed} can be written as
\begin{align}\label{eq:g0-path-reference-detailed}
 &\sqrt{-1}\Lambda_\omega\left(
  \bar\partial B
  +[\theta,g^{-1}\theta_{h_{\rm ref}}^*g
             -\theta_{h_{\rm ref}}^*]
 \right)+c_*u\\
 &\qquad=(1-\tau)\mathscr S_g\Psi-\Psi.\notag
\end{align}
Taking the real trace pairing of
\eqref{eq:g0-path-reference-detailed} with \(u=\log g\) and applying
Lemma~\ref{lem:logarithmic-pairing}, we obtain
\begin{equation}\label{eq:g0-UY-detailed}
 c_*|u|_{h_{\rm ref}}^2
 -\frac12\Delta_\omega^{\mathbb C}|u|_{h_{\rm ref}}^2
 +\mathscr Q_\nabla(u)+\mathscr Q_\theta(u)
 =
 -\tau\langle\Psi,u\rangle_{h_{\rm ref}}.
\end{equation}
Here we used
\eqref{eq:transport-pairing-detailed} to compute
\[
 \langle\mathscr S_g\Psi,u\rangle_{h_{\rm ref}}=\mathrm{Re}\tr_E\bigl((\mathscr S_g\Psi)u\bigr)=\mathrm{Re}\tr_E(\Psi u)=\langle\Psi,u\rangle_{h_{\rm ref}}.
\]

Let \(x_0\in M\) be a maximum point of
\(|u|_{h_{\rm ref}}^2\).  Since
\[
 \Delta_\omega^{\mathbb C}|u|_{h_{\rm ref}}^2(x_0)\leq0,
\qquad
 \mathscr Q_\nabla(u)\geq0,
 \qquad
 \mathscr Q_\theta(u)\geq0,
\]
equation \eqref{eq:g0-UY-detailed} gives
\[
 c_*|u(x_0)|_{h_{\rm ref}}^2
 \leq
 \tau|\Psi(x_0)|_{h_{\rm ref}}|u(x_0)|_{h_{\rm ref}}
 \leq
 \|\Psi\|_{C^0}|u(x_0)|_{h_{\rm ref}}.
\]
Consequently,
\begin{equation}\label{eq:g0-C0-detailed}
 \|u\|_{C^0}
 \leq
 \frac{\|\Psi\|_{C^0}}{c_*}.
\end{equation}

Choose initially $c_*\geq2\|\Psi\|_{C^0}+1$.  Then every solution  satisfies
\begin{equation}\label{eq:g0-fixed-cone-detailed}
 e^{-1/2}\Id_E\leq g\leq e^{1/2}\Id_E.
\end{equation}
For the uniqueness argument below, we shall also use the larger fixed
cone
\[
\mathcal K_1
:=
\left\{
G\in\Herm^+(E,h_{\rm ref}):
e^{-3/2}\Id_E\leq G\leq e^{3/2}\Id_E
\right\}.
\]
Lemma~\ref{lem:compact-cone-matrix-estimates}, applied to
\(\mathcal K_1\),  gives constants
\(d_0>0\) and \(C_\Psi>0\), depending only on the fixed background data,
such that
\begin{align}
 \operatorname{Re}
 \left\langle
  \mathcal D(\log)_G(GK),K
 \right\rangle_{Gh_{\rm ref}}
 &\geq d_0|K|_{Gh_{\rm ref}}^2,
 \label{eq:Dlog-lower-g0-detailed}\\
 \left|
  \operatorname{Re}
  \langle\mathcal Z_G(K),K\rangle_{Gh_{\rm ref}}
 \right|
 &\leq C_\Psi|K|_{Gh_{\rm ref}}^2.
 \label{eq:Z-bound-detailed}
\end{align}
for every \(G\in\mathcal K_1\) and every
\(Gh_{\rm ref}\)-Hermitian endomorphism \(K\), where
\[
\mathcal Z_G(K)
:=
\left.\frac{d}{ds}\right|_{s=0}
\mathscr S_{Ge^{sK}}(\Psi).
\]
Increase \(c_*\) once more, if necessary, so that
\begin{equation}\label{eq:cstar-choice}
c_*d_0>C_\Psi+1.
\end{equation}
Increasing \(c_*\) only strengthens
\eqref{eq:g0-C0-detailed}, and hence preserves
\eqref{eq:g0-fixed-cone-detailed}.

Define the solvability set
\[
\mathcal T
:=
\left\{
\tau\in[0,1]:
\eqref{eq:g0-path-detailed}
\text{ admits a smooth positive determinant-one solution}
\right\}.
\]
We have already shown that \(0\in\mathcal T\).

\medskip
\noindent\textit{Claim 1 (Openness).}
The set \(\mathcal T\) is open in \([0,1]\).

\smallskip
Let \(\tau_0\in\mathcal T\), let \(g\) be a solution at \(\tau_0\), and
put \(h=gh_{\rm ref}\).  For
\[
K\in C^{k+2,\gamma}\bigl(\Herm_0(E,h)\bigr),
\]
consider
\[
 g_s=ge^{sK},
 \qquad
h_s=e^{sK}h=g_sh_{\rm ref}.
\]
Then
\[
 \dot g_0=gK,
\]
and the standard metric variation formulas give
\[
 \left.\frac{d}{ds}\right|_{s=0}\partial^{h_s}
 =\partial^hK,
 \qquad
 \left.\frac{d}{ds}\right|_{s=0}\theta_{h_s}^*
 =[\theta_h^*,K],
\]
and
\[
 \left.\frac{d}{ds}\right|_{s=0}\log g_s
 =
 \mathcal D(\log)_g(gK).
\]
Consequently, the linearized operator is
\begin{align}\label{eq:g0-linearized-detailed}
 L_{\tau,g}K
 :=&
\left.\frac{d}{ds}\right|_{s=0}
\left[
\sqrt{-1}\Lambda_\omega R^0_{D_{g_sh_{\rm ref}}}
+c_*\log g_s
-(1-\tau)\mathscr S_{g_s}\Psi
\right]\notag\\
 ={}&
 \sqrt{-1}\Lambda_\omega
 \left(
  \bar\partial\partial^hK
  +[\theta,[\theta_h^*,K]]
 \right)
 +c_*\mathcal D(\log)_g(gK)-(1-\tau)\mathcal Z_g(K).
\end{align}
Explicitly,
\[
\begin{aligned}
 \mathcal Z_g(K)
 ={}&
 \mathcal D(G\mapsto G^{-1/2})_g[gK]\,
 \Psi g^{1/2}\\
 &+
 g^{-1/2}\Psi\,
 \mathcal D(G\mapsto G^{1/2})_g[gK].
\end{aligned}
\]

At a solution, the direct variation
\eqref{eq:g0-linearized-detailed} agrees with the derivative of the
fixed-space map \(\widetilde{\mathcal F}\), because the derivatives of
the conjugating factors are multiplied by the value of the equation
operator, which vanishes.

Since \(g\) lies in \eqref{eq:g0-fixed-cone-detailed},  integration by parts in
\eqref{eq:g0-linearized-detailed}, together with
\eqref{eq:Dlog-lower-g0-detailed} and
\eqref{eq:Z-bound-detailed}, yields
\begin{align}\label{eq:g0-coercive-detailed}
 \operatorname{Re}\int_M
 \langle L_{\tau,g}K,K\rangle_h\,\omega
 \geq
 \int_M
 \left(
  |\partial^hK|_h^2
  +|[\theta_h^*,K]|_h^2
  +|K|_h^2
 \right)\omega.
\end{align}
Thus \(L_{\tau,g}\) is injective.

As an operator
\[
 L_{\tau,g}:
 C^{k+2,\gamma}\bigl(\Herm_0(E,h)\bigr)
 \longrightarrow
 C^{k,\gamma}\bigl(\Herm_0(E,h)\bigr),
\]
it is a second-order elliptic Fredholm operator.  Its principal symbol is the scalar Laplace symbol on
\(\Herm_0(E,h)\).   By deforming
only its lower-order terms, it is therefore homotopic through elliptic operators
to a self-adjoint Laplace-type operator on the same bundle and hence has
Fredholm index zero.  Therefore
\(
\operatorname{ind}L_{\tau_0,g}=0.
\)
Injectivity consequently implies surjectivity, so
\(L_{\tau_0,g}\) is an isomorphism.

Finally,
\[
 D_K\widetilde{\mathcal F}(\tau_0,0)
 =
 L_{\tau_0,g}.
\]
The implicit function theorem on the fixed H\"older spaces produces a unique
nearby solution for every \(\tau\) sufficiently close to \(\tau_0\).
Therefore the solvability set is open.

\medskip
\noindent\textit{Claim 2 (Closedness).}
The set \(\mathcal T\) is closed in \([0,1]\).

\smallskip
Let \(\tau_j\in\mathcal T\), \(\tau_j\to\tau_\infty\), and let \(g_j\)
solve \eqref{eq:g0-path-detailed} at \(\tau=\tau_j\).  Put
\(u_j=\log g_j\) and \(B_j=g_j^{-1}\partial^{h_{\rm ref}}g_j\).
The estimate \eqref{eq:g0-C0-detailed} gives a fixed positive cone
\[
 e^{-L}\Id_E\leq g_j\leq e^L\Id_E,
 \qquad
 L=\frac{\|\Psi\|_{C^0}}{c_*}.
\]
Integrating \eqref{eq:g0-UY-detailed} and using
Lemma~\ref{lem:logarithmic-pairing} yields
\[
 \|\partial^{h_{\rm ref}}u_j\|_{L^2}
 +\|B_j\|_{L^2}\leq C.
\]
Moreover, \eqref{eq:g0-path-reference-detailed} gives
\[
 \sqrt{-1}\Lambda_\omega\bar\partial B_j
 =
 (1-\tau_j)\mathscr S_{g_j}(\Psi)-\Psi-c_*u_j
 -
 \sqrt{-1}\Lambda_\omega
 [\theta,g_j^{-1}\theta_{h_{\rm ref}}^*g_j-\theta_{h_{\rm ref}}^*].
\]
The right-hand side is uniformly bounded in \(L^\infty\).  Since on a
curve
\[
 \bar\partial:\Omega^{1,0}(\End E)\longrightarrow\Omega^{1,1}(\End E)
\]
is elliptic, the standard \(W^{1,p}\)--Morrey bootstrap used in
Lemma~\ref{lem:strictly-perturbed-compactness} first gives a uniform
\(C^{0,\gamma}\)-bound for \(g_j\).  The displayed equation then has a
uniformly \(C^{0,\gamma}\)-bounded right-hand side, and Schauder estimates,
together with
\(
 \partial^{h_{\rm ref}}g_j=g_jB_j,
\)
give uniform \(C^{2,\gamma}\)-bounds.  Repeating the argument gives bounds of
all orders.  Thus, after passing to a subsequence,
\[
 g_j\longrightarrow g_\infty\qquad\text{in }C^\infty.
\]
The fixed cone and \(\det g_j=1\) imply \(g_\infty>0\) and
\(\det g_\infty=1\), and passage to the limit in
\eqref{eq:g0-path-detailed} shows that \(g_\infty\) solves the equation at
\(\tau_\infty\).  Hence \(\tau_\infty\in\mathcal T\), proving closedness.

Since \(\mathcal T\) is nonempty, open, and closed, \(\mathcal T=[0,1]\).

\medskip
\noindent\textit{Claim 3 (Uniqueness).}
For each fixed \(\tau\in[0,1]\), the solution of
\eqref{eq:g0-path-detailed} is unique.

\smallskip
For uniqueness, let \(g_1\) and \(g_2\) be two solutions at the same
parameter \(\tau\).  Put \(h_i=g_i h_{\rm ref}\), let
\[
 H:=g_1^{-1}g_2,
 \qquad
 K:=\log H,
\]
and consider the positive-cone geodesic
\[
 g_s=g_1e^{sK},
 \qquad
 h_s=g_sh_{\rm ref},
 \qquad 0\leq s\leq1.
\]
Then \(K\) is \(h_s\)-Hermitian and trace-free for every \(s\).  Since
both endpoints lie in the fixed cone
\eqref{eq:g0-fixed-cone-detailed}, monotonicity of the weighted geometric
mean shows that the whole path lies in the same cone.

Subtracting the endpoint equations, pairing with \(K\), and integrating
the variation along \(g_s\), we obtain
\begin{align*}
0={}&
\int_0^1\!\!\int_M
\left(
 |\partial^{h_s}K|_{h_s}^2
 +|[\theta_{h_s}^*,K]|_{h_s}^2
\right)\omega\,ds\\
&+
c_*\int_0^1\!\!\int_M
\operatorname{Re}
\left\langle
 \mathcal D(\log)_{g_s}(g_sK),K
\right\rangle_{h_s}
\omega\,ds\\
&-
(1-\tau)\int_0^1\!\!\int_M
\operatorname{Re}
\langle\mathcal Z_{g_s}(K),K\rangle_{h_s}
\,\omega\,ds.
\end{align*}
By \eqref{eq:Dlog-lower-g0-detailed},
\eqref{eq:Z-bound-detailed}, and the choice
\(c_*d_0>C_\Psi+1\), it follows that
\[
0\geq
\int_0^1\!\!\int_M
\left(
 |\partial^{h_s}K|_{h_s}^2
 +|[\theta_{h_s}^*,K]|_{h_s}^2
 +|K|_{h_s}^2
\right)\omega\,ds.
\]
Hence \(K=0\), and therefore \(g_1=g_2\). This proves uniqueness.

Taking \(\tau=1\), we obtain the unique endomorphism \(g_*\)
satisfying \eqref{eq:gstar-detailed}.
\end{proof}

\begin{proposition}[Construction of the initial trace-free metric]
\label{prop:g0-detailed}
There exists a unique smooth positive \(h_{\rm ref}\)-Hermitian
endomorphism \(g_0\) satisfying
\[
 \det g_0=1
\]
and
\begin{equation}\label{eq:g0-detailed}
 \sqrt{-1}\Lambda_\omega
 R^0_{D_{g_0h_{\rm ref}}}
 +\log g_0=0.
\end{equation}
Equivalently,
\begin{equation}\label{eq:g0-expanded-detailed}
 \sqrt{-1}\Lambda_\omega
 \left(
  R^0_{h_{\rm ref}}
  +\bar\partial(g_0^{-1}\partial^{h_{\rm ref}}g_0)
  +[\theta,g_0^{-1}\theta_{h_{\rm ref}}^*g_0]
 \right)
 =
 -\log g_0.
\end{equation}
\end{proposition}

\begin{proof}
Let \(g_*\) be the solution obtained in
Lemma~\ref{lem:g0-forcing-removal}, so that
\[
 \sqrt{-1}\Lambda_\omega R^0_{D_{g_*h_{\rm ref}}}
 +c_*\log g_*=0.
\]
For \(0\leq\sigma\leq1\), set
\[
 c_\sigma:=(1-\sigma)c_*+\sigma
\]
and consider
\begin{equation}\label{eq:g0-coefficient-path-detailed}
 \sqrt{-1}\Lambda_\omega R^0_{D_{g_\sigma h_{\rm ref}}}
 +c_\sigma\log g_\sigma=0,
 \qquad
 \det g_\sigma=1.
\end{equation}
The equation is solvable at \(\sigma=0\), with \(g_\sigma=g_*\).

At a solution, the linearization in a trace-free
\(h_\sigma=g_\sigma h_{\rm ref}\)-Hermitian direction \(K\) is
\begin{equation}\label{eq:g0-coefficient-linearized-detailed}
 \sqrt{-1}\Lambda_\omega
 \left(
  \bar\partial\partial^{h_\sigma}K
  +[\theta,[\theta_{h_\sigma}^*,K]]
 \right)
 +c_\sigma\mathcal D(\log)_{g_\sigma}(g_\sigma K).
\end{equation}
Pairing with \(K\) and integrating gives
\begin{eqnarray*}
&&\operatorname{Re}\int_M
 \left\langle
  L_{\sigma,g_\sigma}K,K
 \right\rangle_{h_\sigma}\omega\\
&=&
 \int_M
 \left(
  |\partial^{h_\sigma}K|_{h_\sigma}^2
  +|[\theta_{h_\sigma}^*,K]|_{h_\sigma}^2
 \right)\omega+
 c_\sigma\int_M
 \operatorname{Re}
 \left\langle
  \mathcal D(\log)_{g_\sigma}(g_\sigma K),K
 \right\rangle_{h_\sigma}\omega.
\end{eqnarray*}

Since \(c_\sigma\geq1\), the last term is strictly positive unless
\(K=0\).  Thus the linearization is injective.  It is elliptic of
Fredholm index zero and hence is an isomorphism.  The solution set is
therefore open.

Writing \(u_\sigma=\log g_\sigma\), equation
\eqref{eq:g0-coefficient-path-detailed}, together with
Lemma~\ref{lem:logarithmic-pairing}, gives
\begin{equation}\label{eq:g0-coefficient-C0-detailed}
 c_\sigma|u_\sigma|_{h_{\rm ref}}^2
 -\frac12\Delta_\omega^{\mathbb C}
  |u_\sigma|_{h_{\rm ref}}^2
 +\mathscr Q_\nabla(u_\sigma)
 +\mathscr Q_\theta(u_\sigma)
 =
 -\langle\Psi,u_\sigma\rangle_{h_{\rm ref}}.
\end{equation}
Since \(c_\sigma\geq1\), the maximum principle yields
\[
 \|u_\sigma\|_{C^0}
 \leq
 \|\Psi\|_{C^0}.
\]
The compactness result
Lemma~\ref{lem:strictly-perturbed-compactness}, applied with
\(
 c_j=c_{\sigma_j}
\)
and
\(
 F_j=-\Psi,
\)
then gives uniform estimates of all orders.  Hence the solution set is
closed.  It follows that \eqref{eq:g0-coefficient-path-detailed} is
solvable for every \(\sigma\in[0,1]\).

At \(\sigma=1\), one has \(c_\sigma=1\); denote the resulting solution
by \(g_0\).  It satisfies \eqref{eq:g0-detailed}.

Finally, if \(g_1\) and \(g_2\) are two solutions of
\eqref{eq:g0-detailed}, set
\[
 H:=g_1^{-1}g_2,\qquad K:=\log H,
\]
and join them by the positive-cone geodesic
\(g_s=g_1e^{sK}\).  Pairing the difference of the equations with \(K\)
and integrating along the path gives
\[
0=
\int_0^1\!\!\int_M
\left(
 |\partial^{h_s}K|_{h_s}^2
 +|[\theta_{h_s}^*,K]|_{h_s}^2
\right)\omega\,ds+
\int_0^1\!\!\int_M
\operatorname{Re}
\left\langle
 \mathcal D(\log)_{g_s}(g_sK),K
\right\rangle_{h_s}
\omega\,ds.
\]
Every term is nonnegative, and the logarithmic term is strictly positive
unless \(K=0\).  Hence \(K=0\) and \(g_1=g_2\).
\end{proof}
Set
\[
 h_0:=g_0h_{\rm ref}.
\]
Choose $\alpha>0$ sufficiently large so that
\begin{equation}\label{eq:initial-A-detailed}
 A_0:=\left(\frac{\beta}{r}+\alpha\right)\Id_E-\log g_0>0,
 \qquad a_0:=\det A_0.
\end{equation}
Then $(f,g)=(0,g_0)$ is an admissible solution at $t=0$.

\subsection{Uniform control, uniqueness, and nondegeneracy at the initial parameter}
Unless otherwise stated, all constants in this subsection depend only
on the fixed background data and are independent of
\(\alpha\), \(\lambda\), and the particular admissible solution.

\begin{lemma}[Uniform control at the initial parameter]
\label{lem:t0-asymptotics}
Fix $\lambda_*>0$.  Uniformly for all $\lambda\geq\lambda_*$ and all
admissible solutions $(f,g)$ at $t=0$, with $u=\log g$, one has
\begin{equation}\label{eq:t0-asymptotic-detailed}
 \lambda\|f\|_{C^0}=O(\alpha^{-1}),
 \qquad \|u\|_{C^0}\leq C,
 \qquad
 \frac{\|\Delta_\omega^{\mathbb C} f\|_{C^0}}{\alpha}=o(1)
 \quad (\alpha\to\infty).
\end{equation}
More precisely,
\[
 \sup_{\lambda\geq\lambda_*}
 \sup_{(f,g)\in\mathcal S_{\alpha,\lambda}}
 \frac{\|\Delta_\omega^{\mathbb C}f\|_{C^0}}{\alpha}
 \longrightarrow0
 \qquad\text{as }\alpha\to\infty,
\]
where \(\mathcal S_{\alpha,\lambda}\) denotes the set of admissible
solutions at \(t=0\).
\end{lemma}

\begin{proof}
At a minimum point of $f$, admissibility and
\eqref{eq:weighted-log-bound-general} imply that every eigenvalue of $A$ is at least
$\alpha-C$.  Since
\[
 a_0=\det\left(\left(\frac\beta r+\alpha\right)\Id_E-\log g_0\right)
 =\alpha^r(1+O(\alpha^{-1}))
\]
uniformly on $M$, the determinant equation yields
\[
 e^{\lambda\min f}\geq1-C\alpha^{-1}.
\]
At a maximum point of $f$, the arithmetic--geometric mean inequality gives
\begin{equation}\label{eq:max-f-alpha-detailed}
 e^{\lambda\max f}\leq
 \frac{(\alpha+C)^r}{a_0}=1+C\alpha^{-1}.
\end{equation}
For large $\alpha$, taking logarithms proves
$\lambda\|f\|_{C^0}\leq C\alpha^{-1}$.
\eqref{eq:weighted-log-bound-general} then gives the uniform bound for $u$.

For the Laplacian estimate, argue by contradiction.  Let
$\alpha_j\to\infty$ and suppose that there are admissible solutions
$(f_j,u_j)$ for which
$\|\Delta_\omega^{\mathbb C} f_j\|_{C^0}/\alpha_j$ does not tend to zero.  Write
\[
 v_j=\frac{\beta/r+\Delta_\omega^{\mathbb C} f_j}{\alpha_j},
 \qquad
 q_{a,j}=\frac{e^{f_j}\ell_{a,j}}{\alpha_j}.
\]
Then $q_{a,j}\to0$ uniformly.  Dividing the determinant equation by
$\alpha_j^r$ gives
\begin{equation}\label{eq:normalized-t0-det-detailed}
 \prod_{a=1}^r(1+v_j-q_{a,j})
 =e^{\lambda f_j}\frac{a_{0,j}}{\alpha_j^r}
 \longrightarrow1
\end{equation}
uniformly.  The factors are positive, and their pairwise differences tend
to zero uniformly.  Lemma~\ref{lem:asymptotic-product-app} therefore gives
$1+v_j-q_{a,j}\to1$ uniformly for every $a$.  Hence $v_j\to0$, a
contradiction.
\end{proof}

\begin{lemma}[Global uniqueness at $t=0$]
\label{lem:t0-global-uniqueness}
There are $\alpha_0,\lambda_0>0$ such that, whenever
$\alpha\geq\alpha_0$ and $\lambda\geq\lambda_0$, the only admissible
solution at $t=0$ is $(0,g_0)$.
\end{lemma}

\begin{proof}
Let \((f,g)\) be an admissible solution at \(t=0\).  We compare it
with the solution \((0,g_0)\), and let \(H\) be the positive
\(h_0\)-Hermitian endomorphism determined
by
\[
 gh_{\rm ref}(\,\cdot\,,\,\cdot\,)
 =
 h_0(H\,\cdot\,,\,\cdot\,).
\]
Thus \(g=g_0H\).  Set
\begin{equation}\label{eq:metric-geodesic-t0-verified}
 K:=\log H,
 \qquad
 \widehat h_s:=e^{sK}h_0=g_sh_{\rm ref},
 \qquad
 g_s:=g_0e^{sK},
 \qquad
 f_s:=sf,
 \quad 0\leq s\leq1,
\end{equation}
and let
\[
 \widetilde h_s:=e^{-f_s}\widehat h_s
\]
be the corresponding full interpolating metric.  Scalar conformal changes do
not alter the trace-free Hitchin--Simpson curvature, so the trace-free part of
the argument may be computed using $\widehat h_s$, whereas the scalar equation
is evaluated on $\widetilde h_s$.
Since \(\det g=\det g_0=1\), we have
\[
 \tr_E K=\log\det H=0.
\]
Moreover,
\[
 \dot g_s=g_sK.
\]
Writing
\[
 u_s:=\log g_s,
 \qquad
 u_0=\log g_0,
 \qquad
 u_1=u=\log g,
\]
we have
\begin{equation}\label{eq:udot-geodesic-t0-verified}
 \dot u_s
 =
 \mathcal D(\log)_{g_s}(g_sK).
\end{equation}

By Lemma~\ref{lem:t0-asymptotics}, \(g\) and \(g^{-1}\) are uniformly
bounded, while \(g_0\) is fixed.  Hence \(H\), \(H^{-1}\), and
\(K=\log H\) are uniformly bounded.  It follows that the whole family
\(\{g_s:0\leq s\leq1\}\) lies in a fixed compact subset of the positive
endomorphisms, independently of the particular admissible solution.

For \(0\leq s\leq1\), define
\[
 A_s
 :=
 \left(
  \frac{\beta}{r}
  +s\Delta_\omega^{\mathbb C}f
  +\alpha
 \right)\Id_E
 -e^{sf}u_s.
\]
Lemma~\ref{lem:t0-asymptotics} gives
\[
 \|f\|_{C^0}=o(1),
 \qquad
 \frac{\|\Delta_\omega^{\mathbb C}f\|_{C^0}}{\alpha}=o(1),
 \qquad
 \sup_{0\leq s\leq1}\|u_s\|_{C^0}\leq C.
\]
Consequently, after increasing \(\alpha_0\), there exist constants
\(c,C>0\), independent of \(\alpha,\lambda\), and of the solution, such
that
\begin{equation}\label{eq:As-comparable-geodesic-verified}
 c\alpha\Id_E
 \leq A_s
 \leq C\alpha\Id_E
 \qquad
 \text{for every }s\in[0,1].
\end{equation}
The matrices $A_s$ are the matrices entering the scalar residual
$\mathcal T_1(f_s,g_s)$ along this interpolation.  The intermediate pairs
$(f_s,g_s)$ are not asserted to solve the trace-free equation.  What is needed
below is precisely the uniform positivity in
\eqref{eq:As-comparable-geodesic-verified}; it guarantees that
$\log\det A_s$ and its variation are well defined for every
$s\in[0,1]$.

We first compare the trace-free equations.  Put
\[
 \mathcal R_s^0
 :=
 \sqrt{-1}\Lambda_\omega R^0_{D_{\widehat h_s}}.
\]
Along the path \(\widehat h_s=e^{sK}h_0\), the standard metric variation
formulas give
\[
 \frac{d}{ds}\mathcal R_s^0
 =
 \sqrt{-1}\Lambda_\omega
 \left(
  \bar\partial\partial^{\widehat h_s}K
  +[\theta,[\theta_{\widehat h_s}^*,K]]
 \right).
\]
Pairing this identity with \(K\), integrating by parts over \(M\), and
then integrating in \(s\), we obtain
\begin{equation}\label{eq:HS-geodesic-monotonicity-verified}
 \operatorname{Re}\int_M
 \tr_E\bigl((\mathcal R_1^0-\mathcal R_0^0)K\bigr)\,\omega
 =
 \int_0^1\!\!\int_M
 \left(
  |\partial^{\widehat h_s}K|_{\widehat h_s}^2
  +|[\theta_{\widehat h_s}^*,K]|_{\widehat h_s}^2
 \right)\omega\,ds
 \geq0.
\end{equation}
At the two endpoints, the trace-free equations give
\(
 \mathcal R_1^0=-e^f u,
\)
\(
 \mathcal R_0^0=-u_0.
\)
Therefore
\begin{equation}\label{eq:log-endpoint-sign-verified}
 \operatorname{Re}\int_M
 \tr_E\bigl((e^f u-u_0)K\bigr)\,\omega
 \leq0.
\end{equation}

By \eqref{eq:udot-geodesic-t0-verified},
\[
 u-u_0
 =
 \int_0^1
 \mathcal D(\log)_{g_s}(g_sK)\,ds.
\]
Since \(g_s\) remains in a fixed compact positive cone,
Lemma~\ref{lem:compact-cone-matrix-estimates} and uniform equivalence
of the metrics \(\widehat h_s\) and \(h_0\) give
\begin{align}\label{eq:log-geodesic-coercivity-verified}
 &\operatorname{Re}\int_M
 \tr_E\bigl((u-u_0)K\bigr)\,\omega
 \notag\\
 &\qquad=
 \int_0^1\!\!\int_M
 \operatorname{Re}
 \left\langle
  \mathcal D(\log)_{g_s}(g_sK),K
 \right\rangle_{\widehat h_s}
 \omega\,ds
 \geq
 c\|K\|_{L^2(h_0)}^2.
\end{align}
Using
\(
 e^f u-u_0=(u-u_0)+(e^f-1)u
\)
in \eqref{eq:log-endpoint-sign-verified}, we obtain
\[
 c\|K\|_{L^2(h_0)}^2
 \leq
 \left|
  \int_M
  (e^f-1)\operatorname{Re}\tr_E(uK)\,\omega
 \right|.
\]
Since \(\|u\|_{C^0}\leq C\) and
\[
 |e^f-1|\leq C|f|
\]
for the uniformly small function \(f\), the Cauchy--Schwarz inequality
gives
\[
 c\|K\|_{L^2(h_0)}^2
 \leq
 C\|f\|_{L^2}\|K\|_{L^2(h_0)}.
\]
Hence
\begin{equation}\label{eq:K-controlled-by-f-verified}
 \|K\|_{L^2(h_0)}
 \leq
 C\|f\|_{L^2}.
\end{equation}

We now use the scalar equation.  Define
\[
 \mathcal T_1(f,g)
 :=
 \log\det\left[
  \left(
   \frac{\beta}{r}
   +\Delta_\omega^{\mathbb C}f
   +\alpha
  \right)\Id_E
  -e^f\log g
 \right]
 -\lambda f-\log a_0.
\]
Both endpoints satisfy the scalar equation, so
\[
 \mathcal T_1(f_1,g_1)
 -
 \mathcal T_1(f_0,g_0)
 =
 0.
\]
Differentiating along \((f_s,g_s)\) gives
\begin{align}\label{eq:T1-geodesic-derivative-verified}
 \frac{d}{ds}\mathcal T_1(f_s,g_s)
 ={}&
 \tr_E(A_s^{-1})\Delta_\omega^{\mathbb C}f
 -\tr_E(A_s^{-1}e^{sf}\dot u_s)
 \notag\\
 &-
 f\,\tr_E(A_s^{-1}e^{sf}u_s)
 -\lambda f.
\end{align}
Set
\[
 F_1
 :=
 \int_0^1\tr_E(A_s^{-1})\,ds.
\]
The two-sided bound \eqref{eq:As-comparable-geodesic-verified} implies
that there exist constants \(c,C>0\) such that
\begin{equation}\label{eq:F1-two-sided-verified}
 c\alpha^{-1}
 \leq
 F_1
 \leq
 C\alpha^{-1},
 \qquad
 c\alpha
 \leq
 F_1^{-1}
 \leq
 C\alpha.
\end{equation}

Integrating \eqref{eq:T1-geodesic-derivative-verified} in \(s\),
multiplying the resulting identity by \(-F_1^{-1}f\), and integrating
over \(M\), we obtain
\begin{align}\label{eq:T1-geodesic-paired-verified}
 0={}&
 \int_M|\partial f|^2\,\omega
 +\lambda\int_MF_1^{-1}f^2\,\omega
 \notag\\
 &+
 \int_MF_1^{-1}f
 \int_0^1
 \tr_E(A_s^{-1}e^{sf}\dot u_s)\,ds\,\omega
 +
 \int_MF_1^{-1}f^2
 \int_0^1
 \tr_E(A_s^{-1}e^{sf}u_s)\,ds\,\omega.
\end{align}
Here the first term follows from
\[
 -F_1^{-1}f\cdot F_1\Delta_\omega^{\mathbb C}f
 =
 -f\Delta_\omega^{\mathbb C}f
\]
and integration by parts.

Since \(g_s\) remains in a fixed compact positive cone,
\eqref{eq:udot-geodesic-t0-verified} implies
\[
 |\dot u_s|\leq C|K|.
\]
Moreover, by
\eqref{eq:As-comparable-geodesic-verified} and
\eqref{eq:F1-two-sided-verified},
\[
 \|F_1^{-1}A_s^{-1}\|\leq C.
\]
Together with the uniform bounds for \(f\) and \(u_s\), this shows that
the last two terms in
\eqref{eq:T1-geodesic-paired-verified} are bounded in absolute value by
\[
 C\|f\|_{L^2}\|K\|_{L^2(h_0)}
 +
 C\|f\|_{L^2}^2.
\]
Using \(F_1^{-1}\geq c\alpha\), we therefore obtain
\[
 \|\partial f\|_{L^2}^2
 +c\lambda\alpha\|f\|_{L^2}^2
 \leq
 C\|f\|_{L^2}\|K\|_{L^2(h_0)}
 +C\|f\|_{L^2}^2.
\]
By \eqref{eq:K-controlled-by-f-verified},
\[
 \|\partial f\|_{L^2}^2
 +\bigl(c\lambda\alpha-C\bigr)\|f\|_{L^2}^2
 \leq0.
\]
After choosing \(\alpha_0\) and \(\lambda_0\) so that
\(c\lambda\alpha>C\), we conclude that \(f=0\).
Equation \eqref{eq:K-controlled-by-f-verified} then gives \(K=0\), and
hence \(g=g_0\).
\end{proof}

\begin{lemma}[Nondegeneracy of the initial solution]
\label{lem:t0-nondegenerate}
After increasing the thresholds $\alpha_0$ and $\lambda_0$ in
Lemma~\ref{lem:t0-global-uniqueness} if necessary, for every
$\alpha\geq\alpha_0$ and $\lambda\geq\lambda_0$, the linearization
\[
L:
C^{2,\gamma}(M,\mathbb R)
\oplus
C^{2,\gamma}\bigl(\Herm_0(E,h_0)\bigr)
\longrightarrow
C^{0,\gamma}(M,\mathbb R)
\oplus
C^{0,\gamma}\bigl(\Herm_0(E,h_0)\bigr)
\]
of the
Higgs--Demailly system at $(0,g_0,0)$ is an isomorphism.
\end{lemma}

\begin{proof}
Let $u_0=\log g_0$ and
$A_0=(\beta/r+\alpha)\Id_E-u_0$.  Vary
$f_s=s\varphi$ and set
\[
 g_s=g_0e^{sK},
 \qquad
 \widehat h_s=g_sh_{\rm ref},
 \qquad
 h_s=e^{-f_s}\widehat h_s,
\]
where $K\in C^{2,\gamma}(\Herm_0(E,g_0h_{\rm ref}))$.
Here $\widehat h_s$ is the determinant-one metric variable, while
$h_s$ is the corresponding full metric in
\eqref{eq:metric-decomp-intro-full}.  Scalar conformal changes do not
alter the trace-free Hitchin--Simpson curvature, so the second component
is linearized through $\widehat h_s$.  The two components of the
linearization are
\begin{align}
 L_1(\varphi,K)={}&\tr(A_0^{-1})\Delta_\omega^{\mathbb C}\varphi
 -\tr\bigl(A_0^{-1}\mathcal D(\log)_{g_0}(g_0K)\bigr)\notag\\
 &-\varphi\tr(A_0^{-1}u_0)-\lambda\varphi,
 \label{eq:linear-first-t0-detailed}\\
 L_2(\varphi,K)={}&\sqrt{-1}\Lambda_\omega\left(
 \bar\partial\partial^{h_0}K+[\theta,[\theta_{h_0}^*,K]]\right)\notag\\
 &+\mathcal D(\log)_{g_0}(g_0K)+\varphi u_0.
 \label{eq:linear-second-t0-detailed}
\end{align}

Suppose $L_1=L_2=0$.  Set
$w_0=(\tr A_0^{-1})^{-1}$.  Since
$A_0=\alpha\Id_E+O(1)$, one has
\[
 c\alpha\leq w_0\leq C\alpha,\qquad w_0A_0^{-1}=O(1).
\]
Multiplying \eqref{eq:linear-first-t0-detailed} by $-w_0\varphi$ and integrating gives
the exact identity
\begin{eqnarray}\label{eq:linear-scalar-identity}
 0&=&\|\partial\varphi\|_{L^2}^2
 +\lambda\int_Mw_0\varphi^2\,\omega\notag\\
 &&+\int_Mw_0\varphi\,
 \tr\bigl(A_0^{-1}\mathcal D(\log)_{g_0}(g_0K)\bigr)\,\omega
 +\int_Mw_0\varphi^2\tr(A_0^{-1}u_0)\,\omega.
\end{eqnarray}
The bounds for $w_0A_0^{-1}$ and the fixed tensors $u_0,g_0$ imply
\begin{equation}\label{eq:linear-scalar-energy}
 \|\partial\varphi\|_{L^2}^2
 +\lambda\int_Mw_0\varphi^2\,\omega
 \leq C\|\varphi\|_{L^2}^2
      +C\|\varphi\|_{L^2}\|K\|_{L^2}.
\end{equation}
Pairing \eqref{eq:linear-second-t0-detailed} with $K$ gives the exact
identity
\begin{align}\label{eq:linear-matrix-identity}
 0={}&\|\partial^{h_0}K\|_{L^2}^2
 +\|[\theta_{h_0}^*,K]\|_{L^2}^2\\
 &+\int_M\left\langle
   \mathcal D(\log)_{g_0}(g_0K),K\right\rangle_{h_0}\omega
 +\int_M\varphi\langle u_0,K\rangle_{h_0}\omega.\notag
\end{align}
The logarithmic term is bounded below by $c\|K\|_{L^2}^2$; hence
\begin{equation}\label{eq:linear-matrix-energy}
 \|\partial^{h_0}K\|_{L^2}^2
 +\|[\theta_{h_0}^*,K]\|_{L^2}^2+c\|K\|_{L^2}^2
 \leq C\|\varphi\|_{L^2}\|K\|_{L^2}.
\end{equation}
Combining \eqref{eq:linear-scalar-energy} and
\eqref{eq:linear-matrix-energy}, using $w_0\geq c\alpha$, and applying
Young's inequality gives
\begin{align}\label{eq:linear-coercive-t0-detailed}
 0\geq{}&c\|\partial\varphi\|_{L^2}^2
 +c(\lambda\alpha-C)\|\varphi\|_{L^2}^2
 +\|\partial^{h_0}K\|_{L^2}^2\\
 &+\|[\theta_{h_0}^*,K]\|_{L^2}^2
 +c\|K\|_{L^2}^2.\notag
\end{align}
Thus $\varphi=K=0$.  The principal symbol is the direct sum of the
Laplace symbols on the scalar and trace-free Hermitian components.  Keeping
this principal symbol fixed and homotoping the lower-order terms to zero joins
$L$ through elliptic Fredholm operators to the corresponding diagonal
Laplace-type operator, whose Fredholm index is zero.  Hence $L$ also has index
zero.  Since its kernel is trivial, its cokernel is trivial as well, and $L$
is an isomorphism.
\end{proof}

\section{Uniform a priori estimates}
\label{sec4}

We now derive estimates for an arbitrary smooth admissible solution
$(f,g)$ at any $t\in[0,1]$.  All constants in this section are independent of
$t$ and of the particular solution.  We put
\[
 u:=\log g,
 \qquad
 \ell_1\geq\ell_2\geq\cdots\geq\ell_r
\]
for the eigenvalues of $u$, and write $\ell_{\max}=\ell_1$.

Rewriting the trace-free equation relative to \(h_{\rm ref}\), we have
\begin{equation}\label{eq:tracefree-reference}
 \sqrt{-1}\Lambda_\omega
 \left(
  \bar\partial(g^{-1}\partial^{h_{\rm ref}}g)
  +[\theta,g^{-1}\theta_{h_{\rm ref}}^*g
             -\theta_{h_{\rm ref}}^*]
 \right)
 +e^f u+\Psi=0.
\end{equation}
Taking the real trace pairing with \(u\) and applying
Lemma~\ref{lem:logarithmic-pairing}, we obtain
\begin{equation}\label{eq:logarithmic-identity-detailed}
 e^f|u|_{h_{\rm ref}}^2
 -\frac12\Delta_\omega^{\mathbb C}|u|_{h_{\rm ref}}^2
 +\mathscr Q_\nabla(u)+\mathscr Q_\theta(u)
 =
 -\langle\Psi,u\rangle_{h_{\rm ref}}.
\end{equation}
Since the two quadratic terms are nonnegative,
\begin{equation}\label{eq:logarithmic-inequality-detailed}
 e^f|u|_{h_{\rm ref}}^2
 -\frac12\Delta_\omega^{\mathbb C}|u|_{h_{\rm ref}}^2
 \leq
 |\Psi|_{h_{\rm ref}}|u|_{h_{\rm ref}}.
\end{equation}
Consequently, if \(f\geq-C_0\), then the maximum principle gives
\[
 \|u\|_{C^0}
 \leq
 e^{C_0}\|\Psi\|_{C^0}.
\]

The estimates in this section follow the ordinary-bundle strategy of
Pingali \cite{Pingali2023} and Murakami \cite{Murakami2026}; our main task is to verify that the additional Higgs
commutator has the required sign in the logarithmic and largest-eigenvalue
estimates.

\subsection{The largest eigenvalue and the Higgs sign}

The eigenvalues need not be smooth where their multiplicities change.  Let
$W\subset M$ be the open dense set on which all multiplicities are locally
constant.  On $W$ choose a local $h_{\rm ref}$-unitary frame
$\{s_a\}_{a=1}^r$ diagonalizing $u$ and write
\[
 g=\sum_{a=1}^r e^{\ell_a}s_a\otimes s_a^*.
\]
At a point $p\notin W$ we use the following regularization, keeping
track of the fact that the perturbed endomorphism does not solve the
Higgs--Demailly system.  Choose a smooth local $h_{\rm ref}$-unitary frame
$\{\widetilde s_a\}$ which diagonalizes $u(p)$.  After reordering the frame
inside each eigenspace of $u(p)$, choose a smooth Hermitian endomorphism
\begin{align*}
 P_\varepsilon&=\sum_a b_a\widetilde s_a\otimes\widetilde s_a^*,
 &P_\varepsilon&\leq0,\\
 b_1(p)&=0>b_2(p)>\cdots>b_r(p),
 &\|P_\varepsilon\|_{C^2}&\leq\varepsilon.
\end{align*}
Set
\[
 \widehat u_\varepsilon=u+P_\varepsilon,
 \qquad \widehat g_\varepsilon=e^{\widehat u_\varepsilon}.
\]
For a sufficiently small coordinate ball, the largest eigenvalue
$\widehat\ell_{1,\varepsilon}$ is smooth and simple, satisfies
$\widehat\ell_{1,\varepsilon}\leq\ell_{\max}$ on the ball, and equals
$\ell_{\max}$ at $p$.  Since the matrix maps
\[
 U\longmapsto e^U,
 \qquad
 G\longmapsto G^{-1},
\]
are smooth on the relevant compact subsets, there exists a constant
\(C\), depending on the fixed local solution but independent of
\(\varepsilon\), such that
\begin{equation}\label{eq:eigen-regularization-C2}
 \|\widehat g_\varepsilon-g\|_{C^2}
 +\|\widehat g_\varepsilon^{-1}-g^{-1}\|_{C^2}
 \leq C\varepsilon.
\end{equation}
The eigenvalue identity may be applied to
$\widehat\ell_{1,\varepsilon}$, whereas the Higgs--Demailly equation \eqref{eq:HD-detailed}
will always be applied to the original endomorphism \(g\); compare the proof of \cite[Lemma~2.3]{Murakami2026}.

Let
\[
 \partial^{h_{\rm ref}}s_a=(A^{1,0})^b_a s_b,
 \qquad
 \bar\partial s_a=(A^{0,1})^b_a s_b,
\]
and define
\[
 C_{ab}:=-\sqrt{-1}\Lambda_\omega
 (A^{1,0})^b_a\wedge(A^{0,1})^a_b\geq0.
\]

The Chern part of the largest-eigenvalue computation is the ordinary-bundle
identity of \cite[Lemma~2.3]{Murakami2026}.  We record it in our notation.
\begin{lemma}\label{lem:eigen-laplacian-detailed}
On $W$,
\begin{align}\label{eq:eigen-laplacian-detailed}
 \Delta_\omega^{\mathbb C}\ell_a={}&-
 \left(\sqrt{-1}\Lambda_\omega
 \bar\partial(g^{-1}\partial^{h_{\rm ref}}g)\right)^a_a\\
 &+\sum_b(e^{\ell_a-\ell_b}-1)C_{ab}
 -\sum_b(e^{\ell_b-\ell_a}-1)C_{ba}.\notag
\end{align}
\end{lemma}

\begin{proof}
This is exactly \cite[Lemma~2.3]{Murakami2026}, after expressing the
background Chern connection relative to \(h_{\rm ref}\).  No Higgs term
enters this identity; the new contribution is isolated in the next lemma.
\end{proof}

The Higgs term improves the estimate for the largest eigenvalue.

\begin{lemma}\label{lem:Higgs-largest-sign-detailed}
At a point where $u$ is diagonal, write $\theta=\Theta\,dz$.  Then
\begin{eqnarray}\label{eq:Higgs-largest-sign-detailed}
 &&\left(\sqrt{-1}\Lambda_\omega
 [\theta,g^{-1}\theta_{h_{\rm ref}}^*g-\theta_{h_{\rm ref}}^*]
 \right)^1_1\notag\\
 &=&c_\omega\sum_b\left[
  (e^{\ell_1-\ell_b}-1)|\Theta_{1b}|^2
 -(e^{\ell_b-\ell_1}-1)|\Theta_{b1}|^2\right]\geq0.
\end{eqnarray}
\end{lemma}

\begin{proof}
Set
\[
 S:=g^{-1}\theta_{h_{\rm ref}}^*g-\theta_{h_{\rm ref}}^*.
\]
In the \(h_{\rm ref}\)-unitary frame diagonalizing \(u\), one has
\[
 S_{ab}
 =
 \bigl(e^{\ell_b-\ell_a}-1\bigr)
 \overline{\Theta_{ba}}\,d\bar z.
\]
Hence
\[
 \sqrt{-1}\Lambda_\omega[\theta,S]^1_1
 =
 c_\omega\sum_b
 \left[
  \bigl(e^{\ell_1-\ell_b}-1\bigr)|\Theta_{1b}|^2
  -
  \bigl(e^{\ell_b-\ell_1}-1\bigr)|\Theta_{b1}|^2
 \right],
\]
where
\(
c_\omega=\sqrt{-1}\Lambda_\omega(dz\wedge d\bar z)>0.
\)
Since \(\ell_1\geq\ell_b\), both terms in each summand are
nonnegative.  This proves the assertion.
\end{proof}

\begin{lemma}\label{lem:lmax-subharmonic-detailed}
There is a constant $C$ such that $\Delta_\omega^{\mathbb C}\ell_{\max}\geq-C$ in
the viscosity sense.  Consequently,
\begin{equation}\label{eq:lmax-osc-average-detailed}
 \frac1{\Vol(M)}\int_M(\sup_M\ell_{\max}-\ell_{\max})\,\omega\leq C.
\end{equation}
\end{lemma}

\begin{proof}
On \(W\), apply Lemma~\ref{lem:eigen-laplacian-detailed} with \(a=1\).  Since
\(\ell_1\geq\ell_b\) and \(C_{ab}\geq0\), both sums in the
eigenvalue formula are nonnegative.  Hence
\[
 \Delta_\omega^{\mathbb C}\ell_1
 \geq
 -\left(
  \sqrt{-1}\Lambda_\omega
  \bar\partial(g^{-1}\partial^{h_{\rm ref}}g)
 \right)^1_1.
\]
The trace-free equation gives
\[
 -\sqrt{-1}\Lambda_\omega
 \bar\partial(g^{-1}\partial^{h_{\rm ref}}g)
 =
 \sqrt{-1}\Lambda_\omega
 [\theta,g^{-1}\theta_{h_{\rm ref}}^*g-\theta_{h_{\rm ref}}^*]
 +e^f u+\Psi.
\]
By Lemma~\ref{lem:Higgs-largest-sign-detailed}, the \((1,1)\)-component of the Higgs term is
nonnegative.  Since \(\operatorname{tr}_E u=0\), one has
\(\ell_1\geq0\), and therefore
\[
 \Delta_\omega^{\mathbb C}\ell_1
 \geq e^f\ell_1+\Psi^1_1
 \geq-C
 \qquad\text{on }W.
\]

We now consider a point \(p\in M\setminus W\).  Let
\(\varphi\in C^2\) touch \(\ell_{\max}\) from above at \(p\), and let
\(\widehat u_\varepsilon\), \(\widehat g_\varepsilon\), and
\(\widehat\ell_{1,\varepsilon}\) be the regularization introduced
above.  Since
\[
 \widehat\ell_{1,\varepsilon}\leq\ell_{\max}\leq\varphi,
 \qquad
 \widehat\ell_{1,\varepsilon}(p)
 =
 \ell_{\max}(p)
 =
 \varphi(p),
\]
the function
\(\varphi-\widehat\ell_{1,\varepsilon}\) has a local minimum at \(p\).
Thus
\[
 \Delta_\omega^{\mathbb C}\varphi(p)
 \geq
 \Delta_\omega^{\mathbb C}
 \widehat\ell_{1,\varepsilon}(p).
\]
Applying Lemma~\ref{lem:eigen-laplacian-detailed} to \(\widehat g_\varepsilon\), and again using the
nonnegativity of both sums in the largest-eigenvalue direction, gives
\[
 \Delta_\omega^{\mathbb C}\varphi(p)
 \geq
 -\left(
  \sqrt{-1}\Lambda_\omega
  \bar\partial(
   \widehat g_\varepsilon^{-1}
   \partial^{h_{\rm ref}}\widehat g_\varepsilon
  )
 \right)^1_1(p).
\]
The \(C^2\)-estimate for
\(\widehat g_\varepsilon-g\) and
\(\widehat g_\varepsilon^{-1}-g^{-1}\) implies
\[
 \Delta_\omega^{\mathbb C}\varphi(p)
 \geq
 -\left(
  \sqrt{-1}\Lambda_\omega
  \bar\partial(g^{-1}\partial^{h_{\rm ref}}g)
 \right)^1_1(p)
 -C\varepsilon.
\]
At \(p\), the first vector of the regularized eigenframe belongs to
the \(\ell_{\max}(p)\)-eigenspace of \(u(p)\).  Applying the original
trace-free equation to \(g\) and Lemma~\ref{lem:Higgs-largest-sign-detailed} in this direction, we obtain
\[
 \Delta_\omega^{\mathbb C}\varphi(p)
 \geq
 e^{f(p)}\ell_{\max}(p)-C-C\varepsilon
 \geq-C-C\varepsilon.
\]
Letting \(\varepsilon\downarrow0\) proves
\[
 \Delta_\omega^{\mathbb C}\ell_{\max}\geq-C
\]
in the viscosity sense on \(M\).

Since \(\ell_{\max}\) is continuous, this viscosity inequality is
equivalent to the corresponding distributional inequality for the
linear operator \(\Delta_\omega^{\mathbb C}\).  The standard
Green-function estimate then yields \eqref{eq:lmax-osc-average-detailed}.
\end{proof}

\subsection{Curvature bounds independent of a lower bound for \texorpdfstring{$f$}{f}}

\begin{lemma}[Scalar one-sided bounds]
\label{lem:scalar-one-sided}
Every smooth admissible solution satisfies
\begin{equation}\label{eq:scalar-one-sided}
 \Delta_\omega^{\mathbb C} f\geq-C,
 \qquad f\leq C.
\end{equation}
\end{lemma}

\begin{proof}
Taking the trace of $A_t>0$ and using $\tr u=0$ gives
\[
 \beta+r\Delta_\omega^{\mathbb C} f+r(1-t)\alpha>0,
\]
which proves the lower bound for the Laplacian.  At a maximum point $x_0$
of $f$, one has $\Delta_\omega^{\mathbb C} f(x_0)\leq0$ and
\[
 \tr A_t(x_0)\leq\|\beta\|_{C^0}+r\alpha.
\]
If $a_1,\ldots,a_r$ are the positive eigenvalues of $A_t(x_0)$, then
\[
 e^{\lambda f(x_0)}a_0(x_0)=\prod_a a_a
 \leq\left(\frac1r\sum_a a_a\right)^r\leq C.
\]
The fixed positive lower bound for $a_0$ yields $f\leq C$.
\end{proof}

\begin{lemma}[Weighted maximum-eigenvalue estimate]
\label{lem:raw-max-eigenvalue}
Every smooth admissible solution satisfies
\begin{equation}\label{eq:raw-max-eigenvalue}
 e^f\ell_{\max}\leq C,
 \qquad e^f|u|\leq C.
\end{equation}
\end{lemma}

\begin{proof}
Since $\tr u=0$, $\ell_{\max}\geq0$.  Positivity of $A_t$ in a largest
eigenvalue direction gives
\[
 e^f\ell_{\max}<\frac\beta r+\Delta_\omega^{\mathbb C} f+(1-t)\alpha.
\]
After integration, the Laplacian term vanishes, so
\begin{equation}\label{eq:int-eflmax-detailed}
 \int_M e^f\ell_{\max}\,\omega\leq C.
\end{equation}
Lemma~\ref{lem:lmax-subharmonic-detailed} gives
\begin{equation}\label{eq:lmax-L1-osc-detailed}
 \int_M(\sup_M\ell_{\max}-\ell_{\max})\,\omega\leq C.
\end{equation}
The Green representation formula and
Lemma~\ref{lem:scalar-one-sided} give
\begin{equation}\label{eq:f-average-max-detailed}
 \frac1{\Vol(M)}\int_M(f-\sup_Mf)\,\omega\geq-C.
\end{equation}
Now
\begin{align}
 \int_M e^f\ell_{\max}\,\omega
 ={}&\int_M e^f(\ell_{\max}-\sup\ell_{\max})\,\omega\notag\\
 &+e^{\sup f}\sup\ell_{\max}
   \int_M e^{f-\sup f}\,\omega.
 \label{eq:split-eflmax-detailed}
\end{align}

Since \(f\leq C\), Lemma~\ref{lem:lmax-subharmonic-detailed} gives
\[
 \int_M e^f(\ell_{\max}-\sup_M\ell_{\max})\,\omega
 \geq
 -e^{\sup_Mf}
 \int_M(\sup_M\ell_{\max}-\ell_{\max})\,\omega
 \geq-C.
\]
On the other hand, Jensen's inequality yields
\[
 \int_Me^{f-\sup_Mf}\,\omega\geq c>0.
\]  Combining this with
\eqref{eq:int-eflmax-detailed} proves
$e^{\sup f}\sup\ell_{\max}\leq C$, hence
$e^f\ell_{\max}\leq C$.

Finally, if $\ell_1\geq\cdots\geq\ell_r$ and $\sum_a\ell_a=0$, then
$-\ell_r\leq(r-1)\ell_1$ and $|u|\leq C_r\ell_{\max}$.  This proves the
second estimate.
\end{proof}

\begin{lemma}[Two-sided Laplacian and curvature bounds]
\label{lem:laplacian-curvature-bounds}
Every smooth admissible solution satisfies
\begin{equation}\label{eq:laplacian-curvature-bounds}
 |\Delta_\omega^{\mathbb C} f|+\osc_M f\leq C,
 \qquad
 \left\|\sqrt{-1}\Lambda_\omega R_{D_h}\right\|_{C^0}\leq C.
\end{equation}
\end{lemma}

\begin{proof}
Set
\[
 s:=
 \frac{\beta}{r}
 +\Delta_\omega^{\mathbb C}f
 +(1-t)\alpha.
\]
The eigenvalues of \(A_t\) are
\(
 s-e^f\ell_a>0,
\)
and
\(
1\leq a\leq r.
\)
By Lemma~\ref{lem:raw-max-eigenvalue}, there exists \(C_0>0\) such that
\(
 |e^f\ell_a|\leq C_0
\)
for every \(a\).

If \(s\leq2C_0+1\), then \(s\) is already uniformly bounded above.
If \(s>2C_0+1\), then
\[
 s-e^f\ell_a\geq s-C_0
\]
for every \(a\), and hence
\[
 (s-C_0)^r
 \leq
 \det A_t
 =
 e^{\lambda f}a_0
 \leq C,
\]
where the last inequality follows from the upper bound for \(f\) and
the fact that \(a_0\) is fixed.  Thus \(s\leq C\) in both cases.
Consequently,
\(
 \Delta_\omega^{\mathbb C}f\leq C.
\)
Together with Lemma~\ref{lem:scalar-one-sided}, this gives
\(
 |\Delta_\omega^{\mathbb C}f|\leq C.
\)

The Green representation formula gives $\osc_M f\leq C$.  Finally,
\begin{equation}\label{eq:K-bounded-detailed}
 \sqrt{-1}\Lambda_\omega R_{D_h}
 =\left(\frac\beta r+\Delta_\omega^{\mathbb C} f\right)\Id_E-e^f u,
\end{equation}
which is uniformly bounded by the preceding estimates.
\end{proof}

\begin{proposition}\label{prop:unconditional-estimates-detailed}
There is a constant $C$, independent of $t$ and of the solution, such that
\begin{equation}\label{eq:unconditional-list-detailed}
 \quad f\leq C,
 \quad e^f\ell_{\max}+e^f|u|
 +|\Delta_\omega^{\mathbb C} f|+\osc_M f\leq C.
\end{equation}
The contracted Hitchin--Simpson curvature is uniformly bounded in $C^0$.
\end{proposition}

\begin{proof}
Combine Lemmas~\ref{lem:scalar-one-sided},
~\ref{lem:raw-max-eigenvalue}, and
~\ref{lem:laplacian-curvature-bounds}.
\end{proof}

\subsection{Estimates assuming the missing lower bound}

Once the missing lower bound for \(f\) is available, the remaining estimates
are standard elliptic consequences of the equations.  We keep the argument
short and refer to Appendix~B for the bundle-valued compactness estimate.

\begin{lemma}[Zeroth-order bounds and strict admissibility]
\label{lem:conditional-C0-strict}
Assume $f\geq-C_0$.  Then
\begin{equation}\label{eq:g-C0-two-sided-detailed}
 C^{-1}\Id_E\leq g\leq C\Id_E,
 \qquad c\Id_E\leq A_t\leq C\Id_E.
\end{equation}
\end{lemma}

\begin{proof}
By Corollary~\ref{cor:weighted-logarithmic-estimate},
\(
 \|u\|_{C^0}\leq e^{C_0}\|\Psi\|_{C^0},
\)
which gives the two-sided metric bound.  Proposition
~\ref{prop:unconditional-estimates-detailed} gives \(\tr A_t\leq C\),
while
\(
 \det A_t=e^{\lambda f}a_0\geq c_0>0.
\)
If \(0<a_1\leq\cdots\leq a_r\) are the eigenvalues of \(A_t\), then
\(a_r\leq C\) and \(a_1\geq c_0/C^{r-1}\).
\end{proof}

\begin{proposition}\label{prop:conditional-estimates-detailed}
Assume $f\geq-C_0$.  Then, for every integer $k\geq0$ and
$0<\gamma<1$, there is a constant $C_{k,\gamma}$, independent of $t$, such
that
\begin{equation}\label{eq:all-estimates-detailed}
 \|f\|_{C^{k+2,\gamma}}+\|g\|_{C^{k+2,\gamma}}
 +\|g^{-1}\|_{C^{k+2,\gamma}}\leq C_{k,\gamma}.
\end{equation}
Moreover, $A_t\geq c\Id_E$.
\end{proposition}

\begin{proof}
The strict lower bound for \(A_t\) is Lemma
~\ref{lem:conditional-C0-strict}.  Proposition
~\ref{prop:unconditional-estimates-detailed} gives uniform \(C^0\)-bounds
for \(f\) and \(\Delta_\omega^{\mathbb C}f\).  Hence, for \(p>2\),
standard scalar elliptic estimates and Morrey embedding give
\[
 \|f\|_{W^{2,p}}+\|f\|_{C^{1,\gamma}}\leq C_p
 \qquad
 \left(0<\gamma<1-\frac2p\right).
\]
The trace-free equation can be written as
\[
 \sqrt{-1}\Lambda_\omega
 \left(
  \bar\partial(g^{-1}\partial^{h_{\rm ref}}g)
  +[\theta,g^{-1}\theta_{h_{\rm ref}}^*g-\theta_{h_{\rm ref}}^*]
 \right)
 +e^f\log g=-\Psi.
\]
Thus Lemma~\ref{lem:strictly-perturbed-compactness}, with
\(c=e^f\) and \(F=-\Psi\), applies on the fixed positive cone supplied by
Lemma~\ref{lem:conditional-C0-strict}; it yields
\[
 \|g\|_{C^{2,\gamma}}+\|g^{-1}\|_{C^{2,\gamma}}\leq C.
\]

It remains only to recover the same regularity for \(f\) from the determinant
equation.  Put
\[
 s=\frac{\beta}{r}+\Delta_\omega^{\mathbb C}f+(1-t)\alpha,
 \qquad
 A_t=s\Id_E-e^fu,
 \qquad u=\log g,
\]
and define
\[
 \mathcal F(x,s,y,U)
 =
 \log\det(s\Id_E-e^yU)-\lambda y-\log a_0(x).
\]
Along a solution \(\mathcal F(x,s,f,u)=0\), and
\[
 \partial_s\mathcal F=\tr_E(A_t^{-1}).
\]
The two-sided bound in Lemma~\ref{lem:conditional-C0-strict} gives
\(0<c\leq\partial_s\mathcal F\leq C\).  The finite-dimensional implicit
function theorem therefore resolves the admissible root smoothly as
\(s=\mathcal S(x,f,u)\) on a fixed neighborhood of the range of the
solutions.  Hence
\[
 \Delta_\omega^{\mathbb C}f
 =
 \mathcal S(x,f,u)-\frac{\beta}{r}-(1-t)\alpha,
\]
and the scalar Schauder estimate yields \(\|f\|_{C^{2,\gamma}}\leq C\).

The higher-order estimates follow by alternating these two standard
bootstraps: the trace-free equation together with
Lemma~\ref{lem:strictly-perturbed-compactness} improves \(g\) once the
coefficient \(e^f\) is controlled, while the resolved scalar equation improves
\(f\) once \(u=\log g\) is controlled.  Iteration gives
\eqref{eq:all-estimates-detailed} for every \(k\geq0\).
\end{proof}

\section{The Leray--Schauder reduction}
\label{sec5}

In this section we prove that a uniform lower bound for the scalar
variable implies solvability at \(t=1\). More precisely, we assume
that every smooth admissible solution along the parameter interval
satisfies
\[
 f\geq-C_0,
\]
where \(C_0\) is independent of \(t\) and of the solution.  By
Proposition~\ref{prop:conditional-estimates-detailed}, all such solutions then satisfy uniform estimates of
every order and remain a fixed positive distance from the boundary of
the admissible cone. The fixed-point construction follows the Leray--Schauder reduction in
\cite[Section~4]{Pingali2023}; the second auxiliary equation is replaced
by its Higgs analogue.

Fix \(0<\gamma<1\) and set
\[
\mathbb X:=
 C^{2,\gamma}(M,\mathbb R)
 \times
 C^{2,\gamma}\bigl(\Herm_0(E,h_{\rm ref})\bigr).
\]

For \(R>0\) and \(\delta>0\), let
\[
 \mathcal B:=
 \left\{
 (f,u,t)\in \mathbb X\times[0,1]:
 \|(f,u)\|_{\mathbb X}<R,\
 A(f,u,t)>\delta\Id_E
 \right\},
\]
where
\[
 A(f,u,t):=
 \left(
  \frac{\beta}{r}
  +\Delta_\omega^{\mathbb C}f
  +(1-t)\alpha
 \right)\Id_E-e^f u.
\]
The constants \(R\) and \(\delta\) will be chosen so that all
admissible solutions satisfying the assumed lower bound lie in
\(\mathcal B\).

We construct the fixed-point map by solving separately a scalar
equation and a trace-free Higgs metric equation.

\begin{lemma}[The scalar solution operator]
\label{lem:scalar-solver-detailed}
Fix a bounded subset of triples $(f,u,t)$ such that $u$ is trace-free
Hermitian and, for some $\delta>0$ independent of the triple,
\begin{equation}\label{eq:scalar-margin}
 \left(\frac\beta r+\Delta_\omega^{\mathbb C} f+(1-t)\alpha\right)\Id_E
 -e^f u\geq\delta\Id_E.
\end{equation}
Then the equation
\begin{equation}\label{eq:auxiliary}
 \det\left[
  \left(\frac\beta r+\Delta_\omega^{\mathbb C} U+(1-t)\alpha\right)\Id_E
  -e^f u
 \right]=e^{\lambda U}a_0
\end{equation}
has a unique admissible solution $U\in C^{2,\gamma}(M,\mathbb R)$.  On
bounded subsets satisfying \eqref{eq:scalar-margin},
$U$ is bounded in $C^{4,\gamma}$ and depends smoothly on $(f,u,t)$
in the corresponding Banach spaces.
\end{lemma}

\begin{proof}
This is the scalar part of the Leray--Schauder construction used in
\cite[Section~4]{Pingali2023}; we record the points needed for uniformity in
our notation.  For each \(x\in M\), let
\[
 s_*(x)=e^{f(x)}\ell_{\max}(u(x))
 -\frac{\beta(x)}r-(1-t)\alpha
\]
and, for \(s>s_*(x)\), set
\[
 L_x(s)=
 \det\left[
 \left(\frac{\beta(x)}r+s+(1-t)\alpha\right)\Id_E-e^{f(x)}u(x)
 \right].
\]
Then \(L_x:(s_*(x),\infty)\to(0,\infty)\) is a smooth increasing
diffeomorphism, because
\[
 \frac{d}{ds}\log L_x(s)
 =
 \tr_E\left[
 \left(
 \left(\frac{\beta(x)}r+s+(1-t)\alpha\right)\Id_E-e^{f(x)}u(x)
 \right)^{-1}
 \right]>0.
\]
Denote its inverse by \(R_x\).  Equation \eqref{eq:auxiliary} is equivalent to
\begin{equation}\label{eq:scalar-inverse}
 \Delta_\omega^{\mathbb C}U
 =
 R_x(e^{\lambda U}a_0(x)).
\end{equation}

Connect this equation to \(U=f\) by
\begin{equation}\label{eq:scalar-path}
 \Delta_\omega^{\mathbb C}U
 =
 (1-\tau)(U-f+\Delta_\omega^{\mathbb C}f)
 +\tau R_x(e^{\lambda U}a_0(x)),
 \qquad 0\leq\tau\leq1.
\end{equation}
The linearization is \(\Delta_\omega^{\mathbb C}-c_\tau(x)\) with
\(c_\tau>0\), so the maximum principle and index zero give openness.

The margin assumption \eqref{eq:scalar-margin} implies
\(s_*\leq\Delta_\omega^{\mathbb C}f-\delta\).  On the prescribed bounded
family, the inverse functions \(R_x\) approach \(s_*\) uniformly as their
argument tends to \(0\), and tend uniformly to \(+\infty\) as their argument
does.  Indeed, after diagonalizing \(u\) and writing
\(\ell_1=\ell_{\max}(u)\), there is a uniform constant \(C_B>0\) such that,
for \(0<\varepsilon\leq1\),
\[
 \varepsilon^r
 \leq L_x(s_*(x)+\varepsilon)
 \leq \varepsilon(\varepsilon+C_B)^{r-1}.
\]
The left inequality gives the uniform convergence
\(R_x(\eta)\downarrow s_*(x)\) as \(\eta\downarrow0\), while uniform
boundedness of the coefficients of \(L_x\) gives
\(R_x(\eta)\to+\infty\) uniformly as \(\eta\to+\infty\).  Hence there
are uniform constants \(C_\pm>0\) such that \(f-C_-\) and \(f+C_+\) are
respectively lower and upper barriers for \eqref{eq:scalar-path}.  Thus
\[
 f-C_-\leq U\leq f+C_+.
\]
In particular, \(e^{\lambda U}a_0\) is bounded above and below by positive
constants.  The inverse roots \(\rho(x):=R_x(e^{\lambda U}a_0(x))\) are
uniformly bounded, and the right inequality above gives
\(\rho(x)-s_*(x)\geq\delta_1>0\).  Thus the matrices
\[
 B_\rho:=\left(\frac\beta r+\rho+(1-t)\alpha\right)\Id_E-e^fu
\]
are uniformly positive.  This statement concerns the inverse roots; the
determinant matrix formed with \(\Delta_\omega^{\mathbb C}U\) need not be
positive at intermediate values of \(\tau\).
The implicit function theorem applied to \(L_x(\rho)=e^{\lambda U}a_0\)
shows that the inverse root is a smooth function of the matrix coefficients
and the positive determinant argument, with uniformly bounded derivatives
on this set.  No differentiability of \(s_*\) is needed.
The right-hand side of \eqref{eq:scalar-path} is uniformly bounded in
\(L^\infty\).  Scalar \(W^{2,p}\) estimates and Sobolev embedding give a
uniform \(C^{1,\eta}\)-bound for some \(\eta>\gamma\).  The right-hand
side is then bounded in \(C^{0,\gamma}\), and Schauder estimates give a
uniform \(C^{2,\gamma}\)-bound.  For each fixed triple \((f,u,t)\), take
a subsequence converging in \(C^{2,\gamma'}\), where
\(0<\gamma'<\gamma\), as \(\tau\) tends to a limiting value.  The equation
gives \(C^{2,\gamma}\) regularity of the limit, proving closedness.
At \(\tau=1\),
\eqref{eq:scalar-inverse} gives \(\Delta_\omega^{\mathbb C}U=\rho\),
so the solution is admissible.  Since the coefficients are bounded in
\(C^{2,\gamma}\), the same equation gives the asserted
\(C^{4,\gamma}\)-bound.

Uniqueness follows directly from the maximum principle and the strict
monotonicity of \(L_x\).  Finally, the linearization of
\eqref{eq:scalar-inverse} in \(U\) is again
\(\Delta_\omega^{\mathbb C}-c(x)\) with \(c>0\); the implicit function
theorem therefore gives smooth dependence on \((f,u,t)\).
\end{proof}

\begin{lemma}[The Higgs metric solution operator]
\label{lem:metric-solver-detailed}
Let $f$ range in a bounded subset of $C^{2,\gamma}(M)$.  There is a unique
positive determinant-one endomorphism $V$ satisfying
\begin{equation}\label{eq:aux-V-detailed}
 \sqrt{-1}\Lambda_\omega\left(
 R^0_{h_{\rm ref}}+\bar\partial(V^{-1}\partial^{h_{\rm ref}}V)
 +[\theta,V^{-1}\theta_{h_{\rm ref}}^*V]
 \right)=-e^f\log V.
\end{equation}
Moreover, $V$ is bounded in $C^{4,\gamma}$ on bounded sets of $f$ and
depends smoothly on $f$.
\end{lemma}

\begin{proof}
For \(0\leq s\leq1\), set
\[
 c_s:=(1-s)+se^f
\]
and consider
\begin{equation}\label{eq:metric-path}
 \sqrt{-1}\Lambda_\omega R^0_{D_{V_sh_{\rm ref}}}
 =-c_s\log V_s,
 \qquad
 \det V_s=1.
\end{equation}
At \(s=0\), Proposition~\ref{prop:g0-detailed} gives the solution \(V_0=g_0\).

Using the fixed-space identification from Subsection~\ref{subsec:strictly-perturbed}, the
linearization at a solution in a trace-free
\(h_s=V_sh_{\rm ref}\)-Hermitian direction \(K\) is
\[
 \mathscr L_sK
 =
 \sqrt{-1}\Lambda_\omega
 \bigl(
  \bar\partial\partial^{h_s}K
  +[\theta,[\theta_{h_s}^*,K]]
 \bigr)
 +c_s\mathcal D(\log)_{V_s}(V_sK).
\]
Pairing with \(K\) and integrating gives
\[
\begin{aligned}
 \operatorname{Re}\int_M
 \langle\mathscr L_sK,K\rangle_{h_s}\,\omega
 &=
 \int_M
 \left(
  |\partial^{h_s}K|_{h_s}^2
  +|[\theta_{h_s}^*,K]|_{h_s}^2
 \right)\omega\\
 &\quad+
 \int_M c_s\operatorname{Re}
 \left\langle
  \mathcal D(\log)_{V_s}(V_sK),K
 \right\rangle_{h_s}\omega.
\end{aligned}
\]
Since \(c_s>0\), Lemma~\ref{lem:matrix-monotonicity-app} implies that the kernel is trivial.
The operator is elliptic of Fredholm index zero, hence is an
isomorphism.  The solution set is therefore open.

Since \(f\) ranges in a bounded subset of \(C^{2,\gamma}\), the
functions \(c_s\) are uniformly bounded above and below by positive
constants and are uniformly bounded in \(C^{2,\gamma}\).  Writing
\(u_s=\log V_s\), Lemma~\ref{lem:logarithmic-pairing} gives
\[
 c_s|u_s|^2
 -\frac12\Delta_\omega^{\mathbb C}|u_s|^2
 +\mathscr Q_\nabla(u_s)+\mathscr Q_\theta(u_s)
 =
 -\langle\Psi,u_s\rangle.
\]
The maximum principle yields a uniform \(C^0\)-bound for \(u_s\).
Lemma~\ref{lem:strictly-perturbed-compactness}, applied with coefficient \(c_s\) and forcing term
\(-\Psi\), then gives a uniform \(C^{2,\gamma}\)-bound for \(V_s\).
Since \(c_s\) is uniformly bounded in \(C^{2,\gamma}\), Schauder
bootstrapping in \eqref{eq:metric-path} gives
\[
 \|V_s\|_{C^{4,\gamma}}\leq C.
\]
These estimates prove closedness and hence existence at \(s=1\).

For uniqueness, let \(V_1,V_2\) be two solutions of the terminal equation
\eqref{eq:aux-V-detailed}.  Put \(h_i=V_i h_{\rm ref}\), let
\(H=V_1^{-1}V_2\), and set \(K=\log H\).  Since
\(\det V_1=\det V_2=1\), one has \(\tr_E K=0\).  Define
\[
 V_\tau=V_1e^{\tau K},
 \qquad
 h_\tau=V_\tau h_{\rm ref}.
\]
Pairing the difference of the two equations with \(K\) and integrating
along this path gives
\[
\begin{aligned}
 0
 &=
 \int_0^1\int_M
 \left(
  |\partial^{h_\tau}K|^2
  +|[\theta_{h_\tau}^*,K]|^2
 \right)\omega\,d\tau \\
 &\quad+
 \int_0^1\int_M
 e^f\operatorname{Re}
 \left\langle
  \mathcal D(\log)_{V_\tau}(V_\tau K),K
 \right\rangle_{h_\tau}
 \omega\,d\tau.
\end{aligned}
\]
Every term is nonnegative, and the logarithmic term is strictly
positive unless \(K=0\).  Thus \(V_1=V_2\).

The same linearized coercivity shows that the derivative with respect
to \(V\) is an isomorphism at every solution.  The Banach-space
implicit function theorem, together with uniqueness, therefore gives
smooth dependence of \(V\) on \(f\).
\end{proof}

For \((f,u,t)\in\mathcal B\), let
\(\mathcal U_t(f,u)\) be the solution given by
Lemma~\ref{lem:scalar-solver-detailed}, and let
\(\mathcal V(f)\) be the solution given by
Lemma~\ref{lem:metric-solver-detailed}.  Define
\begin{equation}\label{eq:Phi-detailed}
\Phi(f,u,t)
:=
\Phi_t(f,u)
:=
\bigl(\mathcal U_t(f,u),\log\mathcal V(f)\bigr).
\end{equation}
Lemmas~\ref{lem:scalar-solver-detailed} and
~\ref{lem:metric-solver-detailed} show that
\(
\Phi:\mathcal B\longrightarrow\mathbb X
\)
is \(C^1\) and maps bounded closed subsets of \(\mathcal B\) into
bounded subsets of \(C^{4,\gamma}\).  Since the embedding
\(
C^{4,\gamma}\hookrightarrow C^{2,\gamma}
\)
is compact, \(\Phi\) maps bounded closed subsets of \(\mathcal B\)
into relatively compact subsets of \(\mathbb X\).

\begin{lemma}[Continuation on a parameterized open set]
\label{lem:parameterized-continuation}
Let \(X\) be a Banach space and let
\(\mathcal D\subset X\times[0,1]\) be relatively open.  Suppose that
\[
 \Phi:\mathcal D\longrightarrow X
\]
is continuous and maps bounded closed subsets of \(\mathcal D\) into
relatively compact subsets of \(X\).  Let
\(\mathcal O\subset X\times[0,1]\) be bounded and relatively open, with
\[
 \overline{\mathcal O}\subset\mathcal D,
\]
and assume that
\begin{equation}\label{eq:no-boundary-fixed-point}
 x\neq\Phi(x,t)
 \qquad
 \text{for every }(x,t)\in\partial\mathcal O,
\end{equation}
where the boundary is taken relative to \(X\times[0,1]\).  If
\[
 \mathcal O_t:=\{x\in X:(x,t)\in\mathcal O\},
 \qquad
 \Phi_t(x):=\Phi(x,t),
\]
then
\[
 \deg(\Id-\Phi_t,\mathcal O_t,0)
\]
is well defined and independent of \(t\).  The degree of an empty
slice is understood to be zero.
\end{lemma}

\begin{proof}
Let
\[
 Z=\{(x,t)\in\overline{\mathcal O}:x=\Phi(x,t)\}.
\]
Compactness of \(\Phi\) makes \(Z\) compact, while
\eqref{eq:no-boundary-fixed-point} gives \(Z\subset\mathcal O\).
Thus each slice has no boundary fixed point and its degree is defined.
For a fixed \(t_0\), compactness of \(Z\) and openness of \(\mathcal O\)
provide a bounded open set \(U\subset X\) and an interval \(I\ni t_0\)
such that all fixed points for \(t\in I\) lie in \(U\) and
\(\overline U\times I\subset\mathcal O\).  Excision and the usual
fixed-domain homotopy invariance then give
\[
 \deg(\Id-\Phi_t,\mathcal O_t,0)
 =
 \deg(\Id-\Phi_t,U,0)
 =
 \deg(\Id-\Phi_{t_0},U,0),
 \qquad t\in I.
\]
(The same conclusion is immediate if the fixed-point slice at \(t_0\) is
empty.)  Hence the degree is locally constant, and therefore constant on
\([0,1]\).
\end{proof}

\begin{theorem}\label{thm:degree-reduction-detailed}
Assume that every smooth admissible solution of the Higgs--Demailly system
satisfies a uniform lower bound $f_t\geq-C$.  Then the system has a smooth
admissible solution at $t=1$.
\end{theorem}

\begin{proof}
By Proposition~\ref{prop:conditional-estimates-detailed}, choose \(R\)
strictly larger than the uniform \(\mathbb X\)-norm bound and \(\delta\)
less than half the uniform lower bound for the smallest eigenvalue of
\(A(f,u,t)\).  Thus every admissible solution satisfying the assumed
lower bound lies in \(\mathcal B\) with a uniform margin.  Recall that
\[
 \Phi:\mathcal B\longrightarrow\mathbb X
\]
is the compact \(C^1\) map defined in \eqref{eq:Phi-detailed}.  A fixed point of
\(\Phi_t\) is precisely an admissible \(C^{2,\gamma}\) solution of
\eqref{eq:HD-detailed}.  Since \(\Phi_t\) takes values in
\(C^{4,\gamma}\), every fixed point is \(C^{4,\gamma}\), and standard
elliptic bootstrapping makes it smooth.  The assumed lower bound
therefore applies to every fixed point.

Set
\[
 \mathcal S:=
 \{(x,t)\in\mathcal B:x=\Phi_t(x)\}.
\]
By Proposition~\ref{prop:conditional-estimates-detailed}, the fixed-point set \(\mathcal S\) is compact in
\(\mathbb X\times[0,1]\) and remains a uniform positive distance from
the boundary of \(\mathcal B\).  Hence there exists a bounded
relatively open neighborhood \(\mathcal O\) of \(\mathcal S\) such
that
\(
 \overline{\mathcal O}\subset\mathcal B.
\)

Since \(\mathcal S\) is the full fixed-point set of \(\Phi\) in
\(\mathcal B\),
\[
 x\neq\Phi_t(x)
 \qquad
 \text{for every }(x,t)\in\partial\mathcal O.
\]
Lemma~\ref{lem:parameterized-continuation}, applied with \(\mathcal D=\mathcal B\), therefore gives
\begin{equation}\label{eq:degree-continuation}
 \deg(\Id-\Phi_t,\mathcal O_t,0)
 =
 \deg(\Id-\Phi_0,\mathcal O_0,0).
\end{equation}

It remains to evaluate the degree at \(t=0\).  Write \(x=(f,u)\) and
define the scalar auxiliary residual by
\[
 F_1(U;x,t):=
 \log\det\left[
 \left(
  \frac{\beta}{r}
  +\Delta_\omega^{\mathbb C}U
  +(1-t)\alpha
 \right)\Id_E-e^f u
 \right]
 -\lambda U-\log a_0.
\]
Let \(F_2(w;f)\) denote the fixed-space formulation, obtained by the
conjugation used in Subsection~\ref{subsec:strictly-perturbed}, of
\[
 \sqrt{-1}\Lambda_\omega
 R^0_{D_{e^wh_{\rm ref}}}
 +e^f w=0.
\]
Thus, if
\[
 N_t(x):=
 \bigl(F_1(f;x,t),F_2(u;f)\bigr),
\]
then \(N_t(x)=0\) is the Higgs--Demailly system.

Write \(W(f):=\log V(f)\).  The auxiliary solution operators satisfy
\[
 F_1(U_t(x);x,t)=0,
 \qquad
 F_2(W(f);f)=0.
\]
Differentiating these identities at a fixed point gives
\[
 D_UF_1\,DU_t+D_xF_1=0,
 \qquad
 D_wF_2\,DW+D_fF_2=0.
\]
Consequently,
\begin{equation}\label{eq:linearization-factorization}
 DN_t
 =
 \begin{pmatrix}
  D_UF_1&0\\
  0&D_wF_2
 \end{pmatrix}
 (\Id-D\Phi_t).
\end{equation}
The block-diagonal factor is an isomorphism by Lemmas~\ref{lem:scalar-solver-detailed} and~\ref{lem:metric-solver-detailed}.

By Lemma~\ref{lem:t0-global-uniqueness}, the only fixed point at \(t=0\) is
\[
 x_0=(0,\log g_0).
\]
The metric variation used in Lemma~\ref{lem:t0-nondegenerate} is
\[
 h_s=e^{sK}h_0,
 \qquad
 g_s=g_0e^{sK},
\]
whereas the second coordinate in \(\mathbb X\) is
\(u=\log g\).  These variations are related by the linear map
\[
 \mathcal J_{g_0}(K)
 :=
 \mathcal D(\log)_{g_0}(g_0K).
\]
The map
\[
 \mathcal J_{g_0}:
 C^{2,\gamma}\bigl(\Herm_0(E,h_0)\bigr)
 \longrightarrow
 C^{2,\gamma}\bigl(\Herm_0(E,h_{\rm ref})\bigr)
\]
is an isomorphism.  Under this change of variables and the fixed-space
identification of the target introduced in Subsection~\ref{subsec:strictly-perturbed}, the
linearization \(DN_0(x_0)\) corresponds to the PDE linearization in
Lemma~\ref{lem:t0-nondegenerate}.  Therefore \(DN_0(x_0)\) is an isomorphism.

Evaluating \eqref{eq:linearization-factorization} at \(x_0\), and using
the invertibility of its block-diagonal factor, we conclude that
\(
 \Id-D\Phi_0(x_0)
\)
is an isomorphism.

Since \(x_0\) is the unique fixed point at
\(t=0\), its local Leray--Schauder degree is \(\pm1\), and therefore
\[
 \deg(\Id-\Phi_0,\mathcal O_0,0)=\pm1.
\]
By \eqref{eq:degree-continuation}, the degree at \(t=1\) is nonzero.
Therefore \(\Phi_1\) has a fixed point, which is a smooth admissible
solution of the Higgs--Demailly system.
\end{proof}

\section{The Higgs--Uhlenbeck--Yau quotient construction}
\label{sec6}

It remains to prove the lower bound required in
Theorem~\ref{thm:degree-reduction-detailed}.  We use a quotient-side
normalization, following Murakami's ordinary-bundle argument and Simpson's
Higgs adaptation of the Uhlenbeck--Yau method.

\subsection{Weak Higgs projections and coherent quotients}

We recall the regularity statement that converts the analytic limit into an
algebro-geometric object.  An endomorphism
$\pi\in W^{1,2}(M,\End E)$ is a \emph{weakly holomorphic projection} if
\begin{equation}\label{eq:weak-holomorphic-definition-detailed}
 \pi^{*_{h_{\rm ref}}}=\pi=\pi^2,
 \qquad
 (\Id_E-\pi)\bar\partial_E\pi=0
\end{equation}
in the distributional sense, where the adjoint is taken with respect to $h_{\rm ref}$.  It is a \emph{weak Higgs projection} if, in
addition,
\begin{equation}\label{eq:weak-Higgs-definition-detailed}
 (\Id_E-\pi)\theta\pi=0.
\end{equation}

\begin{theorem}[Uhlenbeck--Yau regularity
\cite{LubkeTeleman1995,UhlenbeckYau1986}]
\label{thm:UY-regularity-detailed}
Let $\pi$ satisfy \eqref{eq:weak-holomorphic-definition-detailed}.  Then
there are a saturated coherent subsheaf $\mathcal F\subset E$ and an analytic subset
$S\subset M$ with $\operatorname{codim}_{\mathbb C}S\geq2$ such that, on
$M\setminus S$, $\pi$ is smooth and is the orthogonal projection onto the
holomorphic subbundle $\mathcal F|_{M\setminus S}$.
\end{theorem}

In the present curve setting, the analytic subset \(S\) is empty.
Hence \(\pi\) is smooth on \(M\) and its image is a holomorphic
subbundle \(\mathcal F\subset E\).  If \eqref{eq:weak-Higgs-definition-detailed}
holds, then for every local section \(v\) of \(\mathcal F\),
\[
 (\Id_E-\pi)\theta v
 =
 (\Id_E-\pi)\theta\pi v
 =0.
\]
Thus
\[
 \theta(\mathcal F)\subset\mathcal F\otimes K_M,
\]
so \(\mathcal F\) is a holomorphic Higgs subbundle and
\(
 \mathcal Q:=E/\mathcal F
\)
is a Higgs quotient bundle.

The orientation of the projection is important.  We shall construct
a nonzero projection \(q\) and set
\[
 \pi:=\Id_E-q.
\]
Then \(\pi\) projects onto \(\mathcal F\), while
\(q\) projects onto \(\mathcal F^{\perp_{h_{\rm ref}}}\), which is
naturally identified with the quotient bundle \(\mathcal Q\).
Consequently, \(q\neq0\) implies \(\mathcal Q\neq0\).  The case
\(\pi=0\) is allowed and corresponds to \(\mathcal Q=E\).

\subsection{Power inequalities and weak Higgs projections}

Let $H$ be a positive $h_{\rm ref}$-Hermitian endomorphism with
$0<H\leq\Id_E$.  For $0<\sigma\leq1$, define $H^\sigma$ by functional
calculus.

\begin{lemma}[Higgs power inequality]\label{lem:Higgs-power-detailed}
For every $0<\sigma\leq1$ one has
\begin{eqnarray}\label{eq:Higgs-power-detailed}
 &&\int_M\tr\left(
 \sqrt{-1}\Lambda_\omega\bar\partial(H^{-1}\partial^{h_{\rm ref}}H)
 H^\sigma\right)\omega\notag\\
 &&+\int_M\tr\left(
 \sqrt{-1}\Lambda_\omega
 [\theta,H^{-1}\theta_{h_{\rm ref}}^*H-\theta_{h_{\rm ref}}^*]
 H^\sigma\right)\omega\notag\\
 &\geq&
 \|H^{-\sigma/2}\partial^{h_{\rm ref}}H^\sigma\|_{L^2}^2
 +\|H^{-\sigma/2}[\theta_{h_{\rm ref}}^*,H^\sigma]\|_{L^2}^2.
\end{eqnarray}
\end{lemma}

\begin{proof}
The first term is the standard Uhlenbeck--Yau matrix power inequality;
see, for example, \cite[Lemma~3.4.4]{LubkeTeleman1995}.  The Higgs
extension appears in \cite[(5.19)--(5.21)]{Jacob2015}; we record the
pointwise computation in our conventions.  Choose a coordinate with
$g_{z\bar z}=1$ at the point,
diagonalize $H$ with eigenvalues $0<\rho_a\leq1$, and write
$\theta=\Theta\,dz$.  Both sides split into sums over ordered pairs \((a,b)\).  In an
\(h_{\rm ref}\)-unitary frame diagonalizing \(H\), the Higgs term on
the left-hand side equals
\[
 \sum_{a,b}
 \left(\frac{\rho_a}{\rho_b}-1\right)
 \left(\rho_a^\sigma-\rho_b^\sigma\right)
 |\Theta_{ab}|^2,
\]
whereas
\[
 \bigl|H^{-\sigma/2}
 [\theta_{h_{\rm ref}}^*,H^\sigma]\bigr|^2
 =
 \sum_{a,b}
 \rho_b^{-\sigma}
 |\rho_a^\sigma-\rho_b^\sigma|^2
 |\Theta_{ab}|^2.
\]
Thus it suffices to prove, for \(x,y>0\),
\[
 \left(\frac{y}{x}-1\right)(y^\sigma-x^\sigma)
 \geq
 x^{-\sigma}|y^\sigma-x^\sigma|^2.
\]
This is precisely the scalar inequality in Lemma~\ref{lem:scalar-powers-app}.  Hence the
Higgs term satisfies the required estimate.  Combining it with the
standard Uhlenbeck--Yau inequality for the Chern term proves
\eqref{eq:Higgs-power-detailed}.
\end{proof}

Let $(f_i,g_i,t_i)$ be a sequence of admissible solutions, put
$h_i:=e^{-f_i}g_i h_{\rm ref}$, and set
\begin{equation}\label{eq:mH-detailed}
 m_i:=\sup_M\ell_{\max}(\log g_i),
 \qquad H_i:=e^{-m_i}g_i.
\end{equation}
Then $0<H_i\leq\Id_E$.

\begin{lemma}[Uniform weighted energy]
\label{lem:weighted-energy-detailed}
For every $0<\sigma\leq1$, the normalized endomorphisms $H_i$ satisfy
\begin{align}\label{eq:weighted-energy-detailed}
 &\|H_i^{-\sigma/2}\partial^{h_{\rm ref}}H_i^\sigma\|_{L^2}^2
 +\|H_i^{-\sigma/2}[\theta_{h_{\rm ref}}^*,H_i^\sigma]\|_{L^2}^2
 \leq C,
\end{align}
where $C$ is independent of both $i$ and $\sigma$.  Moreover, the constant below can be chosen independent of $\sigma$:
\begin{equation}\label{eq:Hsigma-nonzero-detailed}
 \int_M\tr(H_i^\sigma)\,\omega\geq c_0>0.
\end{equation}
\end{lemma}

\begin{proof}
Multiplication of a Hermitian metric by a positive constant changes neither
its Chern connection nor the adjoint of the Higgs field.  Therefore
\[
 R_{D_{H_i h_{\rm ref}}}=R_{D_{g_i h_{\rm ref}}}.
\]
The metric-change identity gives
\begin{align}
 &\sqrt{-1}\Lambda_\omega R_{D_{g_i h_{\rm ref}}}
 -\sqrt{-1}\Lambda_\omega R_{D_{h_{\rm ref}}}\notag\\
 &\qquad=\sqrt{-1}\Lambda_\omega\left(
 \bar\partial(H_i^{-1}\partial^{h_{\rm ref}}H_i)
 +[\theta,H_i^{-1}\theta_{h_{\rm ref}}^*H_i-
                 \theta_{h_{\rm ref}}^*]\right).
 \label{eq:H-metric-change-detailed}
\end{align}
Pair \eqref{eq:H-metric-change-detailed} with $H_i^\sigma$ and integrate.
Lemma~\ref{lem:Higgs-power-detailed} gives the signed inequality
\begin{align*}
 &\|H_i^{-\sigma/2}\partial^{h_{\rm ref}}H_i^\sigma\|_{L^2}^2
 +\|H_i^{-\sigma/2}[\theta_{h_{\rm ref}}^*,H_i^\sigma]\|_{L^2}^2\\
 &\quad\leq\int_M\tr\left(
 (\sqrt{-1}\Lambda_\omega R_{D_{g_i h_{\rm ref}}}
 -\sqrt{-1}\Lambda_\omega R_{D_{h_{\rm ref}}})H_i^\sigma
 \right)\omega.
\end{align*}
The right-hand side is nonnegative by the inequality and is bounded above
by the absolute values of the two curvature integrals.  Proposition~\ref{prop:unconditional-estimates-detailed} gives a uniform $C^0$ bound for
$\sqrt{-1}\Lambda_\omega R_{D_{h_i}}$ and for $\Delta_\omega^{\mathbb C} f_i$.
Since
\[
 \sqrt{-1}\Lambda_\omega R_{D_{g_i h_{\rm ref}}}
 =\sqrt{-1}\Lambda_\omega R_{D_{h_i}}
  -\Delta_\omega^{\mathbb C} f_i\Id_E,
\]
the curvature on the left of \eqref{eq:H-metric-change-detailed} is
uniformly bounded.  Since \(0<H_i^\sigma\leq\Id_E\), these curvature integrals are bounded
uniformly in both \(i\) and \(0<\sigma\leq1\).  This proves
\eqref{eq:weighted-energy-detailed}.

Let $\ell_{\max,i}$ be the largest eigenvalue of $\log g_i$.  Since
$m_i=\sup_M\ell_{\max,i}$,
\[
 \tr(H_i^\sigma)\geq
 e^{-\sigma(m_i-\ell_{\max,i})}.
\]
Lemma~\ref{lem:lmax-subharmonic-detailed} gives
\[
 \frac1{\Vol(M)}\int_M(m_i-\ell_{\max,i})\omega\leq C.
\]
Jensen's inequality therefore yields
\begin{align*}
 \int_M\tr(H_i^\sigma)\,\omega
 &\geq\Vol(M)\exp\left(
 -\frac\sigma{\Vol(M)}
 \int_M(m_i-\ell_{\max,i})\,\omega\right)\\
 &\geq\Vol(M)e^{-C}=:c_0>0.
\end{align*}
\end{proof}

\begin{lemma}[Strong coefficients and weak derivatives]
\label{lem:strong-weak-products}
Let $A_i,A\in L^\infty(M,\End E)$ satisfy
$\sup_i\|A_i\|_{L^\infty}<\infty$ and $A_i\to A$ almost everywhere.
If $X_i\rightharpoonup X$ weakly in $L^2$, then
$A_i X_i\rightharpoonup AX$ weakly in $L^2$.
\end{lemma}

\begin{proof}
For every $Y\in L^2$, dominated convergence gives
$A_i^*Y\to A^*Y$ strongly in $L^2$.  Therefore
\[
 \int_M\langle A_i X_i,Y\rangle
 =\int_M\langle X_i,A_i^*Y\rangle
 \longrightarrow\int_M\langle X,A^*Y\rangle.
\]
\end{proof}

\begin{proposition}[Simpson--Uhlenbeck--Yau cutoff]
\label{prop:Higgs-cutoff-detailed}
After passing to one subsequence, there are a nonzero orthogonal projection
$q\in W^{1,2}(M,\End E)$ and $\pi=\Id_E-q$ such that
\begin{equation}\label{eq:weak-Higgs-projection-detailed}
 \pi^*=\pi=\pi^2,
 \qquad q\bar\partial_E\pi=0,
 \qquad q\theta\pi=0
\end{equation}
in the distributional sense.  The projection $\pi$ defines a saturated
coherent Higgs subsheaf $\mathcal F\subset E$, and $q$ represents the nonzero
Higgs quotient $\mathcal Q=E/\mathcal F$.  On a Riemann surface this quotient
is locally free.
\end{proposition}

\begin{proof}
The argument is the Higgs spectral cutoff of Simpson and Jacob; compare
\cite{Simpson1988} and \cite[(5.17)--(5.24)]{Jacob2015}.  We give the limiting steps because the
holomorphic and Higgs conditions must be obtained for the same projection.

Apply Lemma~\ref{lem:weighted-energy-detailed} with $\sigma=1$.  After
passing to a subsequence,
\begin{equation}\label{eq:H-common-limit}
 H_i\rightharpoonup H_\infty\quad\text{in }W^{1,2},
 \qquad H_i\to H_\infty\quad\text{strongly in }L^2
 \text{ and almost everywhere},
\end{equation}
where $0\leq H_\infty\leq\Id_E$.  For every fixed
$0<\sigma\leq1$, functional calculus and dominated convergence give
\begin{equation}\label{eq:common-power-limit-detailed}
 H_i^\sigma\to H_\infty^\sigma
 \quad\text{strongly in }L^2.
\end{equation}

Since \(0<H_i\leq\Id_E\), one has
\(H_i^{-\sigma/2}\geq\Id_E\), and hence \eqref{eq:weighted-energy-detailed}
gives
\[
 \|\partial^{h_{\rm ref}}H_i^\sigma\|_{L^2}\leq C.
\]
Moreover, \(H_i^\sigma\) is \(h_{\rm ref}\)-Hermitian, so
\[
 \bar\partial H_i^\sigma
 =
 \bigl(\partial^{h_{\rm ref}}H_i^\sigma\bigr)^{*_{h_{\rm ref}}}.
\]
Together with \(0\leq H_i^\sigma\leq\Id_E\), this yields a uniform
\(W^{1,2}\)-bound, so after
a diagonal extraction over any prescribed countable family of exponents,
these convergences are also weak in $W^{1,2}$.  The lower bound in
Lemma~\ref{lem:weighted-energy-detailed} at $\sigma=1$ gives
$H_\infty\neq0$.

Choose $\sigma_j\downarrow0$ and define the support projection
\begin{equation}\label{eq:q-support-projection}
 q=\mathbf1_{(0,\infty)}(H_\infty),
 \qquad \pi=\Id_E-q.
\end{equation}
Then $H_\infty^{\sigma_j}\to q$ almost everywhere and strongly in $L^2$.
After passing to a subsequence of the exponents $\sigma_j$, the uniform
$W^{1,2}$ bound gives weak convergence in $W^{1,2}$; the strong $L^2$ limit
identifies the weak limit with $q$.  Hence $q\in W^{1,2}$, $q^*=q=q^2$, and $q\neq0$.

Fix $j$ and take $0<s\leq\sigma_j/2$.  The scalar inequality
\eqref{eq:scalar-power3-app} gives the operator inequality
\[
 0\leq\Id_E-H_i^s
 \leq\frac{s}{s+\sigma_j/2}H_i^{-\sigma_j/2}.
\]
Since the factors on the right are functions of $H_i$, for every
endomorphism-valued form $X$,
\[
 |(\Id_E-H_i^s)X|^2
 \leq\left(\frac{s}{s+\sigma_j/2}\right)^2
 |H_i^{-\sigma_j/2}X|^2.
\]
Using Lemma~\ref{lem:weighted-energy-detailed} and the fact that the ratio
is at most one, we obtain
\begin{align}\label{eq:cutoff-energy-detailed}
 &\|(\Id_E-H_i^s)\partial^{h_{\rm ref}}H_i^{\sigma_j}\|_{L^2}^2
 +\|(\Id_E-H_i^s)
 [\theta_{h_{\rm ref}}^*,H_i^{\sigma_j}]\|_{L^2}^2\\
 &\qquad\leq C\frac{s}{s+\sigma_j/2}.\notag
\end{align}
Choose a countable sequence $s_{j,k}\downarrow0$ for each $j$, and perform
one diagonal extraction for the countable set
$\{\sigma_j,s_{j,k}:j,k\geq1\}$.  For fixed $j,k$, Lemma~\ref{lem:strong-weak-products} allows passage to the limit $i\to\infty$ in
the derivative term of \eqref{eq:cutoff-energy-detailed}; the Higgs
commutator term passes even strongly because it is zeroth order.  Weak lower
semicontinuity gives
\begin{align}\label{eq:cutoff-limit-fixed-jk}
 &\|(\Id_E-H_\infty^{s_{j,k}})
 \partial^{h_{\rm ref}}H_\infty^{\sigma_j}\|_{L^2}^2\\
 &\quad+\|(\Id_E-H_\infty^{s_{j,k}})
 [\theta_{h_{\rm ref}}^*,H_\infty^{\sigma_j}]\|_{L^2}^2
 \leq C\frac{s_{j,k}}{s_{j,k}+\sigma_j/2}.\notag
\end{align}
Let $k\to\infty$.  Since
$\Id_E-H_\infty^{s_{j,k}}\to\pi$ almost everywhere and is uniformly
bounded, dominated convergence applied to the fixed $L^2$ factors gives
\begin{equation}\label{eq:defects-fixed-j}
 \pi\partial^{h_{\rm ref}}H_\infty^{\sigma_j}=0,
 \qquad
 \pi[\theta_{h_{\rm ref}}^*,H_\infty^{\sigma_j}]=0.
\end{equation}
Finally let $j\to\infty$.  The first identity passes by weak
$W^{1,2}$ convergence and the second by strong $L^2$ convergence, yielding
\begin{equation}\label{eq:two-defects-zero-detailed}
 \pi\partial^{h_{\rm ref}}q=0,
 \qquad \pi\theta_{h_{\rm ref}}^*q=0.
\end{equation}
Taking the adjoint of the first identity in \eqref{eq:two-defects-zero-detailed}
gives
\(
 (\bar\partial_E q)\pi=0.
\)
Hence \(\bar\partial_E q=(\bar\partial_E q)q\).  Differentiating
\(q^2=q\) and multiplying by \(q\) on both sides gives
\(q(\bar\partial_E q)q=0\).  Therefore
\(
 q\bar\partial_E\pi=-q\bar\partial_E q=0.
\)
Taking the adjoint of the second identity in
\eqref{eq:two-defects-zero-detailed} gives \(q\theta\pi=0\).

The first relation is the weak holomorphicity condition.  By
Theorem~\ref{thm:UY-regularity-detailed}, $\pi$ defines a coherent
subsheaf $\mathcal F$.  The Uhlenbeck--Yau regularity theorem makes the
corresponding projection smooth away from an analytic set of complex
codimension at least two; since $\dim_{\mathbb C}M=1$, that set is empty.
Thus $\pi$ is smooth on all of $M$, and the distributional identity
$q\theta\pi=0$ becomes the ordinary pointwise identity.  Consequently
$\mathcal F=\pi E$ is a smooth holomorphic Higgs subbundle.  Since
$q\neq0$, the quotient $\mathcal Q=E/\mathcal F$ is a nonzero Higgs
bundle (equivalently, a locally free Higgs quotient sheaf).
\end{proof}

\subsection{The quotient degree formula}

We first record the Chern--Weil formula for a smooth holomorphic
Higgs subbundle.  It will then be applied to the projection obtained
in Proposition~\ref{prop:Higgs-cutoff-detailed}.

\begin{lemma}\label{lem:Higgs-quotient-degree}
Let
\(
 \pi=\pi^{*_{h_{\rm ref}}}=\pi^2
\)
be the \(h_{\rm ref}\)-orthogonal projection onto a smooth
holomorphic Higgs subbundle \(\mathcal F\subset E\), and put
\(
 q:=\Id_E-\pi.
\)
Identifying the quotient bundle
\(\mathcal Q=E/\mathcal F\) with \(qE\), one has
\begin{equation}\label{eq:Higgs-quotient-degree}
\begin{aligned}
 2\pi\deg\mathcal Q
 &=
 \int_M
 \operatorname{tr}\left(
  q\sqrt{-1}\Lambda_\omega R_{D_{h_{\rm ref}}}
 \right)\omega\\
 &\quad+
 \int_M|\partial^{h_{\rm ref}}\pi|^2\,\omega
 +
 \int_M|[\theta_{h_{\rm ref}}^*,\pi]|^2\,\omega.
\end{aligned}
\end{equation}
For the projection produced in Proposition~\ref{prop:Higgs-cutoff-detailed},
Theorem~\ref{thm:UY-regularity-detailed} and \(\dim_{\mathbb C}M=1\) imply that \(\pi\) is smooth
on all of \(M\); hence \eqref{eq:Higgs-quotient-degree} applies
globally.
\end{lemma}

\begin{proof}
Use the \(h_{\rm ref}\)-orthogonal splitting
\(
 E=\mathcal F\oplus\mathcal Q.
\)
In the block computations below, products of endomorphism-valued forms
include both composition and the exterior product.
Since \(\mathcal F\) is holomorphic,
\(
 q\bar\partial_E\pi=0.
\)
Taking adjoints and differentiating \(\pi^2=\pi\) give
\(
 \partial^{h_{\rm ref}}\pi
 =
 q\partial^{h_{\rm ref}}\pi\,\pi.
\)
Thus
\(
 B:=q\partial^{h_{\rm ref}}\pi
\)
is the ordinary second fundamental form and
\(
 |B|^2=|\partial^{h_{\rm ref}}\pi|^2.
\)

With respect to the above splitting, the Chern connection has block
form
\[
 \nabla^{h_{\rm ref}}
 =
 \begin{pmatrix}
  \nabla^{\mathcal F}&-B^*\\
  B&\nabla^{\mathcal Q}
 \end{pmatrix}.
\]
The quotient curvature formula is
\begin{equation}\label{eq:Chern-quotient-curvature}
 R_{h_{\mathcal Q}}
 =
 qR_{h_{\rm ref}}q+B\wedge B^*.
\end{equation}
After contraction with \(\sqrt{-1}\Lambda_\omega\) and taking the
trace, the second term contributes
\(|\partial^{h_{\rm ref}}\pi|^2\).

The Higgs invariance \(q\theta\pi=0\) gives
\[
 \theta=
 \begin{pmatrix}
  \theta_{\mathcal F}&\varphi\\
  0&\theta_{\mathcal Q}
 \end{pmatrix},
 \qquad
 \theta_{h_{\rm ref}}^*=
 \begin{pmatrix}
  \theta_{\mathcal F}^*&0\\
  \varphi^*&\theta_{\mathcal Q}^*
 \end{pmatrix}.
\]
A direct computation yields
\begin{equation}\label{eq:Higgs-quotient-commutator}
 q[\theta,\theta_{h_{\rm ref}}^*]q
 =
 [\theta_{\mathcal Q},\theta_{\mathcal Q}^*]
 +\varphi^*\wedge\varphi.
\end{equation}
Hence
\[
 [\theta_{\mathcal Q},\theta_{\mathcal Q}^*]
 =
 q[\theta,\theta_{h_{\rm ref}}^*]q-\varphi^*\wedge\varphi.
\]
In particular, writing $\varphi=\Phi\,dz$ in a unitary frame gives
\[
 -\sqrt{-1}\Lambda_\omega(\varphi^*\wedge\varphi)
 =g_{z\bar z}^{-1}\Phi^\dagger\Phi\geq0.
\]
Together with \eqref{eq:Chern-quotient-curvature}, this yields
\[
 R_{D_{h_{\mathcal Q}}}
 =qR_{D_{h_{\rm ref}}}q+B\wedge B^*
   -\varphi^*\wedge\varphi.
\]
Moreover, since the adjoint of \(q\theta\pi=0\) is
\(\pi\theta_{h_{\rm ref}}^*q=0\),
\[
 [\theta_{h_{\rm ref}}^*,\pi]
 =
 q\theta_{h_{\rm ref}}^*\pi
 =
 \varphi^*.
\]
Thus
\[
 |\varphi|^2
 =
 |[\theta_{h_{\rm ref}}^*,\pi]|^2.
\]
Since the trace of the induced Higgs commutator on \(\mathcal Q\)
vanishes, the Chern--Weil formula gives
\[
 2\pi\deg\mathcal Q
 =\int_M\operatorname{tr}\bigl(
 \sqrt{-1}\Lambda_\omega R_{D_{h_{\mathcal Q}}}
 \bigr)\,\omega.
\]
Contracting the combined curvature identity, taking its trace, and
integrating therefore proves \eqref{eq:Higgs-quotient-degree}.
\end{proof}

\begin{lemma}\label{lem:cutoff-degree}
For the sequence and projection obtained in Proposition~\ref{prop:Higgs-cutoff-detailed},
\begin{equation}\label{eq:cutoff-degree}
 2\pi\deg\mathcal Q
 \leq
 \liminf_{j\to\infty}\liminf_{i\to\infty}
 \int_M
 \operatorname{tr}\left(
  \sqrt{-1}\Lambda_\omega
  R_{D_{g_i h_{\rm ref}}}H_i^{\sigma_j}
 \right)\omega.
\end{equation}
\end{lemma}

\begin{proof}
Recall that
\[
 g_i=e^{m_i}H_i,
\]
where \(m_i\) is constant on \(M\).  Multiplication of a Hermitian
metric by a positive constant changes neither its Chern connection
nor the adjoint of the Higgs field.  Consequently,
\[
 R_{D_{g_i h_{\rm ref}}}
 =
 R_{D_{H_i h_{\rm ref}}}.
\]

The metric-change formula gives
\[
\begin{aligned}
 \sqrt{-1}\Lambda_\omega
 \left(
  R_{D_{g_i h_{\rm ref}}}
  -R_{D_{h_{\rm ref}}}
 \right)
 &=
 \sqrt{-1}\Lambda_\omega
 \bar\partial\left(
  H_i^{-1}\partial^{h_{\rm ref}}H_i
 \right)\\
 &\quad+
 \sqrt{-1}\Lambda_\omega
 [\theta,
   H_i^{-1}\theta_{h_{\rm ref}}^*H_i
   -\theta_{h_{\rm ref}}^*].
\end{aligned}
\]
Pairing this identity with \(H_i^{\sigma_j}\), integrating, and
applying Lemma~\ref{lem:Higgs-power-detailed} yield
\begin{align}
 &\int_M
 \operatorname{tr}\left(
  \sqrt{-1}\Lambda_\omega
  R_{D_{h_{\rm ref}}}H_i^{\sigma_j}
 \right)\omega \notag\\
 &\quad+
 \left\|
  H_i^{-\sigma_j/2}
  \partial^{h_{\rm ref}}H_i^{\sigma_j}
 \right\|_{L^2}^2
 +
 \left\|
  H_i^{-\sigma_j/2}
  [\theta_{h_{\rm ref}}^*,H_i^{\sigma_j}]
 \right\|_{L^2}^2 \notag\\
 &\leq
 \int_M
 \operatorname{tr}\left(
  \sqrt{-1}\Lambda_\omega
  R_{D_{g_i h_{\rm ref}}}H_i^{\sigma_j}
 \right)\omega.
 \label{eq:cutoff-curvature-inequality}
\end{align}
Since \(0<H_i\leq\Id_E\),
\[
 H_i^{-\sigma_j/2}\geq\Id_E,
\]
and therefore
\[
 \|\partial^{h_{\rm ref}}H_i^{\sigma_j}\|_{L^2}^2
 \leq
 \left\|
  H_i^{-\sigma_j/2}
  \partial^{h_{\rm ref}}H_i^{\sigma_j}
 \right\|_{L^2}^2,
\]
and
\[
 \|[\theta_{h_{\rm ref}}^*,H_i^{\sigma_j}]\|_{L^2}^2
 \leq
 \left\|
  H_i^{-\sigma_j/2}
  [\theta_{h_{\rm ref}}^*,H_i^{\sigma_j}]
 \right\|_{L^2}^2.
\]

Fix \(j\).  By Proposition~\ref{prop:Higgs-cutoff-detailed},
\[
 H_i^{\sigma_j}
 \longrightarrow H_\infty^{\sigma_j}
 \quad\text{strongly in }L^2
\]
and weakly in \(W^{1,2}\).  Hence the fixed reference-curvature term
converges, while weak lower semicontinuity gives
\begin{align}
 &\int_M
 \operatorname{tr}\left(
  \sqrt{-1}\Lambda_\omega
  R_{D_{h_{\rm ref}}}H_\infty^{\sigma_j}
 \right)\omega
 +
 \|\partial^{h_{\rm ref}}H_\infty^{\sigma_j}\|_{L^2}^2+
 \|[\theta_{h_{\rm ref}}^*,
       H_\infty^{\sigma_j}]\|_{L^2}^2
 \notag\\
 &\leq
 \liminf_{i\to\infty}
 \int_M
 \operatorname{tr}\left(
  \sqrt{-1}\Lambda_\omega
  R_{D_{g_i h_{\rm ref}}}H_i^{\sigma_j}
 \right)\omega.
 \label{eq:fixed-sigma-degree-limit}
\end{align}

Now let \(j\to\infty\).  Proposition~\ref{prop:Higgs-cutoff-detailed} gives
\[
 H_\infty^{\sigma_j}\to q
 \quad\text{strongly in }L^2
 \quad\text{and weakly in }W^{1,2}.
\]
Thus the reference-curvature term converges,
\[
 \partial^{h_{\rm ref}}H_\infty^{\sigma_j}
 \rightharpoonup
 \partial^{h_{\rm ref}}q
 \quad\text{in }L^2,
\]
and, since \(\theta\) is smooth,
\[
 [\theta_{h_{\rm ref}}^*,H_\infty^{\sigma_j}]
 \longrightarrow
 [\theta_{h_{\rm ref}}^*,q]
 \quad\text{strongly in }L^2.
\]
It follows from \eqref{eq:fixed-sigma-degree-limit} that
\begin{align*}
 &\int_M
 \operatorname{tr}\left(
  q\sqrt{-1}\Lambda_\omega R_{D_{h_{\rm ref}}}
 \right)\omega
 +
 \|\partial^{h_{\rm ref}}q\|_{L^2}^2
 +
 \|[\theta_{h_{\rm ref}}^*,q]\|_{L^2}^2\\
 &\leq
 \liminf_{j\to\infty}\liminf_{i\to\infty}
 \int_M
 \operatorname{tr}\left(
  \sqrt{-1}\Lambda_\omega
  R_{D_{g_i h_{\rm ref}}}H_i^{\sigma_j}
 \right)\omega.
\end{align*}
Since \(q=\Id_E-\pi\),
\[
 |\partial^{h_{\rm ref}}q|
 =
 |\partial^{h_{\rm ref}}\pi|,
 \qquad
 |[\theta_{h_{\rm ref}}^*,q]|
 =
 |[\theta_{h_{\rm ref}}^*,\pi]|.
\]
Lemma~\ref{lem:Higgs-quotient-degree} identifies the left-hand side
with \(2\pi\deg\mathcal Q\), proving
\eqref{eq:cutoff-degree}.
\end{proof}

\subsection{Construction of a quotient of nonpositive degree}

\begin{proposition}\label{prop:nonpositive-Higgs-quotient}
Suppose that there are smooth admissible solutions
\((f_i,g_i,t_i)\) and points \(p_i\in M\) such that
\[
 f_i(p_i)\longrightarrow-\infty.
\]
Then \((E,\theta)\) admits a nonzero Higgs quotient bundle
\(\mathcal Q\) satisfying
\[
 \deg\mathcal Q\leq0.
\]
\end{proposition}

\begin{proof}
By Proposition~\ref{prop:unconditional-estimates-detailed},
\(\operatorname{osc}_M f_i\leq C\).  Hence
\begin{equation}\label{eq:uniform-minus-infinity}
 f_i\longrightarrow-\infty
 \qquad\text{uniformly on }M.
\end{equation}
After passing to a subsequence, assume that
\(t_i\to t_\infty\).  Proposition~\ref{prop:Higgs-cutoff-detailed} gives projections
\(q\) and \(\pi=\Id_E-q\), a holomorphic Higgs subbundle
\(\mathcal F=\pi E\), and a nonzero Higgs quotient bundle
\[
 \mathcal Q=E/\mathcal F.
\]

Put \(u_i=\log g_i\), \(h_i=e^{-f_i}g_i h_{\rm ref}\), and
\begin{equation}
 K_i:=
 \sqrt{-1}\Lambda_\omega R_{D_{h_i}}
 =
 \left(
  \frac{\beta}{r}+\Delta_\omega^{\mathbb C}f_i
 \right)\Id_E-e^{f_i}u_i.
\end{equation}
Proposition~\ref{prop:unconditional-estimates-detailed} gives a uniform \(C^0\)-bound for \(K_i\).
Since \(g_i h_{\rm ref}=e^{f_i}h_i\), conformal change gives
\begin{equation}\label{eq:conformal-curvature-splitting}
 \sqrt{-1}\Lambda_\omega R_{D_{g_i h_{\rm ref}}}
 =
 K_i-\Delta_\omega^{\mathbb C}f_i\,\Id_E.
\end{equation}

We first treat the scalar Laplacian term.  After passing to a
subsequence,
\[
 \Delta_\omega^{\mathbb C}f_i
 \stackrel{*}{\rightharpoonup}\varphi
 \qquad\text{in }L^\infty(M).
\]
Since the integral of a Laplacian vanishes,
\begin{equation}
 \int_M\varphi\,\omega=0.
\end{equation}
For each fixed \(j\), the strong \(L^2\)-convergence
\(H_i^{\sigma_j}\to H_\infty^{\sigma_j}\), together with the
uniform \(L^\infty\)-bound for \(\Delta_\omega^{\mathbb C}f_i\),
first allows us to replace \(H_i^{\sigma_j}\) by
\(H_\infty^{\sigma_j}\) in the integral.  The weak-* convergence
then gives
\[
 \lim_{i\to\infty}
 \int_M
 \Delta_\omega^{\mathbb C}f_i\,
 \tr(H_i^{\sigma_j})\,\omega
 =
 \int_M
 \varphi\,\tr(H_\infty^{\sigma_j})\,\omega.
\]
Letting \(j\to\infty\) and using
\(H_\infty^{\sigma_j}\to q\) strongly in \(L^2\), we obtain
\begin{equation}\label{eq:scalar-term-limit}
 \lim_{j\to\infty}\lim_{i\to\infty}
 \int_M
 \Delta_\omega^{\mathbb C}f_i\,
 \tr(H_i^{\sigma_j})\,\omega
 =
 \int_M\varphi\,\tr(q)\,\omega
 =
 \rk(\mathcal Q)\int_M\varphi\,\omega
 =0.
\end{equation}
Here \(q\) is smooth by Theorem~\ref{thm:UY-regularity-detailed}, and hence
\(\tr(q)=\rk(\mathcal Q)\) on the connected surface \(M\).

We next estimate the \(K_i\)-term.  Fix a point outside the null set
on which the relevant almost-everywhere convergences may fail.
Diagonalize \(u_i\) with eigenvalues
\[
 \ell_{1,i}\geq\cdots\geq\ell_{r,i},
\]
and put
\[
 m_i:=\sup_M\ell_{1,i},
 \qquad
 d_{a,i}:=\ell_{a,i}-m_i\leq0.
\]
Then the eigenvalues of \(H_i^{\sigma_j}\) are
\(e^{\sigma_jd_{a,i}}\).

Since \(K_i\) is an affine function of \(u_i\), it is diagonal in
the same frame.  Let \(k_{a,i}\) be its eigenvalues and set
\begin{equation}\label{eq:smallest-eigenvalue}
 a_{a,i}:=k_{a,i}+(1-t_i)\alpha.
\end{equation}
These are the eigenvalues of \(A_i\), and
\[
 a_{1,i}\leq\cdots\leq a_{r,i},
\]
because \(\ell_{1,i}\geq\cdots\geq\ell_{r,i}\).  The determinant
equation and \eqref{eq:uniform-minus-infinity} give
\[
 \prod_{a=1}^r a_{a,i}
 =
 e^{\lambda f_i}a_0
 \longrightarrow0
 \qquad\text{uniformly on }M.
\]
Since all \(a_{a,i}>0\),
\[
 a_{1,i}^r\leq\prod_{a=1}^r a_{a,i},
\]
and therefore
\begin{equation}
 a_{1,i}\longrightarrow0
 \qquad\text{uniformly on }M.
\end{equation}

Fix \(j\) and consider an arbitrary subsequence in \(i\).  Since
there are only finitely many indices \(a\), after passing to one further
subsequence we may assume, for each \(a\), that either
\(d_{a,i}\to-\infty\) or \(d_{a,i}\geq-C_a\).  In the first case,
the uniform bound for \(K_i\) gives
\[
 k_{a,i}e^{\sigma_jd_{a,i}}\longrightarrow0.
\]
In the second case, set \(C:=\max_a C_a\), where the maximum is taken
over the finitely many indices belonging to this alternative.  Since
\(\ell_{a,i}\leq\ell_{1,i}\leq m_i\), one has
\[
 0\leq\ell_{1,i}-\ell_{a,i}\leq C.
\]
Moreover,
\[
 0\leq a_{a,i}-a_{1,i}
 =
 e^{f_i}(\ell_{1,i}-\ell_{a,i})
 \longrightarrow0.
\]
Together with \eqref{eq:smallest-eigenvalue}, this yields
\[
 k_{a,i}
 =
 a_{a,i}-(1-t_i)\alpha
 \longrightarrow
 -(1-t_\infty)\alpha\leq0.
\]

Since \(0<e^{\sigma_jd_{a,i}}\leq1\), every subsequential limit of
\(k_{a,i}e^{\sigma_jd_{a,i}}\) is nonpositive.  The original subsequence
was arbitrary; hence every subsequential limit of the finite sum
$\tr(K_iH_i^{\sigma_j})$ is nonpositive.  Equivalently,
\[
 \limsup_{i\to\infty}
 \tr(K_iH_i^{\sigma_j})
 \leq0
 \qquad\text{for a.e. }x\in M
\]
for every fixed \(j\).  Since
\[
 \bigl|\tr(K_iH_i^{\sigma_j})\bigr|
 \leq r\sup_i\|K_i\|_{C^0},
\]
the reverse Fatou inequality, applied using this uniform integrable bound, gives
\begin{equation}\label{eq:Ki-integral-limit}
 \limsup_{i\to\infty}
 \int_M\tr(K_iH_i^{\sigma_j})\,\omega
 \leq0.
\end{equation}

Combining \eqref{eq:conformal-curvature-splitting},
\eqref{eq:scalar-term-limit}, and \eqref{eq:Ki-integral-limit}, we
obtain
\[
 \limsup_{j\to\infty}\limsup_{i\to\infty}
 \int_M
 \tr\left(
  \sqrt{-1}\Lambda_\omega
  R_{D_{g_i h_{\rm ref}}}H_i^{\sigma_j}
 \right)\omega
 \leq0.
\]
Consequently,
\[
\begin{aligned}
 \liminf_{j\to\infty}\liminf_{i\to\infty}J_{ij}
 &\leq
 \limsup_{j\to\infty}\limsup_{i\to\infty}J_{ij}\\
 &\leq0,
\end{aligned}
\]
where
\[
 J_{ij}:=
 \int_M
 \tr\left(
  \sqrt{-1}\Lambda_\omega
  R_{D_{g_i h_{\rm ref}}}H_i^{\sigma_j}
 \right)\omega.
\]
Lemma~\ref{lem:cutoff-degree} therefore gives
\[
 2\pi\deg\mathcal Q\leq0.
\]
Thus \(\deg\mathcal Q\leq0\).
\end{proof}

\begin{theorem}\label{thm:uniform-lower-bound}
If \((E,\theta)\) is H-ample, then every smooth admissible solution
of the Higgs--Demailly system satisfies
\begin{equation}
 f_t\geq-C,
\end{equation}
where \(C\) is independent of \(t\) and of the solution.
\end{theorem}

\begin{proof}
Otherwise there would be smooth admissible solutions
\((f_i,g_i,t_i)\) and points \(p_i\in M\) such that
\(f_i(p_i)\to-\infty\).  Proposition~\ref{prop:nonpositive-Higgs-quotient}
would then produce a nonzero Higgs quotient bundle
\(\mathcal Q\) with \(\deg\mathcal Q\leq0\), contradicting
Theorem~\ref{thm:Hample-criterion-detailed}.
\end{proof}

\section{Positivity and proof of the main theorem}
\label{sec7}

\begin{lemma}\label{lem:quotient-Griffiths-positive}
Let \((E,\theta,h)\) have Griffiths-positive Hitchin--Simpson
curvature.  Then every nonzero Higgs quotient bundle, equipped with
the quotient metric, also has Griffiths-positive Hitchin--Simpson
curvature.
\end{lemma}

\begin{proof}
Let
\[
 0\longrightarrow\mathcal F
 \longrightarrow E
 \longrightarrow\mathcal Q
 \longrightarrow0
\]
be an exact sequence of Higgs bundles, and identify \(\mathcal Q\)
smoothly with \(\mathcal F^{\perp_h}\).  Let
\[
 B\in\Omega^{1,0}\bigl(\mathrm{Hom}(\mathcal F,\mathcal Q)\bigr)
\]
be the Chern second fundamental form, and let
\[
 \varphi\in
 \Omega^{1,0}\bigl(\mathrm{Hom}(\mathcal Q,\mathcal F)\bigr)
\]
be the off-diagonal component of the Higgs field.

For \(0\neq v\in\mathcal Q_x\), let
\(\widetilde v\in\mathcal F_x^{\perp_h}\) be its orthogonal lift.
For \(0\neq\xi\in T_x^{1,0}M\), the quotient curvature formula gives
\begin{equation}\label{eq:pointwise-Higgs-quotient-curvature}
\begin{aligned}
 \left\langle
  R_{D_{h_{\mathcal Q}}}(\xi,\bar\xi)v,v
 \right\rangle_{h_{\mathcal Q}}
 &=
 \left\langle
  R_{D_h}(\xi,\bar\xi)\widetilde v,
  \widetilde v
 \right\rangle_h\\
 &\quad+
 |B_\xi^*\widetilde v|_h^2
 +
 |\varphi_\xi\widetilde v|_h^2,
\end{aligned}
\end{equation}
where \(B_\xi^*:\mathcal Q_x\to\mathcal F_x\) is the adjoint of
\(B_\xi:\mathcal F_x\to\mathcal Q_x\).
The first term is strictly positive and the remaining terms are
nonnegative.  Hence the quotient curvature is Griffiths positive.
\end{proof}

\begin{proposition}\label{prop:positive-curvature-H-ample}
If \((E,\theta)\) admits a Hermitian metric with
Griffiths-positive Hitchin--Simpson curvature, then
\((E,\theta)\) is H-ample.
\end{proposition}

\begin{proof}
Since the Higgs commutator has zero trace,
\[
 2\pi\deg E
 =
 \int_M
 \tr\left(
  \sqrt{-1}\Lambda_\omega R_{D_h}
 \right)\omega
 >0.
\]
By Lemma~\ref{lem:quotient-Griffiths-positive}, every nonzero Higgs
quotient bundle also has Griffiths-positive Hitchin--Simpson
curvature and hence positive degree.  Theorem~\ref{thm:Hample-criterion-detailed} therefore implies
that \((E,\theta)\) is H-ample.
\end{proof}

\begin{proof}[Proof of Theorem~\ref{thm:main-intro-full}]
\((1)\Rightarrow(2)\).
Let \(h=h_1\) be the metric associated with an admissible solution at
\(t=1\).  Then
\[
 \sqrt{-1}\Lambda_\omega R_{D_h}>0.
\]
Since \(\Lambda^{1,1}T_x^*M\) is one-dimensional, for every
\(0\neq\xi\in T_x^{1,0}M\), put
\(c_\omega(\xi)=-\sqrt{-1}\omega(\xi,\bar\xi)>0\).  Then
\[
 R_{D_h}(\xi,\bar\xi)
 =c_\omega(\xi)\sqrt{-1}\Lambda_\omega R_{D_h}.
\]
Thus \(R_{D_h}\) is Griffiths positive.

\((2)\Rightarrow(3)\).
This is Proposition~\ref{prop:positive-curvature-H-ample}.

\((3)\Rightarrow(1)\).
Theorem~\ref{thm:uniform-lower-bound} verifies the hypothesis of
Theorem~\ref{thm:degree-reduction-detailed}.  Therefore the Higgs--Demailly system has a smooth
admissible solution at \(t=1\).
\end{proof}

\begin{remark}
The maximum-eigenvalue normalization in Section~\ref{sec6} only needs
to produce a nonzero quotient.  In particular, the case
\[
 \pi=0,\qquad \mathcal Q=E
\]
is allowed: the resulting inequality \(\deg E\leq0\) already contradicts
H-ampleness.  There is no need to require the limiting subsheaf to be
nonzero, as in a stability argument.
\end{remark}

\appendix

\section{Local calculations for the logarithmic method}
\label{app}

This appendix collects the matrix differential formulas and compact-cone estimates used in Sections~\ref{sec2} and~\ref{sec3}.

\begin{lemma}[Differentials of exponential and logarithm]
\label{lem:exp-log-differentials-app}
Let $u$ be a Hermitian endomorphism, $X$ an endomorphism-valued one-form,
$G\in\Herm^+(E,h_{\rm ref})$, and $B\in\End E$.  Then
\begin{align}
 e^{-u}\mathcal D(\exp)_u(X)
 &=\int_0^1e^{-su}Xe^{su}\,ds,
 \label{eq:Dexp-integral-app}\\
 \mathcal D(\log)_G(B)
 &=\int_0^\infty(G+\rho\Id_E)^{-1}B(G+\rho\Id_E)^{-1}\,d\rho.
 \label{eq:Dlog-integral-app}
\end{align}
If $u=\operatorname{diag}(\ell_1,\ldots,\ell_r)$ at a point, then
\begin{equation}\label{eq:Dexp-components-app}
 \bigl(e^{-u}\mathcal D(\exp)_u(X)\bigr)_{ab}
 =\chi(\ell_b-\ell_a)X_{ab},
 \qquad \chi(s)=\frac{e^s-1}{s},
\end{equation}
with $\chi(0)=1$.
\end{lemma}

\begin{proof}
\eqref{eq:Dexp-integral-app} follows from Duhamel's formula. Differentiating
\[
 \log G
 =
 \int_0^\infty
 \left(
  \frac{1}{1+\rho}\Id_E-(G+\rho\Id_E)^{-1}
 \right)d\rho.
\]
under the integral sign gives \eqref{eq:Dlog-integral-app}. Finally, if \(u=\mathrm{diag}(\ell_1,\dots,\ell_r)\), evaluating the first
integral entrywise yields \eqref{eq:Dexp-components-app}.
\end{proof}

\begin{lemma}[Uniform estimates on a compact positive cone]
\label{lem:compact-cone-matrix-estimates}
Fix constants \(0<m\leq M\), and let
\[
 \mathcal K_{m,M}
 :=
 \left\{
  G\in\Herm^+(E,h_{\rm ref}):
  m\Id_E\leq G\leq M\Id_E
 \right\}.
\]
Then the following statements hold.

\begin{enumerate}
\item
For every \(G\in\mathcal K_{m,M}\) and every
\(G h_{\rm ref}\)-Hermitian endomorphism \(K\),
\begin{equation}\label{eq:Dlog-compact-cone}
 \operatorname{Re}
 \left\langle
  \mathcal D(\log)_G(GK),K
 \right\rangle_{G h_{\rm ref}}
 \geq
 \frac{m}{M}|K|_{G h_{\rm ref}}^2.
\end{equation}

\item
Let
\[
 \mathscr S_G(\Psi):=G^{-1/2}\Psi G^{1/2},
\]
where \(\Psi\) is a fixed \(h_{\rm ref}\)-Hermitian endomorphism, and set
\[
 \mathcal Z_G(K)
 :=
 \left.\frac{d}{ds}\right|_{s=0}
 \mathscr S_{Ge^{sK}}(\Psi).
\]
There exists a constant
\(C_{\Psi,m,M}>0\) such that
\begin{equation}\label{eq:transport-derivative-compact-cone}
 \left|
  \operatorname{Re}
  \langle\mathcal Z_G(K),K\rangle_{G h_{\rm ref}}
 \right|
 \leq
 C_{\Psi,m,M}|K|_{G h_{\rm ref}}^2
\end{equation}
for every \(G\in\mathcal K_{m,M}\) and every
\(G h_{\rm ref}\)-Hermitian endomorphism \(K\).
\end{enumerate}
\end{lemma}

\begin{proof}
For the first assertion, fix a point and choose an
\(h_{\rm ref}\)-unitary frame in which
\[
 G=\operatorname{diag}(\lambda_1,\ldots,\lambda_r),
 \qquad
 m\leq\lambda_a\leq M.
\]
Set
\[
 \widehat K:=G^{1/2}KG^{-1/2}.
\]
Then \(\widehat K\) is \(h_{\rm ref}\)-Hermitian and
\[
 |\widehat K|_{h_{\rm ref}}=|K|_{G h_{\rm ref}}.
\]
The divided-difference formula gives
\[
 \bigl(\mathcal D(\log)_G(B)\bigr)_{ab}
 =
 q_{ab}B_{ab},
\]
where
\[
 q_{ab}:=
 \begin{cases}
 \dfrac{\log\lambda_a-\log\lambda_b}
       {\lambda_a-\lambda_b},
 &\lambda_a\neq\lambda_b,\\[8pt]
 \dfrac1{\lambda_a},
 &\lambda_a=\lambda_b.
 \end{cases}
\]
Since
\[
 GK=G^{1/2}\widehat K G^{1/2},
\]
we obtain
\[
 \operatorname{Re}
 \left\langle
  \mathcal D(\log)_G(GK),K
 \right\rangle_{G h_{\rm ref}}
 =
 \sum_{a,b}\lambda_a q_{ab}|\widehat K_{ab}|^2.
\]
By the mean value theorem,
\[
 q_{ab}\geq\frac1M.
\]
Since \(\lambda_a\geq m\), it follows that
\[
 \lambda_aq_{ab}\geq\frac mM,
\]
which proves \eqref{eq:Dlog-compact-cone}.

For the second assertion, the maps
\[
 G\longmapsto G^{1/2},
 \qquad
 G\longmapsto G^{-1/2}
\]
and their Fr\'echet derivatives are smooth and uniformly bounded on
\(\mathcal K_{m,M}\).  Differentiating
\(\mathscr S_G(\Psi)=G^{-1/2}\Psi G^{1/2}\) therefore gives
\[
\begin{aligned}
 \mathcal Z_G(K)
 ={}&
 \mathcal D(G\mapsto G^{-1/2})_G[GK]\,
 \Psi G^{1/2}\\
 &+
 G^{-1/2}\Psi\,
 \mathcal D(G\mapsto G^{1/2})_G[GK].
\end{aligned}
\]
Uniform norm equivalence on \(\mathcal K_{m,M}\) then yields
\[
 |\mathcal Z_G(K)|_{G h_{\rm ref}}
 \leq C_{\Psi,m,M}|K|_{G h_{\rm ref}},
\]
and \eqref{eq:transport-derivative-compact-cone} follows from the
Cauchy--Schwarz inequality.
\end{proof}

\begin{lemma}[Detailed Chern and Higgs logarithmic terms]
\label{lem:local-log-calculation-app}
At a point choose a holomorphic frame normal for $h_{\rm ref}$ and then a
constant unitary change of frame diagonalizing $u$.  Write
\(c_\omega:=\sqrt{-1}\Lambda_\omega(dz\wedge d\bar z)>0\).  With
$B=e^{-u}\partial^{h_{\rm ref}}e^u$, one has
\begin{align}
 \tr(Bu)&=\frac12\partial\tr(u^2),
 \label{eq:Bu-trace-app}\\
 \left\langle\sqrt{-1}\Lambda_\omega\bar\partial B,u\right\rangle
 &=-\frac12\Delta_\omega^{\mathbb C}|u|^2
 +c_\omega\sum_{a,b}\chi(\ell_b-\ell_a)
 |(\partial^{h_{\rm ref}}u)_{ab}|^2,
 \label{eq:Chern-log-local-app}\\
 \left\langle\sqrt{-1}\Lambda_\omega
 [\theta,e^{-u}\theta_{h_{\rm ref}}^*e^u-
            \theta_{h_{\rm ref}}^*],u\right\rangle
 &=c_\omega\sum_{a,b}
 (\ell_a-\ell_b)(e^{\ell_a-\ell_b}-1)|\Theta_{ab}|^2.
 \label{eq:Higgs-log-local-app}
\end{align}
\end{lemma}

\begin{proof}
Equation \eqref{eq:Bu-trace-app} follows from
\eqref{eq:Dexp-integral-app}, cyclicity of the trace, and the fact that $u$
commutes with its exponential.  Applying
$\sqrt{-1}\Lambda_\omega\bar\partial$ and differentiating the Hermitian
pairing gives
\[
 \left\langle\sqrt{-1}\Lambda_\omega\bar\partial B,u\right\rangle
 =-\frac12\Delta_\omega^{\mathbb C}|u|^2
 +{\rm Re}\langle B,\partial^{h_{\rm ref}}u\rangle.
\]
Formula \eqref{eq:Dexp-components-app} gives the second term in
\eqref{eq:Chern-log-local-app}.  For the Higgs term write
$\theta=\Theta dz$.  The $(a,b)$ entry of
$e^{-u}\theta_{h_{\rm ref}}^*e^u-\theta_{h_{\rm ref}}^*$ is
$(e^{\ell_b-\ell_a}-1)\overline{\Theta_{ba}}\,d\bar z$.
Expanding the commutator, taking its trace against $u$, and interchanging
$a$ and $b$ in one of the sums gives
\eqref{eq:Higgs-log-local-app}.
\end{proof}

\begin{lemma}[Asymptotically equal positive factors]
\label{lem:asymptotic-product-app}
Let $x_{a,j}>0$, $1\leq a\leq r$, be functions on a compact space.  If
\[
 \prod_{a=1}^r x_{a,j}\to1
 \quad\text{and}\quad
 \max_{a,b}|x_{a,j}-x_{b,j}|\to0
\]
uniformly, then $x_{a,j}\to1$ uniformly for every $a$.
\end{lemma}

\begin{proof}
Put $m_j(x)=\min_a x_{a,j}(x)$ and
$M_j(x)=\max_a x_{a,j}(x)$.  If
$P_j(x)=\prod_a x_{a,j}(x)$, then
\[
 m_j(x)\leq P_j(x)^{1/r}\leq M_j(x),
 \qquad
 0\leq M_j(x)-m_j(x)\leq
 \max_{a,b}|x_{a,j}(x)-x_{b,j}(x)|.
\]
The middle quantity converges uniformly to one and the last quantity
converges uniformly to zero.  Hence both $m_j$ and $M_j$ converge uniformly
to one, and so does every $x_{a,j}$.
\end{proof}

\section{Matrix inequalities and standard analytic inputs}
\label{app:matrix-detailed}

This appendix records the precise finite-dimensional inequalities and the
standard compactness statement used above.  Their inclusion also fixes all
sign conventions in the Higgs terms.

\begin{lemma}[Trace inequalities for exponential and logarithm]
\label{lem:matrix-monotonicity-app}
Let \(A,B\) be Hermitian matrices. Then
\[
 \tr\bigl((e^A-e^B)(A-B)\bigr)\geq0,
\]
with equality if and only if \(A=B\).  If \(G>0\) and \(K\) is
\(G h_{\rm ref}\)-Hermitian, then
\[
 \operatorname{Re}
 \langle \mathcal D(\log)_G(GK),K\rangle_{G h_{\rm ref}}>0
\]
unless \(K=0\).
\end{lemma}

\begin{proof}
Since
\[
 e^A-e^B
 =
 \int_0^1
 \mathcal D(\exp)_{B+s(A-B)}(A-B)\,ds,
\]
the first assertion follows from the strict positivity of the
divided differences of the exponential.  Equality is possible only
when \(A-B=0\).  The second assertion follows from Lemma~\ref{lem:compact-cone-matrix-estimates}, applied
pointwise with
\[
 m=\lambda_{\min}(G),
 \qquad
 M=\lambda_{\max}(G).
\]
\end{proof}

\begin{lemma}[Scalar power inequalities]\label{lem:scalar-powers-app}
Let \(x,y>0\) and \(0<\sigma\leq1\). Then
\begin{equation}\label{eq:scalar-power1-app}
 (x-y)(x^\sigma-y^\sigma)\geq0.
\end{equation}
Moreover,
\begin{equation}\label{eq:scalar-power2-app}
 \left(\frac yx-1\right)(y^\sigma-x^\sigma)
 \geq x^{-\sigma}|y^\sigma-x^\sigma|^2.
\end{equation}
Finally, if \(0<s\leq\sigma/2\) and \(0<x\leq1\), then
\begin{equation}\label{eq:scalar-power3-app}
 0\leq
 \frac{s+\sigma/2}{s}(1-x^s)
 \leq x^{-\sigma/2}.
\end{equation}
\end{lemma}

\begin{proof}
The first inequality follows from the monotonicity of
\(t\mapsto t^\sigma\).

For the second inequality, set \(z=y/x\). Since
\[
 y^\sigma-x^\sigma=x^\sigma(z^\sigma-1),
\]
division by \(x^\sigma>0\) reduces the assertion to
\[
 (z-1)(z^\sigma-1)\geq(z^\sigma-1)^2,
\]
or equivalently,
\[
 (z^\sigma-1)(z-z^\sigma)\geq0.
\]
If \(z\geq1\), then
\[
 z^\sigma-1\geq0,
 \qquad
 z-z^\sigma\geq0,
\]
because \(0<\sigma\leq1\). If \(0<z\leq1\), then both factors are
nonpositive. Hence the product is nonnegative in either case.

For the last inequality, put
\[
 a:=\frac{\sigma}{2},
 \qquad
 x=e^{-t},
 \qquad
 t\geq0.
\]
Since \(0<s\leq a\), the elementary inequality
\(1-e^{-st}\leq st\) gives
\[
 \frac{s+a}{s}(1-x^s)
 =
 \frac{s+a}{s}(1-e^{-st})
 \leq(s+a)t
 \leq2at.
\]
For every \(r\geq0\), one has \(2r\leq e^r\). Applying this with
\(r=at\), we obtain
\[
 \frac{s+a}{s}(1-x^s)
 \leq2at
 \leq e^{at}
 =
 x^{-a}
 =
 x^{-\sigma/2}.
\]
The lower bound is immediate from \(0<x\leq1\).
\end{proof}

\begin{lemma}[Strictly perturbed compactness]
\label{lem:strictly-perturbed-compactness}
Fix \(0<\gamma<1\).  Let
\[
 c_j\in C^{0,\gamma}(M,\mathbb R),
 \qquad
 F_j\in C^{0,\gamma}(M,\End E),
\]
and suppose that, for some constants \(c_-,c_+,C_0>0\),
\[
 0<c_-\leq c_j\leq c_+,
 \qquad
 \|c_j\|_{C^{0,\gamma}}
 +\|F_j\|_{C^{0,\gamma}}
 \leq C_0.
\]
Let \(g_j\) be positive \(C^2\) \(h_{\rm ref}\)-Hermitian endomorphisms satisfying
\(\det g_j=1\) and
\begin{equation}\label{eq:strictly-perturbed-equation}
\begin{aligned}
 &\sqrt{-1}\Lambda_\omega
 \left(
  \bar\partial(g_j^{-1}\partial^{h_{\rm ref}}g_j)
  +[\theta,
    g_j^{-1}\theta_{h_{\rm ref}}^*g_j
    -\theta_{h_{\rm ref}}^*]
 \right)\\
 &\hspace{45mm}
 +c_j\log g_j=F_j.
\end{aligned}
\end{equation}
If
\[
 \|\log g_j\|_{C^0}\leq L,
\]
then
\[
 \|g_j\|_{C^{2,\gamma}}\leq C,
\]
where \(C\) depends only on the fixed background data,
\(c_-,c_+,C_0,L\), and \(\gamma\).  More generally, uniform
\(C^{k,\gamma}\)-bounds for \(c_j\) and \(F_j\) imply uniform
\(C^{k+2,\gamma}\)-bounds for \(g_j\).
\end{lemma}

\begin{proof}
Set
\[
 u_j:=\log g_j,
 \qquad
 B_j:=g_j^{-1}\partial^{h_{\rm ref}}g_j.
\]
The \(C^0\)-bound for \(u_j\) gives the fixed cone estimate
\[
 e^{-L}\Id_E\leq g_j\leq e^L\Id_E.
\tag{B.7}
\]
Taking the real trace pairing of
\eqref{eq:strictly-perturbed-equation} with \(u_j\), applying
Lemma~\ref{lem:logarithmic-pairing}, and integrating over \(M\), we obtain
\[
 \int_M
 \left(
  \mathscr Q_\nabla(u_j)
  +\mathscr Q_\theta(u_j)
  +c_j|u_j|^2
 \right)\omega
 =
 \int_M\langle F_j,u_j\rangle\,\omega.
\]
The right-hand side is uniformly bounded.  By Lemma~\ref{lem:logarithmic-pairing} and
\(\|u_j\|_{C^0}\leq L\),
\[
 \mathscr Q_\nabla(u_j)
 \geq c_L|\partial^{h_{\rm ref}}u_j|^2,
\]
and hence
\[
 \|\partial^{h_{\rm ref}}u_j\|_{L^2}\leq C.
\tag{B.8}
\]
The differential formula for the exponential and the fixed cone
bound imply
\[
 |B_j|\leq C_L|\partial^{h_{\rm ref}}u_j|,
\]
so
\[
 \|B_j\|_{L^2}\leq C.
\tag{B.9}
\]

Equation \eqref{eq:strictly-perturbed-equation} can be rewritten as
\[
\begin{aligned}
 \sqrt{-1}\Lambda_\omega\bar\partial B_j
 &=
 F_j-c_ju_j\\
 &\quad-
 \sqrt{-1}\Lambda_\omega
 [\theta,
   g_j^{-1}\theta_{h_{\rm ref}}^*g_j
   -\theta_{h_{\rm ref}}^*].
\end{aligned}
\tag{B.10}
\]
The right-hand side is uniformly bounded in \(L^\infty\).  Since
\(\dim_{\mathbb C}M=1\), contraction by \(\Lambda_\omega\) is a
pointwise isomorphism on \((1,1)\)-forms, and
\[
 \bar\partial:
 \Omega^{1,0}(\End E)\longrightarrow\Omega^{1,1}(\End E)
\]
is elliptic.  Therefore
\[
 \|B_j\|_{W^{1,p}}
 \leq
 C_p\bigl(
  \|\bar\partial B_j\|_{L^p}
  +\|B_j\|_{L^p}
 \bigr).
\tag{B.11}
\]
Taking first \(p=2\), and then using
\(W^{1,2}\hookrightarrow L^p\) for every finite \(p\), we obtain,
for any sufficiently large \(p>2\),
\[
 \|B_j\|_{W^{1,p}}\leq C_p.
\]
Choose \(p\) so that
\[
 \gamma<1-\frac2p.
\]
Morrey's inequality gives
\[
 \|B_j\|_{C^{0,\eta}}\leq C
\]
for some \(\gamma<\eta<1-2/p\).

Since
\[
 \partial^{h_{\rm ref}}g_j=g_jB_j
\]
and \(g_j\) is \(h_{\rm ref}\)-Hermitian,
\[
 \bar\partial g_j
 =
 \bigl(\partial^{h_{\rm ref}}g_j\bigr)^{*_{h_{\rm ref}}}.
\]
The fixed cone bound and boundedness of \(B_j\) first give a uniform
Lipschitz bound for \(g_j\).  Since \(B_j\) is bounded in \(C^{0,\eta}\),
the product \(g_jB_j\) has the same H\"older bound.  Hence
\[
 \|g_j\|_{C^{1,\eta}}
 +\|g_j^{-1}\|_{C^{1,\eta}}
 +\|u_j\|_{C^{1,\eta}}
 \leq C.
\tag{B.12}
\]
The right-hand side of (B.10) is consequently uniformly bounded in
\(C^{0,\gamma}\).  Schauder estimates for \(\bar\partial\) give
\[
 \|B_j\|_{C^{1,\gamma}}\leq C.
\]
Using \(\partial^{h_{\rm ref}}g_j=g_jB_j\) once more yields
\[
 \|g_j\|_{C^{2,\gamma}}\leq C.
\]

If \(c_j\) and \(F_j\) are uniformly bounded in \(C^{k,\gamma}\),
repeated differentiation of (B.10), followed by the relation
\(\partial^{h_{\rm ref}}g_j=g_jB_j\), gives
\[
 \|g_j\|_{C^{k+2,\gamma}}\leq C_k.
\]
\end{proof}

\section*{Acknowledgements}
The author is grateful to Professor Kefeng Liu for his encouragement and
guidance, Professor Xiaokui Yang for helpful discussions, and Professor
Man-Chun Lee for his warm hospitality during the author's visit to The
Chinese University of Hong Kong, where part of this work was carried out.
The author also thanks Dr. Armando Capasso for pointing out the discrepancy
in the earlier recursive definition of H-ampleness and drawing his attention
to the work of Biswas, Misra, and Ray.

\noindent\textbf{Funding.} This work was supported by the National Natural Science Foundation of China (Grant No.\ 12601080),
the Scientific Research Foundation of
Chongqing University of Technology (Grant No.\ 2026ZDZ012)
and the Youth Project of the Science and Technology Research Program of
Chongqing Education Commission of China
(Grant No.\ KJQN202601108).

\noindent\textbf{AI declaration.}
During the preparation of this manuscript, ChatGPT (OpenAI) was used for
language polishing, structural organization, notation consistency, and
auxiliary checking. The author reviewed the manuscript and takes full
responsibility for all mathematical content.

\noindent\textbf{Conflict of interest.}
The author declares that he has no conflict of interest.


\begin{thebibliography}{99}

\bibitem{Berndtsson2009}
Berndtsson B. \emph{Curvature of vector bundles associated to holomorphic fibrations}.
Ann Math (2), 2009, \textbf{169}: 531--560.

\bibitem{BiswasMisraRay2026}
Biswas I, Misra S, Ray N. \emph{Positivity of Higgs vector bundles}.
arXiv:2605.22402 [math.AG], 2026.

\bibitem{BlochGieseker1971}
Bloch S, Gieseker D. \emph{The positivity of the Chern classes of an ample vector bundle}.
Invent Math, 1971, \textbf{12}: 112--117.

\bibitem{BruzzoCapassoOtero2025}
Bruzzo U, Capasso A, Gra\~na Otero B. \emph{Positivity for Higgs vector bundles: criteria and applications}.
Rev Mat Complut, 2026, \textbf{39}: 597--609.

\bibitem{BruzzoGrana2007Metrics}
Bruzzo U, Gra\~na Otero B. \emph{Metrics on semistable and numerically effective Higgs bundles}.
J Reine Angew Math, 2007, \textbf{612}: 59--79.

\bibitem{BruzzoGrana2007NF}
Bruzzo U, Gra\~na Otero B. \emph{Numerically flat Higgs vector bundles}.
Commun Contemp Math, 2007, \textbf{9}: 437--446.

\bibitem{BruzzoGrana2011}
Bruzzo U, Gra\~na Otero B. \emph{Semistable and numerically effective principal (Higgs) bundles}.
Adv Math, 2011, \textbf{226}: 3655--3676.

\bibitem{BruzzoHernandez2006}
Bruzzo U, Hern\'andez Ruip\'erez D. \emph{Semistability vs. nefness for (Higgs) vector bundles}.
Differ Geom Appl, 2006, \textbf{24}: 403--416.

\bibitem{CampanaFlenner1990}
Campana F, Flenner H. \emph{A characterization of ample vector bundles on a curve}.
Math Ann, 1990, \textbf{287}: 571--575.

\bibitem{CaoSunZhang2026}
Cao H-D, Sun X, Zhang Y. \emph{The Hermitian--Yang--Mills iteration on stable bundles}.
arXiv:2606.20307 [math.DG], 2026.

\bibitem{Corlette1988}
Corlette K. \emph{Flat {$G$}-bundles with canonical metrics}.
J Differential Geom, 1988, \textbf{28}: 361--382.

\bibitem{Demailly2012}
Demailly J-P. \emph{Analytic Methods in Algebraic Geometry}.
Surveys of Modern Mathematics, vol.~1. Somerville, MA: International Press; Beijing: Higher Education Press, 2012.

\bibitem{Demailly2021}
Demailly J-P. \emph{Hermitian--Yang--Mills approach to the conjecture of Griffiths on the positivity of ample vector bundles}.
Sb Math, 2021, \textbf{212}: 305--318.

\bibitem{DemaillyPeternellSchneider1994}
Demailly J-P, Peternell T, Schneider M. \emph{Compact complex manifolds with numerically effective tangent bundles}.
J Algebraic Geom, 1994, \textbf{3}: 295--345.

\bibitem{DemaillySkoda1980}
Demailly J-P, Skoda H. \emph{Relations entre les notions de positivit\'e de P.~A.~Griffiths et de S.~Nakano pour les fibr\'es vectoriels}.
In: Lelong P, Skoda H, eds. \emph{S\'eminaire Pierre Lelong--Henri Skoda (Analyse), Ann\'ees 1978/79}. Lecture Notes in Mathematics, vol.~822. Berlin: Springer, 1980, 304--309.

\bibitem{Donaldson1983}
Donaldson S K. \emph{A new proof of a theorem of Narasimhan and Seshadri}.
J Differential Geom, 1983, \textbf{18}: 269--277.

\bibitem{Donaldson1985}
Donaldson S K. \emph{Anti self-dual Yang--Mills connections over complex algebraic surfaces and stable vector bundles}.
Proc London Math Soc (3), 1985, \textbf{50}: 1--26.

\bibitem{FanWangYangYau2026}
Fan J, Wang M, Yang X, et al. \emph{Existence of Hermitian metrics with prescribed Hermitian--Yang--Mills tensors II}.
arXiv:2604.02679 [math.DG], 2026.

\bibitem{FultonLazarsfeld1983}
Fulton W, Lazarsfeld R. \emph{Positive polynomials for ample vector bundles}.
Ann Math (2), 1983, \textbf{118}: 35--60.

\bibitem{Gieseker1971}
Gieseker D. \emph{P-ample bundles and their Chern classes}.
Nagoya Math J, 1971, \textbf{43}: 91--116.

\bibitem{Griffiths1969}
Griffiths P A. \emph{Hermitian differential geometry, Chern classes, and positive vector bundles}.
In: Spencer D C, Iyanaga S, eds. \emph{Global Analysis: Papers in Honor of K.~Kodaira}. Tokyo: University of Tokyo Press; Princeton, NJ: Princeton University Press, 1969, 185--251.

\bibitem{GriffithsHarris1978}
Griffiths P, Harris J. \emph{Principles of Algebraic Geometry}.
New York: Wiley-Interscience, 1978.

\bibitem{Hartshorne1966}
Hartshorne R. \emph{Ample vector bundles}.
Publ Math Inst Hautes \'Etudes Sci, 1966, \textbf{29}: 63--94.

\bibitem{Hartshorne1971}
Hartshorne R. \emph{Ample vector bundles on curves}.
Nagoya Math J, 1971, \textbf{43}: 73--89.

\bibitem{Hitchin1987}
Hitchin N J. \emph{The self-duality equations on a Riemann surface}.
Proc London Math Soc (3), 1987, \textbf{55}: 59--126.

\bibitem{Jacob2015}
Jacob A. \emph{Stable Higgs bundles and Hermitian--Einstein metrics on non-K\"ahler manifolds}.
In: Feehan P M N, Song J, Weinkove B, et al., eds. \emph{Analysis, Complex Geometry, and Mathematical Physics: In Honor of Duong H.~Phong}. Contemporary Mathematics, vol.~644. Providence, RI: American Mathematical Society, 2015, 117--140.

\bibitem{Kobayashi1987}
Kobayashi S. \emph{Differential Geometry of Complex Vector Bundles}.
Tokyo: Iwanami Shoten; Princeton, NJ: Princeton University Press, 1987.

\bibitem{Lazarsfeld2004}
Lazarsfeld R. \emph{Positivity in Algebraic Geometry II: Positivity for Vector Bundles, and Multiplier Ideals}.
Berlin: Springer, 2004.

\bibitem{LubkeTeleman1995}
\mbox{L\"ubke M}, Teleman A. \emph{The Kobayashi--Hitchin Correspondence}.
River Edge, NJ: World Scientific, 1995.

\bibitem{Mandal2023}
Mandal A. \emph{The Demailly systems with the vortex ansatz}.
Bull Sci Math, 2023, \textbf{187}: 103307.

\bibitem{MourouganeTakayama2007}
Mourougane C, Takayama S. \emph{Hodge metrics and positivity of direct images}.
J Reine Angew Math, 2007, \textbf{606}: 167--178.

\bibitem{MourouganeTakayama2008}
Mourougane C, Takayama S. \emph{Hodge metrics and the curvature of higher direct images}.
Ann Sci \'Ec Norm Sup\'er (4), 2008, \textbf{41}: 905--924.

\bibitem{Murakami2026}
Murakami R. \emph{An analytic proof of Griffiths' conjecture on compact Riemann surfaces}.
Math Ann, 2026, \textbf{394}: 36.

\bibitem{Nakano1955}
Nakano S. \emph{On complex analytic vector bundles}.
J Math Soc Japan, 1955, \textbf{7}: 1--12.

\bibitem{NarasimhanSeshadri1965}
Narasimhan M S, Seshadri C S. \emph{Stable and unitary vector bundles on a compact Riemann surface}.
Ann Math (2), 1965, \textbf{82}: 540--567.

\bibitem{Nitsure1991}
Nitsure N. \emph{Moduli space of semistable pairs on a curve}.
Proc London Math Soc (3), 1991, \textbf{62}: 275--300.

\bibitem{Pingali2021}
Pingali V P. \emph{A note on Demailly's approach towards a conjecture of Griffiths}.
C R Math Acad Sci Paris, 2021, \textbf{359}: 501--503.

\bibitem{Pingali2023}
Pingali V P. \emph{The Demailly system for a direct sum of ample line bundles on Riemann surfaces}.
Calc Var Partial Differential Equations, 2023, \textbf{62}: 172.

\bibitem{Raufi2015}
Raufi H. \emph{Singular Hermitian metrics on holomorphic vector bundles}.
Ark Mat, 2015, \textbf{53}: 359--382.

\bibitem{Simpson1988}
Simpson C T. \emph{Constructing variations of Hodge structure using Yang--Mills theory and applications to uniformization}.
J Amer Math Soc, 1988, \textbf{1}: 867--918.

\bibitem{Simpson1990}
Simpson C T. \emph{Harmonic bundles on noncompact curves}.
J Amer Math Soc, 1990, \textbf{3}: 713--770.

\bibitem{Simpson1992}
Simpson C T. \emph{Higgs bundles and local systems}.
Publ Math Inst Hautes \'Etudes Sci, 1992, \textbf{75}: 5--95.

\bibitem{Simpson1994I}
Simpson C T. \emph{Moduli of representations of the fundamental group of a smooth projective variety I}.
Publ Math Inst Hautes \'Etudes Sci, 1994, \textbf{79}: 47--129.

\bibitem{Simpson1994II}
Simpson C T. \emph{Moduli of representations of the fundamental group of a smooth projective variety II}.
Publ Math Inst Hautes \'Etudes Sci, 1994, \textbf{80}: 5--79.

\bibitem{UhlenbeckYau1986}
Uhlenbeck K, Yau S-T. \emph{On the existence of Hermitian--Yang--Mills connections in stable vector bundles}.
Comm Pure Appl Math, 1986, \textbf{39}: S257--S293.

\bibitem{UhlenbeckYau1989}
Uhlenbeck K, Yau S-T. \emph{A note on our previous paper: On the existence of Hermitian--Yang--Mills connections in stable vector bundles}.
Comm Pure Appl Math, 1989, \textbf{42}: 703--707.

\bibitem{Umemura1973}
Umemura H. \emph{Some results in the theory of vector bundles}.
Nagoya Math J, 1973, \textbf{52}: 97--128.

\bibitem{WangYangYau2026}
Wang M, Yang X, Yau S-T. \emph{Existence of Hermitian metrics with prescribed Hermitian--Yang--Mills tensors I}.
arXiv:2603.10611 [math.DG], 2026.

\bibitem{WangYangYauTwisted2026}
Wang M, Yang X, Yau S-T. \emph{Existence of twisted Hermitian--Einstein metrics on unstable vector bundles}.
arXiv:2606.15102v1 [math.DG], 2026.

\bibitem{XiongYangYau2026}
Xiong Z, Yang X, Yau S-T. \emph{The prescribed Hermitian--Yang--Mills flow II}.
arXiv:2606.21073 [math.DG], 2026.

\end{thebibliography}
\end{document}